\documentclass[11pt]{preprint}
\usepackage{graphicx,psfrag,epsfig}
\usepackage{amssymb,amsmath,amscd,amsthm,mathrsfs}
\usepackage{graphicx,psfrag,epsfig,xcolor}
\usepackage{bm,enumerate}
\usepackage{times}
\usepackage{graphicx}
\usepackage[active]{srcltx}
\usepackage{mathtools}
\usepackage{tikz}
\usepackage{microtype}
\usepackage{mhequ}
\usepackage{comment}

\usepackage[normalem]{ulem}
\usepackage{soul}
\soulregister{\cite}{7}
\soulregister{\ref}{7}
\soulregister{\eqref}{7}
\definecolor{lightyellow}{Hsb}{60,.5,1}
\sethlcolor{lightyellow}
\setuloverlap{0pt}

\usepackage{hyperref}
\makeatletter
\@ifpackagelater{hyperref}{2012/05/28}{
  \definecolor{link1}{Hsb}{240,1,.75}
  \definecolor{link2}{Hsb}{240,1,.5}
  \hypersetup{final,colorlinks,allcolors=link1,citecolor=link2}
}{
  \hypersetup{final}
}
\makeatother

\newtheorem{theorem}{Theorem}[section]
\newtheorem*{theorem*}{Theorem}

\newtheorem{lemma}[theorem]{Lemma}
\newtheorem{proposition}[theorem]{Proposition}
\newtheorem{definition}[theorem]{Definition}

\newtheorem{remark}[theorem]{Remark}
\newtheorem*{remark*}{Remark}

\definecolor{darkred}{rgb}{0.7,0.1,0.1}

\newcommand{\gautam}[1]{\comment[blue]{GI: #1}}
\definecolor{darkgreen}{rgb}{0,.5,0}
\newcommand{\leonid}[1]{\comment[darkgreen]{LK: #1}}
\newcommand{\al}[1]{\comment[orange]{AN: #1}}
\definecolor{teal}{rgb}{0,.5,.5}

\definecolor{red}{rgb}{1,0,0}
\definecolor{green}{rgb}{0,1,0}
\definecolor{blue}{rgb}{0,0,1}
\definecolor{cyan}{rgb}{0,1,1}

\newcommand{\N}{\mathbb{N}}
\newcommand{\R}{\mathbb{R}}

\newcommand{\Z}{\mathbb{Z}}

\def\eps{{\varepsilon}}

\def\cG{\mathcal{G}}
\def\cV{\mathcal{V}}
\def\cL{\mathcal{L}}
\def\cO{\mathcal{O}}
\def\CO{\mathcal{O}}

\def\CC{\mathcal{C}}
\def\CL{\mathcal{L}}
\newcommand{\torus}{\mathcal T}
\newcommand{\fd}{\mathcal T_{0}}
\def\CM{\mathcal{M}}
\def\CX{\mathcal{X}}
\def\CN{\mathcal{N}}

\def\e{\eps}

\def\P{\mathbf{P}}
\def\E{\mathbf{E}}
\def\Pe{\mathbf{P}_{\!\eps}}
\def\Ee{\mathbf{E}_\eps}
\def\Pb{\bar{\mathbf{P}}}
\def\Eb{\bar{\mathbf{E}}}
\DeclareBoldMathCommand{\one}{1}

\def\bbtau{\tau_{\mathrm{ex}}}

\DeclareMathOperator{\Rot}{Rot}

\numberwithin{equation}{section}

\newcommand{\defeq}{\stackrel{\text{\tiny def}}{=}}
\newcommand{\eff}{_\text{eff}}
\renewcommand{\epsilon}{\eps}
\newcommand{\lap}{\Delta}
\newcommand{\grad}{\nabla}

\newcommand{\curlz}{\grad_z \times}

\DeclarePairedDelimiter{\paren}{(}{)}
\DeclarePairedDelimiter{\brak}{[}{]}
\DeclarePairedDelimiter{\set}{\{}{\}}
\DeclarePairedDelimiter{\floor}{\lfloor}{\rfloor}

\DeclarePairedDelimiter{\abs}{\lvert}{\rvert}
\DeclarePairedDelimiter{\norm}{\lVert}{\rVert}
\DeclarePairedDelimiter{\qv}{\langle}{\rangle}
\DeclarePairedDelimiter{\av}{\langle}{\rangle}

\DeclareMathOperator{\var}{Var}
\DeclareMathOperator{\sign}{sign}
\def\TV{\mathrm{TV}}

\renewcommand{\ge}{\geqslant}
\renewcommand{\geq}{\ge}
\renewcommand{\le}{\leqslant}
\renewcommand{\leq}{\le}
\newcommand{\Gammaeps}{\Gamma_{\!\eps}}
\newcommand{\hge}{\widehat{\Gamma}_{\!\eps}}
\newcommand{\bge}{\overline{\Gamma}_{\!\eps}}
\newcommand{\Vep}{\mathcal V_{\!\epsilon}}
\newcommand{\Rde}{\mathcal R_\delta}

\def\DeltaZ{\Delta^{\!\eps}}
\newcommand{\edelta}{e_\delta}
\newcommand{\edeltab}{\bar  e_\delta}

\def\Lbar{{\bar \CL}}
\newcommand{\Teps}{T^\eps}
\newcommand{\bTeps}{\bar T^\eps}
\def\Hep{H_x^\eps}
\def\MHtext#1{\textcolor{darkred}{#1}}
\def\X{\mathscr X}
\def\CP{\mathcal{P}}
\def\CC{\mathcal{C}}
\def\CF{\mathcal{F}}
\def\id{\mathrm{id}}
\def\Gb{{\bar{\mathcal G}}}

\newcommand{\caputo}[1]{\mathcal D^{#1}_t}

\renewcommand{\downarrow}{\to}

\begin{document}

\title{A fractional kinetic process describing the intermediate time behaviour of cellular flows}
\author{%
  Martin Hairer$^1$,
  Gautam Iyer$^2$,
  Leonid Koralov$^3$,
  Alexei Novikov$^4$ and\\
  Zsolt Pajor-Gyulai$^5$}
\institute{
  University of Warwick, \email{M.Hairer@Warwick.ac.uk}
  \and
  Carnegie Mellon University, \email{gautam@math.cmu.edu}
  \and
  University of Maryland, \email{koralov@math.umd.edu}
  \and
  Pennsylvania State University, \email{anovikov@math.psu.edu}
  \and
  New York University, \email{zsolt@cims.nyu.edu}
}
\maketitle
\begingroup%
  \renewcommand\thefootnote{}%
  \footnote{%
    This material is based upon work partially supported by
    the European Research Council (through grant \#615897 ``Critical'' to MH),
    the Philip Leverhulme trust (through a Research Leadership Award to MH),
    the National Science Foundation (through grants
    DMS-1252912 to GI,
    DMS-1515187 to AN,
    DMS-1309084 to LK),
    the Simons Foundation (through grant \#393685 to GI),
    an AMS-Simons Travel Grant to ZPG,
    and the Center for Nonlinear Analysis (through grant NSF OISE-0967140).
  }
  \addtocounter{footnote}{-1}%
\endgroup%

\begin{abstract}
  This paper studies the \emph{intermediate time} behaviour of a small random perturbation of a periodic cellular flow.
  Our main result shows that on time scales shorter than the diffusive time scale, the limiting behaviour of trajectories that start close enough to cell boundaries is a fractional kinetic process: A Brownian motion time changed by the local time of an independent Brownian motion.
  Our proof uses the Freidlin-Wentzell framework, and the key step is to establish an analogous averaging principle on shorter time scales.

  As a consequence of our main theorem, we obtain a homogenization result for the associated advection diffusion equation.
  We show that on intermediate time scales the effective equation is a \emph{fractional time} PDE that arises in modelling anomalous diffusion.
\end{abstract}

\section{Introduction}


The purpose of this paper is to study the \emph{intermediate time} behaviour of tracer particles passively advected by a periodic cellular flow.
Cellular flows arise in various contexts, most notably as a two-dimensional model for heat transport in Bernard convection cells.
Our interest in studying the \emph{intermediate time} behaviour stems from~\cite{Young88} (see also~\cite{YoungJones91}), which proposes a \emph{fractional kinetic} or non-Fickian model governing the behaviour on intermediate time scales.
This is in stark contrast to the well known diffusive behaviour on long time scales, and the deterministic Hamiltonian ODE behaviour on short time scales.

The position of tracer particles diffusing in a cellular flow is governed by the SDE
\begin{equation}\label{eqnSmallnoiseeq1}
d\tilde{X}_t=v(\tilde{X}_t) \, dt+\sqrt{\e} \, dW_t,
\qquad\text{with }\tilde X_0 \sim \mu.
\end{equation}
Here, $\mu$ is a probability measure on $\R^2$ representing the initial distribution,
$\epsilon$ is twice the molecular diffusivity,
$W$ is a standard two-dimensional Brownian motion.
For notational convenience we denote the law of the solution by $\Pe^\mu$, indicating the~$\mu$ and $\epsilon$ dependence on the probability measure instead of on the process~$\tilde X$, which we always take to be the canonical process.
Above, $v$ is the velocity field of a periodic cellular flow.
Namely, there exists a periodic function $H\colon \R^2 \to \R$ (known as the Hamiltonian, or stream function) such that
\begin{equation*}
  v = \nabla^{\perp} H \defeq
      \begin{pmatrix} -\partial_2 H \\ \phantom{-}\partial_1 H \end{pmatrix}.
\end{equation*}
Moreover,
all the critical points of $H$ are non-degenerate, and there is a connected level set of $H$, say $\mathcal L = \{ x \in \R^2 \;:\; H(x) = 0\}$, called the separatrix, which divides the plane into bounded regions (cells) that are each invariant under the (deterministic) flow of the vector field~$v$ (see Figure \ref{fgrCellular_flow1}). For simplicity of notation, we assume that $H$ has no saddle points inside the cells.
An example commonly used in fluid dynamics is $H(x_1,x_2) = \sin( x_1)\sin( x_2)$, as shown in Figure~\ref{fgrCellular_flow2}.
\begin{figure}[htb]
  \begin{minipage}{0.45\textwidth}
    \centering
    \includegraphics[width=\linewidth]{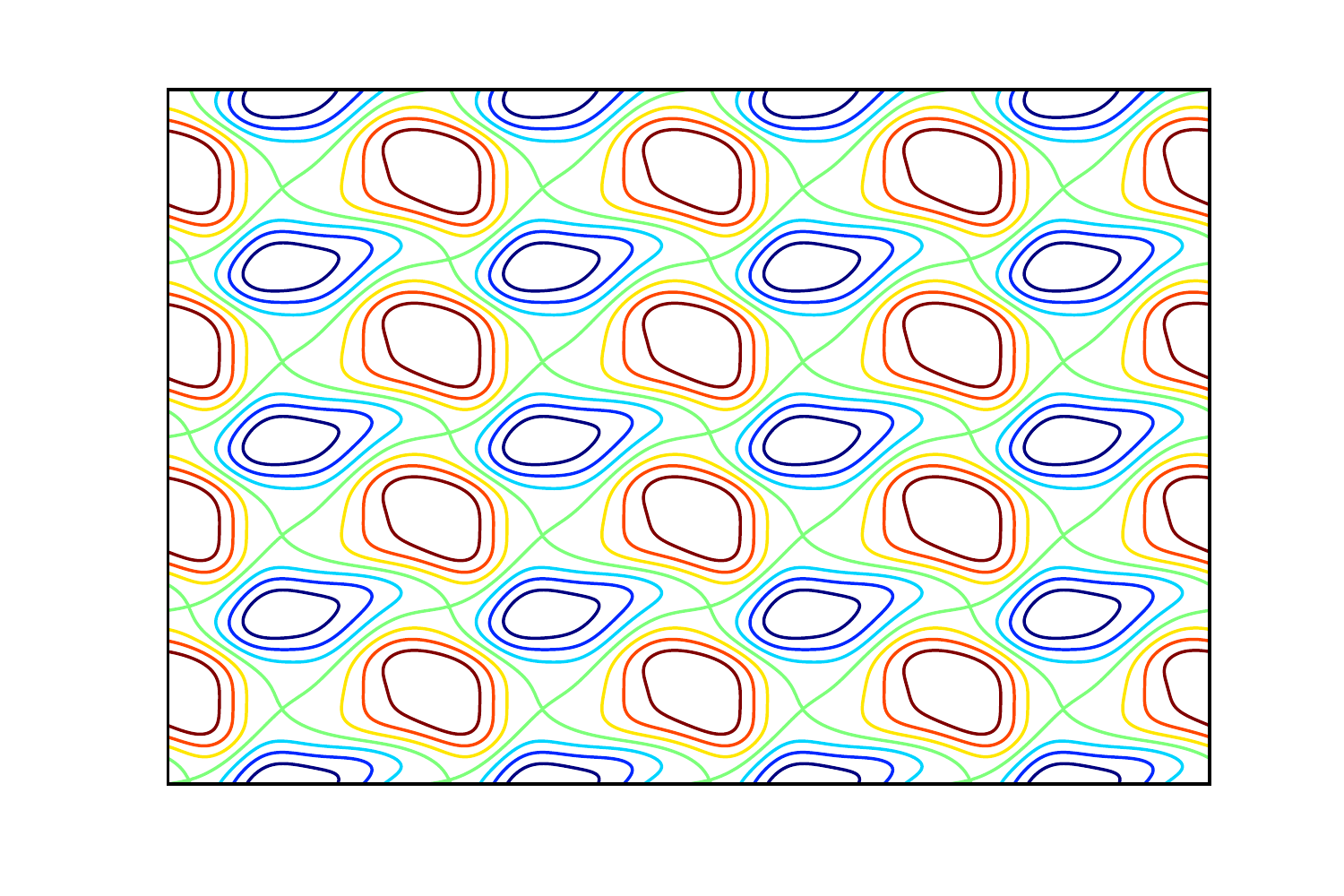}
    \caption{A contour plot of the Hamiltonian in a generic cellular flow.}
    \label{fgrCellular_flow1}
  \end{minipage}\hfill
  \begin{minipage}{0.45\textwidth}
    \centering
    \includegraphics[width=\linewidth]{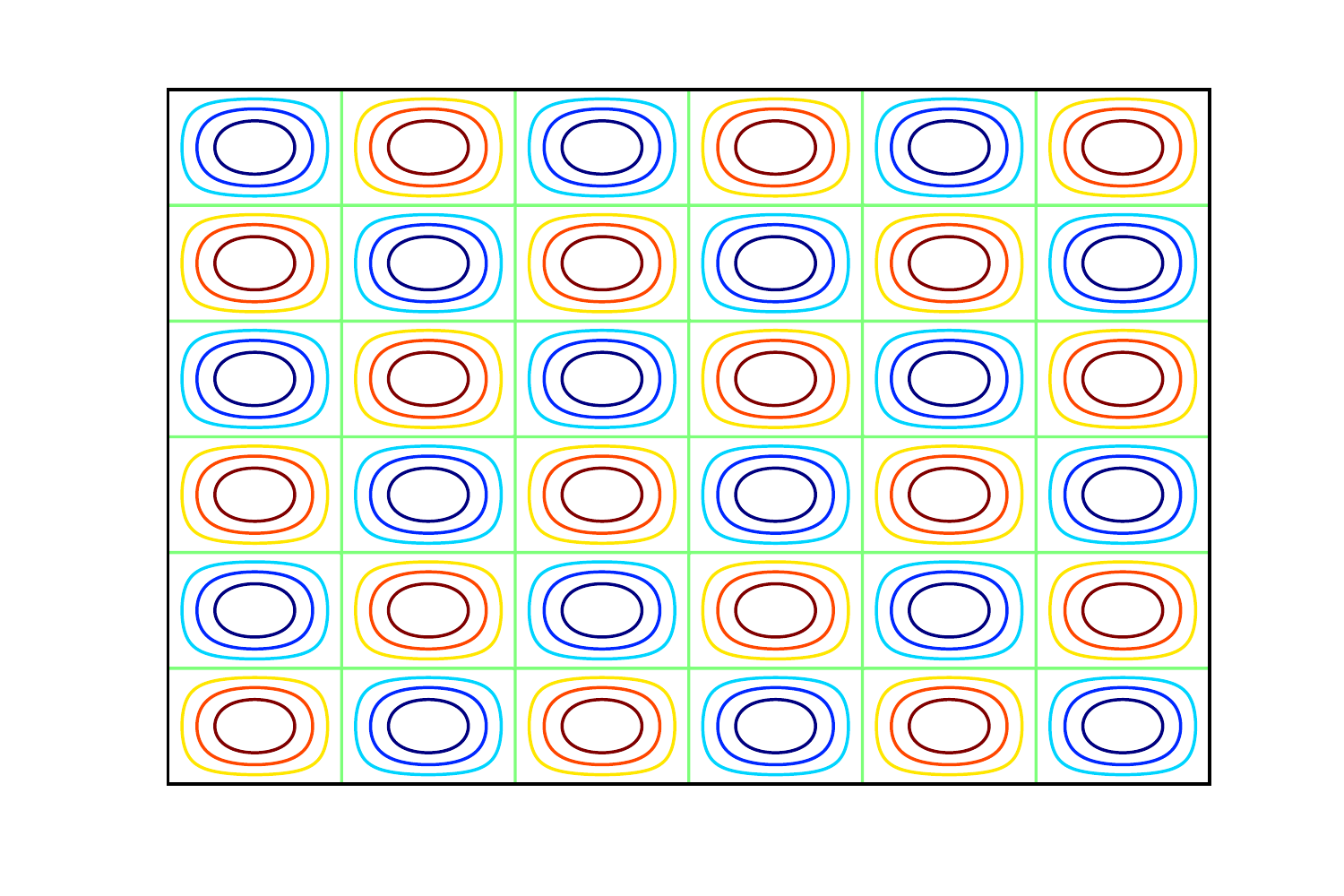}
    \caption{A contour plot of the Hamiltonian $H(x_1,x_2) = \sin( x_1)\sin( x_2)$}
    \label{fgrCellular_flow2}
  \end{minipage}
\end{figure}

The behaviour of $\tilde X$ on both short time scales (i.e.,\ time scales of order $1$) and long time scales (i.e.,\ time scales larger than $1/\epsilon$) is well known.
On short time scales, a large deviations principle~\cite[Chap.~4, Thm~1.1]{FreidlinWentzell12} guarantees that the trajectories of $\tilde X$ deviate from the deterministic trajectories of the flow~$v$ with an exponentially small probability.
On long time scales, standard homogenization results~\cite{Freidlin64} show that $\tilde X$ behaves like a Brownian motion with an enhanced diffusion coefficient.

This paper concerns the effective behaviour of $\tilde X$ on \emph{intermediate time scales}, i.e.,\ time scales much larger than $1$ and much smaller than $1/\epsilon$. 
If the initial condition of $\tilde X$ is chosen in such a way that $H(\tilde X_0) \neq 0$,
then this is again very well understood: at scales of order $1/\epsilon^\alpha$ with $\alpha \in (0,1)$, 
one sees a Brownian motion on the level sets of 
$H$. At scale $1/\epsilon$, one obtains a non-trivial diffusion \cite{FreidlinWentzell93}, as long as the diffusion 
in question does not reach the set $H=0$.
This leaves open the question of the behaviour when the initial condition is chosen close to $H=0$, and this is what
we address in this article.
For such starting points, the limiting behaviour on both time scales above is a \emph{time changed} Brownian motion.
This is a surprising and substantial departure from what is usually expected.
The vast majority of results concerning scaling limits of diffusions obtain a limiting behaviour that is again a diffusion, if not a rescaled Brownian motion.
A time changed Brownian motion was first obtained in~\cite{HairerKoralovPajorGyulai2014} on time scales of order $1/\epsilon$, and here we extend this result to \emph{much shorter} time scales.

Explicitly,
fix $\alpha \in (0, 1)$ and consider the time rescaled process
\begin{equation}\label{eqnZDef}
  Z_t
    = Z(t)
    \defeq
    \tilde X \paren[\Big]{\frac{\alpha\abs{\log\epsilon} t}{\epsilon^{1 - \alpha} }},
\end{equation}
where for notational convenience we sometimes denote time as an argument instead of a subscript.
The process~$Z$ focuses on the behaviour of $\tilde X$ at time scales of order $\abs{\log \epsilon} / \epsilon^{1 - \alpha}$, and the main result of this paper shows that~$Z$ can be spatially rescaled to converge to a \emph{time changed} Brownian motion, provided $\tilde X$ starts on 
(or very close to) cell boundaries. The reason for the extra $\abs{\log \epsilon}$ factor is the logarithmic slow-down of the underlying dynamical system as it approaches hyperbolic saddles, and is revisited in detail later (see also \cite{Kifer81}).
Our main result (Theorem~\ref{thmMainR2}) is a more general version of the following.

\begin{theorem}
There exists a symmetric strictly positive definite matrix $Q$ such that, 
if the initial distribution~$\mu^\epsilon$ is a delta measure at a point that belongs to the 
separatrix~$\mathcal L$, then the laws of $\epsilon^{\frac{1 - \alpha}{4}} Z$
  converge weakly to the law of $W^Q_{L}$.
  Here $W^Q$ is a Brownian motion on $\R^2$ with covariance matrix $Q$, and $L$ is the local time at 
  $0$ of an independent Brownian motion.
\end{theorem}

Note that on the intermediate time scales we consider, if $\tilde X$ starts far away from the separatrix, it will simply make many rotations along 
the flow lines of $v$ without escaping from the cell where it starts.
Thus the assumption that $\tilde X$ starts on (or very close to) the separatrix is necessary in order to observe a non-trivial limiting behaviour. 


As a direct consequence, we also obtain an intermediate time homogenization result for the advection diffusion equation.
Let $\tilde \theta^\epsilon$ satisfy the PDE
\begin{equation}\label{eqnTildeTheta}
  \partial_t \tilde \theta^\epsilon = v \cdot \grad \tilde \theta^\epsilon + \frac{\epsilon}{2} \lap \tilde \theta^\epsilon
  \qquad\text{on } \R^2 \times (0, \infty),
\end{equation}
with initial data $\tilde \theta^\epsilon( x, 0 ) = \tilde \theta^\epsilon_0(x)$.
Standard homogenization results~\cite{PavliotisStuart08,FannjiangPapanicolaou94,Fannjiang02} show
that on time scales longer than $\CO(1/\epsilon)$, $\tilde \theta^\epsilon$ converges weakly to the solution of the standard heat equation,
with
an enhanced diffusion coefficient.
On intermediate time scales, we show $\tilde \theta^\epsilon$ converges to the solution of a \emph{time fractional} heat equation.
Again, this is somewhat unexpected, as the scaling limits of linear parabolic equations usually lead to a parabolic (spatially homogeneous) equation, and not a time fractional equation!

Explicitly, our main PDE result (Theorem~\ref{thmPDEHomog}) can be stated as follows.
\begin{theorem}
  For a fixed $\alpha \in (0, 1)$, define the rescaled functions $\theta^\epsilon$ and $\theta^\epsilon_0$ by
  \begin{equation}\label{eqnThetaEpsDef}
    \theta^\epsilon(x, t)
    = \tilde \theta^\epsilon \paren[\Big]{\frac{x}{\epsilon^{(1 - \alpha) / 4}}, \frac{\alpha \abs{\log \epsilon} t}{\epsilon^{1 - \alpha}} }
    \quad\text{and}\quad
    \theta_0^\epsilon( x ) = \tilde \theta_0^\epsilon \paren[\Big]{ \frac{x}{\epsilon^{(1 - \alpha) / 4} } }.
  \end{equation}
  If~$\theta^\epsilon_0 = \theta_0 \in \CC_b(\R^2)$ is independent of~$\epsilon$, then as $\epsilon \to 0$,
  $\theta^\epsilon$ converges%
  \footnote{%
    The notion by which $\theta^\epsilon \to \vartheta$ is related to the two scale convergence~\cite{Nguetseng89} and is described precisely later.
    Roughly speaking, one needs to test $\theta^\epsilon$ against an $\epsilon$-dependent measure $\hat \nu^\epsilon$ on $\R^2$, where the family of measures $(\hat \nu^\epsilon)$, when rescaled appropriately, converges to a probability measure supported on the separatrix.%
  }
  to $\vartheta$, where $\vartheta$ satisfies
  \begin{equation}\label{eqnVarTheta}
    r_0
    \caputo{1/2} \vartheta
      - \frac{1}{2} \brak{Q : \grad^2} \vartheta = 0,
    \qquad
    \vartheta(x, 0) = \theta_0(x),
  \end{equation}
  for some constant $r_0 > 0$ that can be computed explicitly in terms of $v$.
  Here $\caputo{1/2}$ denotes the \emph{Caputo derivative} of order $1/2$ (see for instance~\cite{Diethelm10}) and is defined by
  \begin{equation}\label{eqnCaputo12Def}
    \caputo{1/2} f
      = \frac{1}{\sqrt{\pi}} \frac{d}{dt} \int_0^t \frac{f(s) - f(0)}{(t - s)^{1/2}} \, ds,
  \end{equation}
  and $Q : \grad^2= \sum_{i,j} {Q}_{ij} \frac{\partial^2}{\partial x_i \partial x_j}$.
\end{theorem}

Time fractional equations of the form~\eqref{eqnVarTheta} often arise when studying anomalous, or non-Fickian diffusions.
In this context, it was first suggested by Young~\cite{Young88} (see also~\cite{YoungPumirEtAl89,YoungJones91}) and supported by both numerics and a heuristic explanation.
Roughly speaking, on intermediate time scales, the heat near the separatrix diffuses to neighbouring cells, and also gets trapped in cell interiors.
This leads to a coupled system governing the effective behaviour, and eliminating the heat in cell interiors from this system leads to~\eqref{eqnVarTheta}.
We elaborate on this and carry out the details in Section~\ref{sxnPDEHomog}.
We remark, however, that even though this is a purely deterministic result, we prove it using our main probabilistic result (Theorem~\ref{thmMainR2}) and the Kolmogorov equation.
In lieu of a rigorous PDE proof of this result, we provide (in Appendix~\ref{sxnAsymptoticExpansion}) a formal asymptotic expansion motivating it.

\subsection*{Plan of this paper}
In order to place our results in the context of the existing literature, Section~\ref{sxnLongShort} provides a brief overview of the effective behaviour of tracer particles on both long and short time scales.
This section is independent of the rest of the paper and can be skipped by the reader familiar with the literature.

In Section~\ref{sxnMainResults}, we state the main result of our paper 
(Theorem~\ref{thmMainR2}) proving the convergence of~$\tilde X$ to an effective process on intermediate time scales.
An important step in the proof is Theorem~\ref{thmAvgPrinShortTime}, which is an analogue of the Freidlin-Wentzell averaging principle on these time scales.
Before proving the two theorems stated above, we digress and prove an intermediate time homogenization result for the advection diffusion equation governing the density of tracer particles (Theorem~\ref{thmPDEHomog}).
This is presented in Section~\ref{sxnPDEHomog}, and is independent of all subsequent sections (except Appendix~\ref{sxnAsymptoticExpansion}).

The remainder of the paper is devoted to proving our main results.
In Section~\ref{sxnMainProof} we prove Theorem~\ref{thmMainR2}, modulo an estimate on how far~$Z$ can travel before exiting a small neighbourhood of the separatrix (Proposition~\ref{ppnCLTFirstHitShort}).
In Sections~\ref{sxnShortTimeAveraging} and~\ref{sxnShortTimeAveraging2}, we prove the intermediate time averaging principle (Theorem~\ref{thmAvgPrinShortTime}).
In Appendix~\ref{sxnCLTFirstHitShort} we prove Proposition~\ref{ppnCLTFirstHitShort}.
Finally, in Appendix~\ref{sxnAsymptoticExpansion}, we provide a formal asymptotic expansion, which serves as an alternative, purely PDE, approach to derive our intermediate time PDE homogenization result (Theorem~\ref{thmPDEHomog}).

\section{The effective short time and long time behaviour of tracer particles}\label{sxnLongShort}

This section contains a brief review of results concerning the effective behaviour of tracer particles on long time scales and short time scales.
Its main purpose is to place our results in the broader context of existing literature, and the familiar reader can skip directly to Section~\ref{sxnMainResults}.

\subsection{Homogenization: Effective behaviour on long time scales}

Well known homogenization results show that
on time scales much larger than the diffusive time scale $1/\epsilon$,
the effective behaviour of~$\tilde X$ is that of a Brownian motion with an enhanced diffusion coefficient.
Explicitly, consider the rescaled process $\tilde Z = \tilde Z^{\epsilon, \delta}$, defined by
\begin{equation}\label{eqnZepDeltaDef}
  \tilde Z_t = \tilde Z^{\e, \delta}_t \defeq \delta^{1/2} \tilde X_{t/\delta},
\end{equation}
where for clarity we
suppress
the dependence of $\tilde X$ and $\tilde Z$ on the parameters $\epsilon$ and $\delta$.
Freidlin~\cite{Freidlin64} (see also~\cite{Olla94,BensoussanLionsEtAl78,PavliotisStuart08}) proved that for fixed~$\epsilon$ we have
\begin{equation*}
  \tilde Z = \tilde Z^{\e, \delta} \xrightarrow[\delta \to 0]{\mathcal L}
    W^{D\eff(\epsilon)},
\end{equation*}
where $D\eff(\epsilon)$ is a constant $2 \times 2$ positive matrix known as the \emph{effective diffusivity}, and $W$ is a 2D Brownian motion with the covariance matrix $D\eff(\epsilon)$.
Intuitively, the temporal rescaling involves waiting for longer and longer times as $\delta \to 0$.
In this time, the process $\tilde Z$ spreads out further and further, and rescaling space by a factor of $\sqrt{\delta}$ produces a non-trivial limit.
The spatial rescaling is akin to an observer zooming out until the microscopic details of the cellular flow cannot be seen anymore and can effectively be replaced by a homogeneous background.

The effective diffusivity $D\eff(\epsilon)$ can be computed explicitly by solving a cell problem, and its asymptotic behaviour as $\epsilon \to 0$ has been extensively
studied~\cite{Childress79,RosenbluthBerkEtAl87,ChildressSoward89,FannjiangPapanicolaou94,Koralov04}.
In particular, it is well known that
\begin{equation}\label{eqnDeff}
  D\eff(\epsilon) = \CO(\sqrt{\epsilon})
\end{equation}
as $\epsilon \to 0$.
We observe that~$D\eff(\epsilon)$ is much larger than the molecular diffusivity $\epsilon$ in~\eqref{eqnSmallnoiseeq1} for small $\epsilon$.

To address the time scales involved, we consider the double limit of $\tilde Z$ as both $\epsilon$ and $\delta$ approach $0$.
Using~\cite{Fannjiang02} (see also~\cite{IyerKomorowskiEtAl14}) it follows that
\begin{equation}\label{eqnZTildeEpDeltaLim}
  \frac{\tilde Z}{ \sqrt{D\eff(\epsilon)} } \xrightarrow[\substack{\epsilon, \delta \to 0,\\\delta \ll \epsilon}]{\mathcal L} W.
\end{equation}
Rewriting this in terms of the original process, this means that~$\tilde X$ behaves like a rescaled Brownian motion on time scales \emph{much larger} than $1 / \epsilon$.


\subsection{Averaging and the effective behaviour on the transition time scale}

As discussed in the previous section, $\tilde X$ homogenizes on time scales larger than
$\CO(1 / \epsilon)$.
Under a compactness assumption (e.g., if the periodic flow is replaced by a flow on a torus)
classical results of Freidlin (discussed below) show that~$\tilde X$ averages along the flow lines of~$v$.
In the non-compact setting that we consider, a recent result~\cite{HairerKoralovPajorGyulai2014} shows that~$\tilde X$ transitions between the homogenized and averaged behaviour in a very natural way, and we describe this behaviour here.

To study the behaviour on time scales of order $t \approx 1 / \epsilon$, consider the time rescaled process $X$ defined by
\begin{equation}\label{timechange}
  X_t = X^\epsilon_t \defeq \tilde X_{t / \epsilon}.
\end{equation}
In this case, $X$ satisfies the SDE
\begin{equation*}
  dX_t=\frac{1}{\eps}v(X_t) \, dt+dW_t,
  \qquad\text{with }
  X_0 \sim \mu.
\end{equation*}
When $\epsilon$ is small, $X$ moves very fast along trajectories of $v$, and diffuses slowly across them.

To explain further,  assume that $H(x)$ is 1-periodic in $x_1$ and $x_2$.
Let $\torus = \R^2 / \Z^2$ be the two-dimensional 
torus, and $\pi:\R^2 \to \torus$ be the projection map.
The Reeb graph~\cite{Reeb46} of $H$ (where~$H$ is viewed as a function on the torus) is obtained by mapping the connected components of level sets of $H$ to individual points, and using a metric that is locally defined by~$H$.
For the Hamiltonians we consider, the Reeb graph is star shaped with each edge corresponding to a cell, and the distance to the vertex corresponding to the absolute value of the Hamiltonian.
One example is shown in Figure~\ref{fig:proj2graph}.
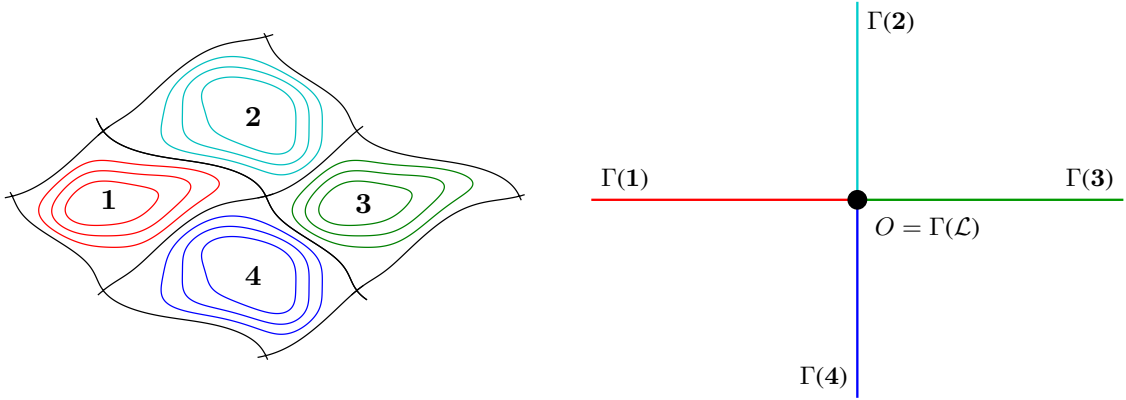
\begin{figure}[htb]
  \centering
  \begin{minipage}{.45\linewidth}
    \centering
    \begin{tikzpicture}
      \node[inner sep=0pt] (russell) at (0,0)
	{\includegraphics[width=\linewidth]{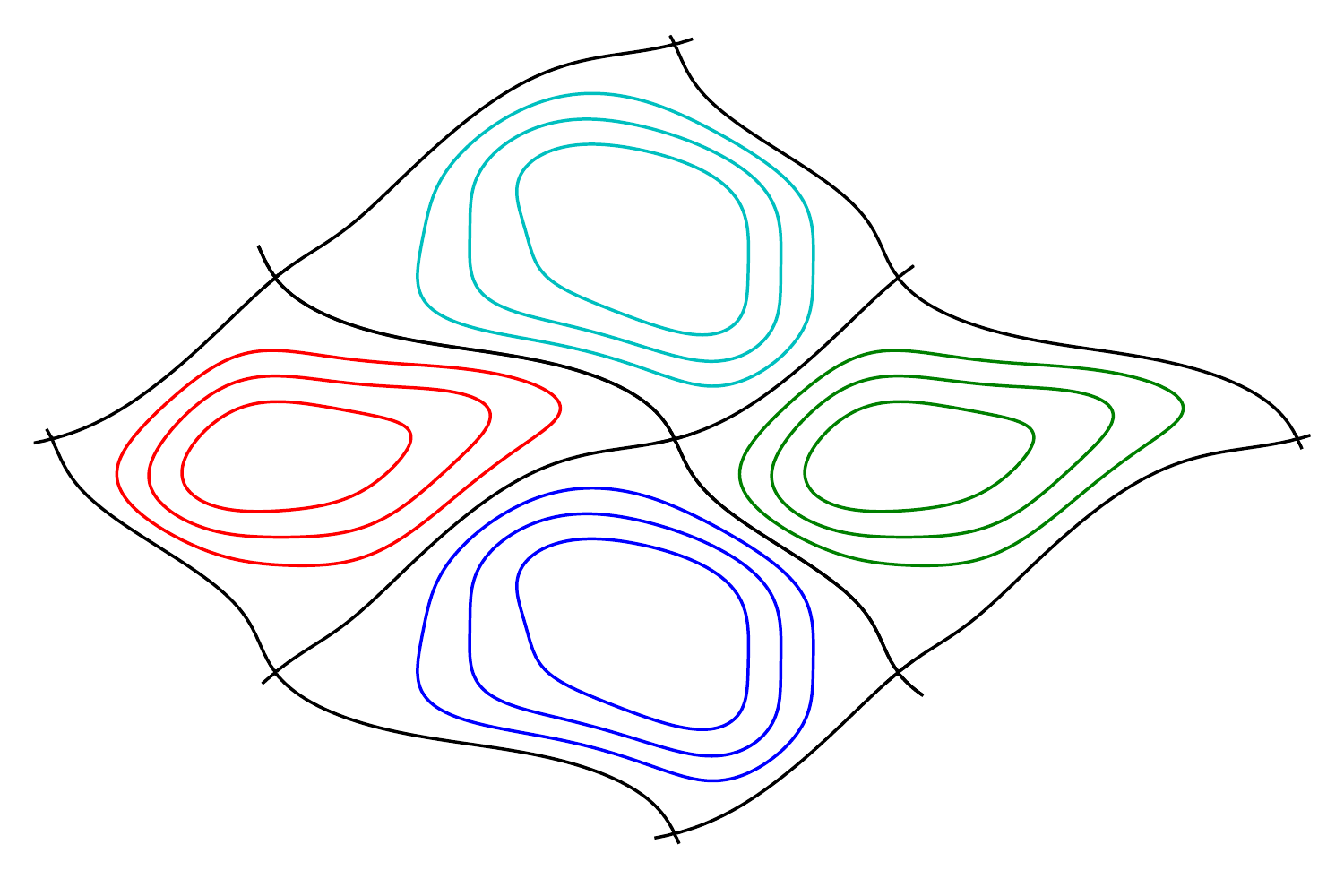}};
      \node at (-2.05,0) {$\mathbf 1$};
      \node at (-0.15,1.1) {$\mathbf 2$};
      \node at (1.3,-0.05) {$\mathbf 3$};
      \node at (-0.15,-1.0) {$\mathbf 4$};
    \end{tikzpicture}
  \end{minipage}
  \hfil
  \begin{minipage}{.45\linewidth}
    \centering
    \resizebox{\linewidth}{!}{
    \begin{tikzpicture}[line width=1pt]
      \draw[color=red] (0,0) -- (-4,0)
	node[above right]{\color{black}$\Gamma(\mathbf{1})$};
      \draw[color=cyan!80!black] (0,0) -- (0,3)
	node[below right]{\color{black}$\Gamma(\mathbf{2})$};
      \draw[color=green!60!black] (0,0) -- (4,0)
	node[above left]{\color{black}$\Gamma(\mathbf{3})$};
      \draw[color=blue] (0,0) -- (0,-3)
	node[above left]{\color{black}$\Gamma(\mathbf{4})$};
      \node[circle,inner sep=3pt,fill=black] at (0,0) [label=below right:{$O=\Gamma(\CL)$}] {};
    \end{tikzpicture}
    }
  \end{minipage}
\caption{The graph corresponding to the structure of the level sets of $H$ on $\torus$}
\label{fig:proj2graph}
\end{figure}

Given $x \in \torus$, define $\Gamma(x)$ to be the point on the Reeb graph corresponding to the connected component of the level set of $H$ that contains $x$.
Freidlin and Wentzell~\cite{FreidlinWentzell93} proved
that
\begin{equation*}
  \Gamma\paren[\big]{ \pi( X ) } \xrightarrow[\epsilon \to 0]{\mathcal L}
  Y,
\end{equation*}
where $Y$ is a diffusion on the Reeb graph with $Y_0 = \Gamma(X_0)$ and with a specific gluing condition at the interior vertex that can be determined explicitly in terms of the Hamiltonian~$H$.
(The exterior vertices are inaccessible and require no boundary condition.)
This is the averaging principle.%
\footnote{%
  Strictly speaking, the classical averaging principle~\cite{FreidlinWentzell93} requires $H(x) \to \infty$ as $\abs{x} \to \infty$ instead of compactness.
  These results can, however, be readily adapted to the scenario where the domain is compact.}

We emphasise that this only determines the effective behaviour of $X$ \emph{projected} onto the compact Reeb graph of $H$, when $H$ is viewed as a function on the torus.
A recent paper~\cite{HairerKoralovPajorGyulai2014} showed how this can be used to obtain the effective behaviour of $X$ on the whole plane $\R^2$.
The main theorem in~\cite{HairerKoralovPajorGyulai2014} shows that
\begin{equation}\label{eqnHKPG1}
  \epsilon^{1/4} X \xrightarrow[\epsilon\to0]{\mathcal L}
    W^{Q}_{L}.
\end{equation}
Here
$Q$ is a strictly positive definite matrix,
and
$W^{Q}$ is a Brownian motion with the covariance matrix $Q$.
The process $L$ is the local time of the limiting diffusion $Y$ at the vertex of the Reeb graph,
and is independent of~$W^Q$.
The notation~$W^Q_{L}$ in~\eqref{eqnHKPG1} above refers to the process $W^Q$, time changed by the process $L$.

 To relate this to the classical homogenization results, note that equation~\eqref{eqnHKPG1} provides information on the effective behaviour of $\tilde X$ on time scales of order $1 / \epsilon$; the borderline time scale, beyond which homogenization results are valid.
In fact, for the process~$\tilde Z^{\epsilon, \delta}$ defined by~\eqref{eqnZepDeltaDef}, the result of~\cite{HairerKoralovPajorGyulai2014} states that for $\delta = \epsilon$, the limiting process in~\eqref{eqnZTildeEpDeltaLim} is now a \emph{subordinated} Brownian motion.
In contrast, for $\delta \ll \epsilon$ (as we had in~\eqref{eqnZepDeltaDef}), the limiting process is simply an effective Brownian motion without any subordination.

In this spirit, even though the construction of the covariance matrix $Q$ in \cite{HairerKoralovPajorGyulai2014} is not explicit, we can find $Q$ by a matching argument with the existing literature on the effective diffusivity.
Indeed, since the process~$Y$ is ergodic on the Reeb graph, we must have
\begin{equation*}
  \lim_{t \to \infty} \frac{L(t)}{t} = \rho,
\end{equation*}
for some $\rho \in (0, \infty)$.
This implies that for large $t$, $W^{Q}_{L(t)}$ has approximately the same law as a Brownian motion with covariance matrix $\rho Q$.
Comparing this with \eqref{eqnZTildeEpDeltaLim} and \eqref{eqnHKPG1}, and taking the time change \eqref{timechange} into account, we get
\begin{equation}\label{Qmatrix}
Q=\frac{1}{\rho}\cdot\lim_{\e\to 0}\frac{D_{\text{eff}}(\e)}{\sqrt{\e}}.
\end{equation}

\subsection{Large Deviations: Effective behaviour on short time scales}

The next natural asymptotic regime is ``intermediate'' time scales for which $1 \ll t \ll 1 / \epsilon$.
This, however, is the main focus of our paper and is described along with our main results in Section~\ref{sxnMainResults}.
Instead, we conclude this section by briefly describing short time scales.

On time scales of order~$1$, the trajectories of $\tilde X$ deviate from the flow lines of~$v$ with an exponentially small probability.
To elaborate, let $\phi\colon \R^2 \times \R_+ \to \R^2$ be%
\footnote{
  Throughout this paper we use the convention that $\R_+ = [0, \infty)$.
}
the flow of the vector field~$v$, defined by the ordinary differential equation
\begin{equation}\label{eqnFlowOfV}
  \partial_t \phi_t(x) =v \circ \phi_t(x),
  \qquad \text{with}~~\phi_0(x) = x.
\end{equation}
  
Then,
for every $T,\eta>0$, we have
\begin{equation*}
  - \log \paren[\Big]{
      \Pe^x \paren[\Big]{
	\sup_{t\in[0,T]} \abs{\tilde{X}_t - \phi_t(x)}>\eta
      }
  } = \CO(\epsilon),
\end{equation*}
where we write $\Pe^x = \Pe^{\delta_x}$ for brevity.
For details, we refer the reader to \cite[Chapter 4, Theorem 1.1]{FreidlinWentzell12}.

This is not surprising as the qualitative effect of the noise is a motion across the flow lines on a time scale of order $1 / \e$, which is much longer than the order one natural time scale of the deterministic motion.
We remark, however, that at the slightly longer time scale $t \approx |\log(\varepsilon)| $, an interesting behaviour is observed near the separatrices.
The effective process in this regime is a piecewise constant \emph{non-Markovian} process that jumps between the saddle points of~$H$, and we refer the reader to~\cite{Bakhtin2010,BakhtinAlmada2011} for details. 

\section{Main results: Effective behaviour on intermediate time scales}\label{sxnMainResults}

The main contribution of this paper is the precise description of the effective behaviour of $\tilde X$ on \emph{intermediate} time scales where $1 \ll t \ll 1/\epsilon$.
As we have outlined earlier, the effective behaviour of $\tilde X$ on these time scales might seem trivial  at first glance.
Indeed, convection only transports $\tilde X$ along flow lines of~$v$, which are all closed orbits inside each cell.
On the other hand, for diffusion to transport~$\tilde X$ to a different cell, it will take time of order $1/\epsilon$, which is much longer than the time scales under consideration.
Thus, if $\tilde X$ starts at a generic point inside one of the cells, it will
simply make many rotations along the flow lines of~$v$ without escaping the cell.

The interesting behaviour is observed when $\tilde X$ starts close enough to (or on) the separatrix.
The diffusion is then strong enough to transport $\tilde X$ from one cell to another and, combined with the effect of the drift, the process $\tilde X$ can conceivably travel large distances in a short time.
Indeed, a recent result~\cite{IyerNovikov16} proves that on time scales for which $1 \ll t \ll 1 / \epsilon$, the variance of $\tilde X_t$ is of order $\sqrt{t}$, up to a logarithmic correction.
The main result of the present article goes much further than a
variance estimate, and provides an effective process on these intermediate time scales.

For a given $\alpha \in (0, 1)$, we study the behaviour of $\tilde X$ on time scales of order%
\footnote{%
  Choosing $\alpha \in (0, 1)$ and restricting to time scales of order~$\abs{\log \epsilon} / \epsilon^{1 - \alpha}$ is performed mainly for convenience, and does not have any bearing on the final result.
  In fact, the main results of this paper can be formulated more generally by choosing a parameter~$\delta = \delta(\epsilon)$ such that both $\delta \to 0$ and $\epsilon / \delta \to 0$ as $\epsilon \to 0$.
  Now a description of~$\tilde X$ on time scales of order~$\delta$ can be obtained from our main results by replacing all occurrences of $\epsilon^{1-\alpha}$, $\epsilon^{\alpha / 2}$ and $\alpha \abs{\ln \epsilon}$ with $\delta$, $(\delta / \epsilon)^{1/2}$ and $\ln (\delta / \epsilon)$ respectively.%
}
$\abs{\log \epsilon} / \epsilon^{1 - \alpha}$ using the time rescaled process~$Z = Z^\epsilon$ defined by~\eqref{eqnZDef}.
As before, we suppress the $\epsilon$-dependence of the process
$Z$ and use time as an argument instead of a subscript when notationally convenient.
Clearly, time scales of order $\abs{\log \epsilon} / \epsilon^{1-\alpha}$ are shorter than time scales of order $1/\epsilon$, and longer than time scales of order $1$.

We describe the effective behaviour of~$Z$ in two steps:
First, we compactify the state space by projecting $Z$ onto the periodic torus.
In this case, we prove a direct analogue of the classical Freidlin-Wentzell averaging principle \cite{FreidlinWentzell12} on shorter time scales, and show that the limiting process is a diffusion~$Y$ on a (rescaled) Reeb graph.
Next, we show that the limiting behaviour of~$Z$ on $\R^2$ is exactly an independent two-dimensional Brownian motion time changed by the local time of~$Y$ at the vertex of the rescaled Reeb graph.
That is, the effective process only moves when the graph diffusion~$Y$ is at the vertex.
These steps are described below in Sections~\ref{sxnIntermediateAveraging} and~\ref{sxnIntermediateR2} respectively.

\subsection{Intermediate time averaging on the torus}\label{sxnIntermediateAveraging}

The purpose of this section is to state an analogue of the classical Freidlin-Wentzell averaging principle \cite{FreidlinWentzell12} when~$Z$ is projected onto the torus.
While we state our result in the context of cellular flows, it is applicable more generally to behaviour of Hamiltonian systems  around heteroclinic connections.

We begin with some notation describing the geometry of the Hamiltonian and the projection on the Reeb graph.
We recall that we normalised $H$ so that it has period~$1$, and the separatrix, denoted by $\mathcal L$, is exactly
\begin{equation*}
  \mathcal L = \{ x \in \R^2 \;:\; H(x) = 0\},
\end{equation*}
and is assumed to be connected.
Let $\torus \defeq \R^2 / \Z^2$ be the torus, $\pi\colon \R^2 \to \torus$ be the projection map, and define $\CL_{\torus} = \pi( \CL)$.  
Let $A_1$, \dots, $A_M$ denote the saddle points of $H$ on the separatrix~$\CL_{\torus}$.
Then $\CL_{\torus}$ (or $\CL$) is the union of the saddles $\set{A_i}$ (or $\pi^{-1} (\set{A_i})$, respectively), and the heteroclinic orbits connecting these saddles.
For notational simplicity in the proof, we assume
that there are no homoclinic orbits (i.e.,\ orbits that connect a saddle to itself).

By Euler's polyhedron formula (recall that the torus has Euler characteristic zero), there are exactly~$M$ connected components of the complement of the separatrix $\torus \setminus \mathcal L_{\torus}$, and we denote these domains by $U_1$, \dots, $U_M$.
(There is, however, no particular relation between the numbering of the $U_i$'s and that of the $A_i$'s.)
For convenience, we further assume that there are no saddle points of $H$ in the interior of the sets $U_1$, \dots, $U_M$.

We now define the space $\cG$ that serves as the rescaled Reeb graph of~$H$.
Let $\cG$ be the topological quotient space obtained from $\{1,\dots,M\}\times\mathbb{R}^+$ by identifying all the points $(1,0), \ldots, (M,0)$ with each other.
We observe that~$\cG$ is a star shaped graph with semi-infinite edges $I_i \defeq \{i\}\times \mathbb{R}^+$, corresponding to the rescaled distance into the interior of~$U_i$, and one interior vertex $O=(1,0) = \cdots = (M, 0)$ corresponding to the separatrix~$\mathcal L_{\torus}$. A natural metric on $\cG$ is given by
\begin{equ}[e:defdG]
d_\cG((i,y),(j,\bar y)) = 
\left\{\begin{array}{cl}
	|y-\bar y| & \text{if $i=j$,} \\
	|y| + |\bar y| & \text{otherwise.}
\end{array}\right.
\end{equ}
Define the projection $\Gammaeps \colon \torus \to \cG$ by
\begin{equation}\label{eqnGammaepsDef}
  \Gammaeps(x) = \paren[\big]{ i,\e^{-\alpha/2}|H(x)| }
  \quad\text{if }
  x \in \overline{U}_i,
\end{equation}
and extend it periodically to $\R^2$.
Note that $\Gammaeps$ projects each invariant region $U_i$ into an edge $I_i$ of the graph and $\Gammaeps(\CL_{\torus})=O$.
Even though the sets $\overline{U}_i$ overlap, the map $\Gammaeps$ is well-defined since the points $(i,0)$ have all been identified.

We claim that the law of the projected process
$(\Gammaeps \circ \pi)(Z)$
converges weakly to a process $Y$ on the graph $\cG$.
The limiting process~$Y$ is a diffusion, and it can be characterised by its generator~$\mathcal A$. We describe this before stating the main convergence result of this section.

On the $i$-th edge of the graph, define the operator $\mathcal A_i$ by
\[
  \mathcal{A}_i=\frac{a_i}{2}D_i^2.
\]
Here $D_i$ denotes the derivative along the $i$-\text{th} edge of~$\cG$, and the coefficients $a_i$ are defined by
\begin{equation}\label{eq:shorttimecoeff}
a_i \defeq q_i \cdot\lim_{\e \to 0}\frac{\abs{\log\epsilon}}{T_i(\e^{1/2})},
\end{equation}
where
\begin{equation}\label{eqnQiTi}
  q_i=\oint_{\partial U_i} \abs{\nabla H} \,dl
  \qquad\text{and}\qquad
  T_i(h)=\oint_{\{\abs{H(x)}=h\}\cap U_i} \abs{\nabla H}^{-1} \, dl.
\end{equation}
We recall that $T_i(h)$ 
is the time the flow~$\phi$ (defined in equation~\eqref{eqnFlowOfV}) takes to complete one rotation along the periodic orbit starting from any point $x \in U_i$ for which $\abs{H(x)} = h$.
Since $\partial U_i$ contains hyperbolic saddles, we know that as $h \to 0$, the period $T_i(h)$ 
diverges at a logarithmic rate, and hence $a_i$ is finite and strictly positive.

Now we define the domain $D(\mathcal A)$ to be the set of all functions $F$ that satisfy the following conditions:
\begin{enumerate}[\hspace{4ex}(a)]
  \item
   $F \in \CC_0(\cG) \cap \CC^2( \cG \setminus\{O\})$.
    That is, $F$ is continuous on $\cG$, tends to zero at infinity, and is twice continuously differentiable away from the interior vertex $O$.

  \item
Writing $\one_A$ for the
indicator function of a set $A$, the function
    \begin{equation}\label{eqnADef}
      y \mapsto \sum_{i = 1}^M \one_{\set{y \in I_i}} \mathcal A_i F(y),
    \end{equation}
    defined on $\cG \setminus\{O\}$, extends to a $\CC_0$ function on all of $\cG$.
  \item
    The function~$F$ satisfies the flux condition
    \begin{equation*}
      \sum_{i=1}^M q_iD_iF(O)=0,
    \end{equation*}
\end{enumerate}
Finally, for $F \in D(\mathcal A)$ we define $\mathcal A F$ to be the unique $\CC_0(\cG)$ extension of the function defined in~\eqref{eqnADef}.

Replicating results from~\cite{Mandl68} for our one-dimensional operator, we can show that for every $u \in D(\mathcal A)$ and $\lambda > 0$ the resolvent equation $\lambda f - \mathcal{A}f = u$ has a unique solution $f \in \mathcal{D}(\mathcal{A})$.
Since $\mathcal A$ is a closed operator, the Hille-Yosida theorem~\cite[Theorem 2.2 in Ch 4]{EthierKurtz86}, shows that $\mathcal{A}$ generates a strongly continuous positive contraction semigroup on $\CC_0(\cG) = \overline{D(\mathcal{A})}$, the $L^\infty$
closure of  $D(\mathcal{A})$.
Therefore there is a Fellerian Markov family~$Y$ with generator $\mathcal{A}$ (see~\cite[\S 4.1-4.2]{EthierKurtz86}), and we use the Kolmogorov continuity theorem to replace~$Y$ with a modification with continuous trajectories on~$\cG$.
The process $Y$ will arise in the main result of this subsection (Theorem~\ref{thmAvgPrinShortTime}, below) as the weak limit of $\Gammaeps (Z)$.

As with the process $Z$, we transfer the dependence of~$Y$ on its initial position to the associated probability measure.
When the law of $Y_0$ is $\mu$, we denote the corresponding probability measure on $\CC(\R_+,\cG)$ by $\Pb^\mu$, and the associated expectation operator by~$\Eb^\mu$.
When $\mu$ is concentrated at a point $y\in \cG$, we will simply write $\Pb^y$ and $\Eb^y$ as appropriate.

\begin{remark}\label{rmkConstruct}
The process~$Y$ can alternately be constructed directly
as follows.
Take a standard Brownian motion $B$ and decompose it into excursions away from the origin. Since there are countably many such excursions, we can enumerate them.
Say that the $k$th excursion happens during the interval $(s_k, t_k)$, then these intervals are all disjoint and the complement of their union consists precisely of the
null set of times for which $B(t) = 0$.
Consider now a sequence $i_k$ of i.i.d.\ $\{1,\ldots,M\}$-valued random variables independent of $B$ with $\P(i_k = j)$ proportional to $q_j / \sqrt{a_j}$.
We then define a $\cG$-valued process $Y$ in the following way.
If $t \in (s_k,t_k)$ for some $k$ and $i_k = j$, we set $Y_t = (j, \sqrt{a_j} |B(t)|) \in \cG$, otherwise we set $Y_t = O$.
In order to start this process with an initial condition $y = (j,c) \neq O$, one can perform the same construction, with the difference that one sets $B(0) = c/\sqrt{a_j}$, and we set the value $i_k$ corresponding to the excursion containing time $0$ to $j$.
We refer the reader to~\cite{Lejay06} for more details and various other constructions.
\end{remark}

We can now state the main result of this section.

\begin{theorem}\label{thmAvgPrinShortTime}
  Let $\mu^{\e}$ be a family of probability measures on $\torus$ such that the push forward measures~$\Gammaeps^* \mu^\epsilon$ converge weakly, as $\epsilon\to0$, to a probability measure $\mu$ on $\cG$.
  Then the laws of $\Gammaeps(Z)$ under $\Pe^{\mu^\eps}$
converge weakly 
to that of the
process $Y$ under $\Pb^\mu$.
That is, for every bounded continuous $f \colon \CC(\R_+; \cG) \to \R$, we have
\begin{equation*}
  \lim_{\epsilon \to 0} \Ee^{\mu^\epsilon} f( \Gammaeps (Z_\cdot) )
  =  \Eb^{\mu} f( Y_{\cdot} )\;.
\end{equation*}
Moreover, for fixed $\mu$ and $f$, the convergence is uniform with respect to all choices of $\mu^\eps$ such that $\Gammaeps^*\mu^{\e} =  \mu$.
\end{theorem}

The proof of Theorem~\ref{thmAvgPrinShortTime} is given in Section~\ref{sxnShortTimeAveraging}.
As mentioned earlier, Theorem~\ref{thmAvgPrinShortTime} is not restricted to the cellular flow setting of this paper.
It describes the behaviour of a generic Hamiltonian system around heteroclinic connections, and serves as a direct analogue of the classical Freidlin-Wentzell averaging principle \cite{FreidlinWentzell12} at shorter time scales.

\subsection{The intermediate time behaviour on the plane}\label{sxnIntermediateR2}

Theorem~\ref{thmAvgPrinShortTime}, stated in the previous section, shows that a limiting behaviour of the projection $\Gammaeps(Z)$ is a diffusion on the rescaled Reeb graph~$\cG$.
In this section, we show that the limiting behaviour of $Z$ itself is an independent Brownian motion on the plane time changed by the local time of~$Y$ at~$O$.
We begin by recalling the abstract definition of the local time of~$Y$.


\begin{definition}\label{def_local_time}
The local time of $Y$ is the unique nonnegative random field
\[
  L \defeq \{L_t(y)\,:\, (t,y)\in \R_+\times \cG\}
\]
such that the following hold:
\begin{enumerate}[\hspace{2ex}(a)]
\item The mapping $(t,y)\to L_t(y)$ is jointly measurable, and $L_t(y)$ is adapted.
\item For every $y \in \cG$, the mapping $t\to L_t(y)$ is non-decreasing and constant on all open intervals for which $Y_t\neq y$.
\item For every bounded Borel measurable $f:\cG\to \R_+$ and every $y_0 \in \cG$, we have
\[
\int_0^t f(Y_s) a(Y_s)\,ds=2\int_\cG f(y)L_t(y)\,dy\qquad \text{$\Pb^{y_0}$-a.s.},
\]
where $a:\cG \to \R^+$ is defined%
\footnote{%
  Strictly speaking, $a(O)$ is not well defined.
  This, however, does not affect the left hand side since, with probability one, the process $Y$ spends time of measure zero at the interior vertex~$O$.
}
by $a(y) = a_i$ if $y \in I_i$.
\item For every $y_0 \in \cG$, $L_t(y)$ is $\Pb^{y_0}$-a.s.\ jointly continuous in $t$ and $y$ for $y\neq O$.
  Moreover, at $O$, we have
\[
  L_t(O)=\sum_{i=1}^M~ \lim_{\substack{y\to O,\\ y\in I_i}}L_t(y)\;.
\]
\end{enumerate}
\end{definition}
The existence and uniqueness of local time for diffusions on the real line is relatively well studied (see for instance~\cite{RogersWilliams00,RogersWilliams00a}).
These standard results, together with \cite[Lemma~2.2]{freidlin2000sheu},  give the existence and uniqueness of~$L$.
Moreover, in view of Remark~\ref{rmkConstruct}, the process $L_t(O)$ is a constant multiple of the local time of a one-dimensional Brownian motion $B$ at the origin.


With this definition we state our result concerning the limiting behaviour of $Z$.
To state our full convergence result, we introduce the state space
\begin{equ}
\Gb \defeq \R^2 \times \cG \;,\quad
\text{with metric }
d_\Gb((x,g), (\bar x, \bar g)) \defeq \abs{x - \bar x} + d_\cG(g,\bar g)\;.
\end{equ}
Our main probabilistic convergence result then reads as follows.

\begin{theorem}\label{thmMainR2}
  Define $\hge:\R^2 \to \Gb$ by
  \begin{equation*}
    \hge(x) =
    \paren[\big]{
      \epsilon^{(1 - \alpha)/4} x,
      (\Gammaeps \circ \pi)\paren{x}
    }\;,
  \end{equation*}
  and let $\nu_\epsilon$ be a family of probability measures on $\R^2$ such that the push forward measures $\hge^* \nu_\epsilon$ converge weakly to a probability measure  $\hat\nu$ on $\Gb$.
  Then, there exists a pair of processes $(\tilde W^Q, Y)$ defined on
  some probability space $(\bar \Omega, \bar{\mathcal F}, \Pb^{\hat\nu})$
  such that the following hold.
  \begin{enumerate}[\hspace{2ex}(a)]
    \item
      The initial distribution of $(\tilde W^Q, Y)$ under $\Pb^{\hat\nu}$ is $\hat\nu$.
    \item
      The process~$\tilde W^Q_\cdot - \tilde W^Q_0$ is a Brownian motion on $\R^2$ with strictly positive definite covariance matrix~$Q$.
    \item
      The process $Y$ is a diffusion on the graph $\cG$ with generator~$\mathcal A$.
      Moreover, conditioned on $(\tilde W^Q_0, Y_0)$, the processes $\tilde W^Q$ and~$Y$ are independent.
    \item
      As $\epsilon \to 0$, the law of the process $\hge(Z_\cdot)$ under~$\Pe^{\nu_\epsilon}$ converges weakly to that of $\Xi \defeq (\tilde{W}^Q_{L_\cdot}, Y)$ under~$\Pb^{\hat\nu}$, where $L_t = L_t(O)$ is the local time of $Y$ at the interior vertex $O$.
  \end{enumerate}
\end{theorem}

Since $L_t(O)$ is simply a constant multiple of Brownian local time,
$\tilde W^Q_L$ is a fractional kinetic process of index $1/2$.
This process arises naturally as the scaling limit of many trap models, such as continuous time random walks with heavy tailed jump times \cite{MeerSheff04,MeerScheff08} or the Bouchaud trap model \cite{Bouchaud,BenArousCerny07}. 
Intuitively, the scaling limit of the time of an excursion of $X_t$ away from the separatrix (when the process is trapped inside a cell) is approximately an excursion of a Brownian motion, and its length is accordingly heavy tailed with index $1/2$. 

If the support of $\hat\nu$ concentrates on $\R^2 \times \set{O}$,
then by Brownian scaling, $\Eb^{O} L_t=c\sqrt{t}$ for some constant $c(a_i,q_i)>0$.
In this case the variance of the limit process $\tilde W^Q_L$ is proportional to $\sqrt{t}$ for all time.
This was proved earlier in~\cite{IyerNovikov16} in the case $H(x_1,x_2)=\sin(x_1)\sin(x_2)$.
The proof of Theorem~\ref{thmMainR2} is presented in Section~\ref{sxnMainProof}.
Even though many ingredients in the proofs rely on the corresponding techniques from the companion paper~\cite{HairerKoralovPajorGyulai2014},
we keep the current paper self contained by sketching the main steps and highlighting the differences involved.

\section[PDE Application: Intermediate time homogenization]{A PDE Application: Intermediate time homogenization of the advection diffusion equation}\label{sxnPDEHomog}

This section is devoted to an intermediate time homogenization result for the advection diffusion equation.  We emphasize that the proof of the probabilistic result (Theorem \ref{thmMainR2}) does not rely on the arguments in this section.
The PDE~\eqref{eqnTildeTheta} is closely related to the process~$\tilde X$, and understanding the behaviour of~$\tilde X$ on intermediate time scales yields an intermediate time homogenization result for~\eqref{eqnTildeTheta} in a natural way.
Since~$\tilde X$ behaves like a fractional kinetic process on these time scales, it is natural to expect that~$\tilde \theta$ satisfies the (time) fractional heat equation~\eqref{eqnVarTheta}, and this was heuristically derived by Young~\cite{Young88}.
We prove it rigorously below (Theorem~\ref{thmPDEHomog}) using Theorem~\ref{thmMainR2}.

\begin{theorem}\label{thmPDEHomog}
  Let~$\tilde \theta^\epsilon$ satisfy the PDE~\eqref{eqnTildeTheta} for $x \in \R^2$, $t > 0$ with initial data~$\tilde \theta^\epsilon_0$.
  For~$\alpha \in (0, 1)$ define the rescaled functions $\theta^\epsilon$ and $\theta^\epsilon_0$ by~\eqref{eqnThetaEpsDef},
  and suppose $\theta^\epsilon_0 = \theta_0 \in \CC_b(\R^2)$ and is independent of $\epsilon$.
  Define the rescaled projection $\bge$ by $\bge(x) = \hge(x / \epsilon^{(1 - \alpha)/4})$.
  Then for any family of probability measures $\bar \nu^\epsilon$ on $\R^2$ such that $\bge^* \bar \nu^\epsilon$ converges weakly to a probability measure $\nu$ on $\Gb$, we have

  \begin{equation}\label{eqnLimThEp}
    \lim_{\epsilon \to 0} \int_{\R^2} \theta^\epsilon(x, t) \, d\bar \nu^\epsilon(x)
    = \int_{\R^2 \times \cG} \theta( x, y, t ) \, d\nu(x, y).
  \end{equation}
  Here $\theta$ is the unique classical  solution to the system
  \minilab{eqnThetaSys}
  \begin{equs}[2]
      \partial_t \theta - \mathcal A_y \theta &= 0
      \qquad
    & \text{for } y &\neq O,\, t > 0,
    \label{eqnThetaY}
    \\
      \frac{1}{2} \brak{Q:\grad^2_x} \theta
      + \sum_{i = 1}^M \bar q_i D^y_{i} \theta &= 0
      \qquad
    & \text{for } y &= O,\, t > 0,
    \label{eqnThetaX}
    \\
      \theta(x, y, 0) &= \theta_0(x). 
    \label{eqnThetaID}
  \end{equs}
  Here $\bar q_i = q_i / \sum_{j=1}^M q_j$, $D^y_i = D_i$ denotes the derivative along the $i^\text{th}$ edge of $\cG$,
  and $\mathcal A_y$ is the generator of the process $Y$ acting only on the variable~$y$.

  Moreover, if $\nu = \nu' \times \delta_O$, then
  \begin{equation*}
    \lim_{\epsilon \to 0} \int_{\R^2} \theta^\epsilon(x, t) \, d\bar \nu^\epsilon(x)
    = \int_{\R^2} \vartheta(x, t) \, d\nu'(x),
  \end{equation*}
  where $\vartheta(x, t) \defeq \theta(x, O, t)$ satisfies the Caputo time fractional equation 
  \begin{equation}\label{eqnThetaFPDE}
    \paren[\Big]{ \sum_{i=1}^M \frac{\bar q_i}{\sqrt{a_i / 2}} }
    \caputo{1/2} \vartheta
    - \frac{1}{2} \brak{Q : \grad^2} \vartheta = 0,
  \end{equation}
  with initial data $\theta_0$.
  Here $\caputo{1/2}$ denotes the \emph{Caputo derivative} of order $1/2$ defined by~\eqref{eqnCaputo12Def}.
\end{theorem}

\begin{remark*} 
Even though this is a purely deterministic result, our proof is probabilistic and relies on Theorem~\ref{thmMainR2}.
In lieu of additionally presenting a direct PDE proof, we provide in Appendix~\ref{sxnAsymptoticExpansion} a formal asymptotic expansion motivating~\eqref{eqnThetaSys}.
\end{remark*}

Well-posedness and regularity of 
solutions to~\eqref{eqnThetaSys} is standard. Parabolic problems similar to~\eqref{eqnThetaFPDE}  and the regularity of their solutions are discussed in, e.g.~\cite{Allen2015}. 
The notion of convergence used in equation~\eqref{eqnLimThEp} is known as two-scale convergence, and was introduced by Nguetseng~\cite{Nguetseng89}.
It has proved to be an invaluable tool in the theory of homogenization and has been applied various contexts.
In most situations, however, the underlying small-scale manifold is the torus.
The key difference in~\eqref{eqnLimThEp} is that the underlying small scale naturally arises as the rescaled Reeb graph of the Hamiltonian.

Before proving Theorem~\ref{thmPDEHomog}, we momentarily pause to consider an illustrative special case.
For $x\in \R^2$, take the sequence of measures $\bar \nu^\epsilon$ defined by
\begin{equation*}
  \bar{\nu}^\epsilon = \delta\paren[\Big]{ \epsilon^{(1 - \alpha)/4} \floor[\Big]{ \frac{x}{\epsilon^{(1 - \alpha)/4}} } },
\end{equation*}
where $\delta(z)$ denotes the delta measure supported at the point $z \in \R^2$, and we assume that the origin belongs to the separatrix (and therefore so does every point with integer coordinates).
Applying Theorem~\ref{thmPDEHomog} now shows that
\begin{equation*}
  \theta^\epsilon\paren[\Big]{ \epsilon^{(1 - \alpha)/4} \floor[\Big]{ \frac{x}{\epsilon^{(1 - \alpha)/4}} }, t }
    \xrightarrow{\epsilon \to 0} \vartheta(x, t),
\end{equation*}
for all $x \in \R^2$ and $t > 0$.
That is, at time~$t$ the value of the temperature~$\theta^\epsilon$ at the corner of the domain of periodicity containing~$x$ converges to $\vartheta(x, t)$.

At first sight, this is extremely surprising.
Long time scaling limits of~\eqref{eqnTildeTheta} have been studied extensively, and the limiting behaviour is simply the heat equation with an enhanced diffusion coefficient.
In our situation, the ``intermediate time'' scaling limit of~\eqref{eqnTildeTheta} is a \emph{time fractional} heat equation~\eqref{eqnThetaFPDE}!

The heuristic explanation of this is as follows.
On time scales shorter than $1/\epsilon$, any heat trapped in the interior of one cell will not escape the cell.
Thus to observe a non-trivial limiting behaviour, at these time scales one needs to zoom in close to the separatrix.
Indeed, if the family of measures $\hge^* \hat \nu^\epsilon$ converge weakly to a probability measure (as required in Theorem~\ref{thmMainR2}), the supports of $\hat\nu^\epsilon$  must asymptotically concentrate on the separatrix.

Now, in a small neighbourhood of the separatrix, there are two effects at play: heat diffuses to neighbouring cells, and heat is ``trapped'' in the cell interior.
Thus the limiting behaviour should be a coupled system balancing these two effects.
This is precisely what~\eqref{eqnThetaY} and~\eqref{eqnThetaX} accomplish.
Many similar models for anomalous diffusion have been studied by various authors. For example, Young~\cite{Young88} (see also~\cite{YoungJones91}) heuristically derived a similar system in the context of cellular flows.
Theorem~\ref{thmPDEHomog} establishes this rigorously, and we prove it below.

\begin{proof}[Proof of Theorem~\ref{thmPDEHomog}]
  By the Feynman-Kac formula
  \begin{equation*}
    \theta^\epsilon(x, t)
    = \tilde \theta^\epsilon \paren[\Big]{\frac{x}{\epsilon^{(1 - \alpha) / 4}}, \frac{\alpha \abs{\log \epsilon} t}{\epsilon^{1 - \alpha}} }
    = \Ee^{x / \epsilon^{(1 - \alpha)/4}}
    \theta_0( \epsilon^{(1 - \alpha)/4} Z_t ).
  \end{equation*}
  Hence, by Theorem~\ref{thmMainR2},
  \begin{equation}\label{eqnLimTh1}
    \int_{\R^2} \theta^\epsilon(x, t) \, d\bar \nu^\epsilon(x)
    = \Ee^{\hat \nu^\epsilon}
    \theta_0 \paren[\Big]{ \epsilon^{(1 - \alpha) / 4} Z_t }
    \xrightarrow{\epsilon \to 0}
    \Eb^{\nu} \theta_0( \tilde W^Q_{L(t)} ),
  \end{equation}
  where $\hat \nu^\epsilon$ is the rescaled measure defined by
  \begin{equation*}
    d\hat \nu^{\epsilon}(x) = d\bar \nu^{\epsilon}\paren[\big]{\epsilon^{(1 - \alpha)/4} x},
  \end{equation*}
  and $\tilde W^Q$ and $L$ are as in Theorem~\ref{thmMainR2}.
  Define the function $\theta\colon\R^2 \times \cG \times \R_+ \to \R$ by
  \begin{equation}\label{eqnThetaDef}
    \theta( x, y, t)
    \defeq \Eb^{(x,y)} \theta_0( \tilde W^Q_{L(t)} ),
  \end{equation}
  where we recall $\Eb^{(x,y)}$ is the expectation operator with respect to the probability measure $\Pb^{(x,y)}$ under which $\Pb^{(x,y)} ( \tilde W^Q_0 = x ~~\&~~ Y_0 = y ) = 1$.
  Now,
  \begin{equation*}
    \Eb^{\nu} \theta_0( \tilde W^Q_{L(t)} )
      = \int_{\R^2 \times \cG} \theta(x, y, t) \, d\nu(x,y)
  \end{equation*}
  and hence~\eqref{eqnLimThEp} follows from~\eqref{eqnLimTh1}.

  The fact that $\theta$ satisfies the system~\eqref{eqnThetaSys}
   follows from~\eqref{eqnThetaDef} and an It\^o formula for~$Y$ that was proved in~\cite{freidlin2000sheu}.
  Since this is interesting in its own right, we single it out as a proposition (Proposition~\ref{ppnFKY}, below) and defer it to the end of this section.

  Finally, to prove~\eqref{eqnThetaFPDE} when $\nu = \nu' \times \delta_O$, we only need to show that given a solution to~\eqref{eqnThetaSys}, the function~$\vartheta(x, t) \defeq \theta(x, O, t)$ satisfies~\eqref{eqnThetaFPDE}.
  This follows from the explicit solution formula for the heat equation on the half line, and similar results are readily available in the literature (see for instance~\cite[\S4.5]{MeerschaertSikorskii12}).
  For convenience, we derive it below.

  Along the $i^\text{th}$ edge, equation~\eqref{eqnThetaY} is simply the one dimensional heat equation.
  Rearranging~\eqref{eqnThetaX}, we obtain the boundary condition
  \begin{equation}\label{eqnThetaIBC}
    D^y_i \theta(x, O, t) = -\frac{1}{\bar q_i} \paren[\Big]{
      \frac{1}{2} \brak{Q : \grad_x^2} \theta(x, O, t)
	+ \sum_{j \neq i}^M \bar q_j D^y_{j} \theta(x, O, t)
    }.
  \end{equation}
  Treating the right hand side of~\eqref{eqnThetaIBC} as a given function, we can explicitly solve~\eqref{eqnThetaY} on the $i^\text{th}$ edge, with boundary condition~\eqref{eqnThetaIBC} and constant (in $y$) initial data $\theta_0(x)$.
  This gives
  \begin{equation*}
    \theta(x, O, t)
      = \theta_0(x)
	+ \frac{1}{\bar q_i} \paren[\Big]{ \frac{a_i}{2 \pi} }^{1/2}
	  \int_0^t 
	    \paren[\Big]{
	      \frac{1}{2} \brak{Q: \grad_x^2} \theta(x, O, s)
		+ \sum_{j \neq i}^M \bar q_j D^y_{j} \theta(x, O, s)
	    }
	    \frac{ds}{\sqrt{t - s}},
  \end{equation*}
  for every $i \in \set{1, \dots, M}$.
  Multiplying both sides by $\bar q_i / \sqrt{a_i/2}$, summing over $i$ and using~\eqref{eqnThetaX} yields
  \begin{equation*}
    \sum_{i =1}^M \frac{\bar q_i}{\sqrt{a_i / 2}}
      \paren[\Big]{ \theta(x, O, t) - \theta_0(x) }
      = \frac{1}{2 \sqrt{\pi}}
	  \int_0^t \brak{ Q : \grad_x^2} \theta(x, O, s)
	    \frac{ds}{\sqrt{t - s}}.
  \end{equation*}
  Applying $\caputo{1/2}$ to both sides and using $\vartheta(x, t) = \theta(x, O, t)$ yields~\eqref{eqnThetaFPDE} as desired.
\end{proof}

In the above proof we used the fact that~$\theta$ defined by~\eqref{eqnThetaDef} satisfies the system~\eqref{eqnThetaSys}.
We state and prove this next (see also~\cite{PajorGyulaiSalins15} for a related result).

\begin{proposition}\label{ppnFKY}
  Let $\theta_0 \in \CC_b(\Gb)$, and define $\theta$ by~\eqref{eqnThetaDef}.
  Then~$\theta$ satisfies the system~\eqref{eqnThetaSys} for $t>0$, and is continuous at $t = 0$.
\end{proposition}
\begin{proof}
  The first step is to obtain an It\^o formula for the process~$\Xi \defeq ( \tilde W^Q_L, Y )$.
  For the process~$Y$ alone, an It\^o formula is known and can be found in Freidlin and Sheu~\cite{freidlin2000sheu}.
  Explicitly, there exists%
  \footnote{
    The Brownian motion~$B$ can be directly obtained from the construction outlined in Remark~\ref{rmkConstruct}.
    Indeed, if~$\tilde B$ denotes the Brownian motion in Remark~\ref{rmkConstruct}, then we have $dB = \sign(\tilde B) \, d\tilde B$.
  }
  a Brownian motion~$B$ such that
  \begin{equation}\label{eq:Ito_in_Freidlin-Sheu}
    f(Y_t) - f(Y_0) = 
	\sum_{i = 1}^M \int_0^t
	    D_i f(Y_s) \sigma_i(Y_s) \, dB_s
	+ \int_0^t \mathcal A_y f(Y_s) \, ds
	+ \sum_{i=1}^M \bar q_i D_i f(O) \, L_t,
  \end{equation}
  holds any $f \in \CC^2_b(\cG)$.
  Here $\sigma_i(y) = \sqrt{a_i}$ if $y \in I_i$ and $\sigma_i(y) = 0$ otherwise.

  Now, since $\tilde W^Q$ and $Y$ are independent, the time changed process $\tilde W^Q_L$ is a martingale with joint quadratic variations given by
  \begin{equation*}
    d\qv{\tilde W^{Q,i}_L, \tilde W^{Q,j}_L}_t = Q_{i,j} dL_t
    \qquad\text{and}\qquad
    d\qv{\tilde W^{Q,i}_L, B}_t = 0
  \end{equation*}
  for all $i,j \in \set{1, 2}$.
  Here $\tilde W^{Q,i}$ denotes the $i^\text{th}$ component of $\tilde W^Q$, and~$Q_{i,j}$ is the $i$-$j^\text{th}$ entry of the matrix~$Q$.
  For $f \in \CC^2_b(\R^2 \times \cG)$, we thus obtain by It\^o's formula
  \begin{multline}\label{eqnXiIto}
    f(\Xi_t) - f(\Xi_0) = 
	\sum_{i = 1}^M \int_0^t
	    D_i^y f(\Xi_s) \sigma_i(Y_s) dB_s
	+ \int_0^t \mathcal A_y f(\Xi_s) \, ds
	\\
        + \sum_{i=1}^2 \int_0^t
	    \partial_{x_i} f(\Xi_s)
	    \, d \tilde W^{Q,i}_{L_s}
	+ \int_0^t \paren[\Big]{
	    \frac{1}{2} \brak{Q : \grad_x^2} 
	    + \sum_{i=1}^M \bar q_i D_i^y 
	  } f(\Xi_s) \, dL_s.
  \end{multline}

  Now we use the It\^o formula to compute~$\mathcal A_\Xi$, the generator of $\Xi$.
  Indeed, for~$f \in \CC^2_b(\R^2 \times \cG)$ we have
  \begin{align}
    \nonumber
    \mathcal A_\Xi f(x,y)
      &= \lim_{t \to 0} \frac{1}{t} \Eb^{(x, y)} \paren{ f(\Xi_t) - f(\Xi_0) }
    \\
    \label{eqnAxi1}
      &= \lim_{t \to 0} \frac{1}{t} \Eb^{(x,y)} \paren[\Big]{
	  \int_0^t \mathcal A_y f(\Xi_s) \, ds
	  + \int_0^t \paren[\Big]{
	      \frac{1}{2} \brak{Q : \grad_x^2} 
	      + \sum_{i=1}^M \bar q_i D_i^y 
	    } f(\Xi_s) \, dL_s
	},
  \end{align}
  since the other two terms on the right of~\eqref{eqnXiIto} are martingales and have expectation~$0$.
  Now, as $t \to 0$, the first term on the right of~\eqref{eqnAxi1} converges to~$\mathcal A_y f(x, y)$.
  For the second term on the right of~\eqref{eqnAxi1}, the fact that~$L$ is a constant multiple of Brownian local time gives
  \begin{equation*}
    \Eb^{(x,y)}
    \int_0^t \paren[\Big]{
      \frac{1}{2} \brak{Q : \grad_x^2} 
      + \sum_{i=1}^M \bar q_i D_i^y 
    } f(\Xi_s) \, dL_s
    = \begin{dcases}
	o(t) & y \neq O,\\
	\CO(\sqrt{t})
	\paren[\Big]{
	  \frac{1}{2} \brak{Q : \grad_x^2} 
	  + \sum_{i=1}^M \bar q_i D_i^y 
	} f(x, O)
	  & y = O.
      \end{dcases}
  \end{equation*}
  After dividing by~$t$ and taking the limit as $t \to 0$, this vanishes without any further restriction on $f$ if $y \neq O$.
  For $y = O$, this limit only exists provided that the compatibility condition
  \begin{equation}\label{eqnXiCompat}
    \frac{1}{2} \brak{Q : \grad_x^2}f(x, O)
      + \sum_{i=1}^M \bar q_i D_i^y f(x, O) = 0
  \end{equation}
  holds.
  This shows that if $f \in \CC^2_b(\Gb) \cap D(\mathcal A_\Xi)$, then for every $x \in \R^2$ we must have $f( x, \cdot ) \in D(\mathcal A_y)$, $\mathcal A_\Xi f(x,y) = \mathcal A_y f(x,y)$ for every $y\in \cG\setminus O$, and the compatibility condition~\eqref{eqnXiCompat} must be satisfied as well.

  From this, equation~\eqref{eqnThetaSys} follows from standard techniques.
  Indeed, for~$\theta$ defined by~\eqref{eqnThetaDef}, standard results imply that~$\theta$ is continuous at $t = 0$ and satisfies the Kolmogorov equation $\partial_t \theta - \mathcal A_\Xi \theta = 0$ giving~\eqref{eqnThetaY} and~\eqref{eqnThetaID}.
  Moreover, for positive time we must have~$\theta(\cdot, t) \in \CC^2_b(\Gb) \cap D(\mathcal A_\Xi)$, and~\eqref{eqnXiCompat} gives the flux balance condition~\eqref{eqnThetaX} as desired. 
\end{proof}

\section{Proof of Theorem~\ref{thmMainR2}}\label{sxnMainProof}

We devote this section to proving Theorem~\ref{thmMainR2}.
Our proof resembles the proof in~\cite{HairerKoralovPajorGyulai2014}, where a similar result appeared.
The main difference in our situation is that we rely on Theorem~\ref{thmAvgPrinShortTime} instead of the classical averaging principle.
Our first task is to describe how far $Z(t)$ can travel inside a small neighbourhood of the separatrix.
Given $\delta > 0$, define $\cV^{\delta} \subset \cG$ by
\begin{equation*}
  \cV^{\delta} \defeq \set{(i,y)\in \cG : |y| \leq \delta}\;,
\end{equation*}
and introduce two sequences of stopping times~$\mu^{\epsilon,\delta}_n$ and~$\kappa^{\epsilon,\delta}_n$ corresponding to successive visits to $O$ and $\partial \cV^{\delta}$.
Namely, let $\mu_0^{\epsilon,\delta} = \kappa_{-1}^{\epsilon,\delta}=0$
and then define recursively 
\begin{equation}\label{eqnMuDef}
  \mu_n^{\epsilon,\delta}  = \inf \{t \geq \kappa_{n-1}^{\epsilon,\delta}\,:\, \Gammaeps(Z_t) \in \partial \cV^{\delta} \}
  \qquad\text{and}\qquad
  \kappa_n^{\epsilon,\delta}  = \inf \{t \geq \mu_{n}^{\epsilon,\delta}\,:\, \Gammaeps(Z_t) \in O \}\;,
\end{equation}
for $n\geq 1$ and $n\geq 0$ respectively.
Let $\DeltaZ_n  = Z(\kappa^{\epsilon,\delta}_n) - Z(\kappa^{\epsilon,\delta}_{n-1})$,
be the displacement between successive visits to $\CL$.
With this notation, the distance covered by $Z(t)$ before  hitting 
$\Gamma_{\e}^{-1}(\partial \cV^{\delta})$, as well as 
the cell it then hits can be described as follows.

\begin{proposition}\label{ppnCLTFirstHitShort}
There exists a $2\times 2$ non-degenerate matrix $Q$ and a vector $(p_1,\ldots,p_M)$ such that
the distributions of $
  \paren{
    \eps^{\frac{1-\alpha}{4}} \DeltaZ_1, \Gammaeps(Z(\mu_1^{\epsilon,\delta}))
  }
$
under $\Pe^x$ converge, as $\epsilon \to 0$, to the distribution of 
$\paren{ \sqrt{\delta\xi} \CN(0,Q), \zeta }$,
uniformly for $x \in \CL$.
Here, $\xi$, $\zeta$, and $\mathcal N(0, Q)$ are three independent random variables such that
  $\xi$ is exponentially distributed with parameter one, 
  $\CN(0,Q)$ is a two-dimensional normally distributed random variable with mean $0$ and covariance matrix $Q$,
  and $\zeta$ is a $\cG$-valued random variable that is almost surely at distance $\delta$ from $O$ and $\P(\zeta \in I_i) = p_i$.

Moreover, for each $\eta > 0$ there is $\delta_0 > 0$ such that
\begin{equation} \label{extraeq}
\lim_{\eps \downarrow 0} \sup_{x \in \fd} \Pe^{x} \Big(\eps^{\frac{1-\alpha}{4}} \sup_{0 \leq t \leq \kappa^{\epsilon,\delta}_1 } |Z_{t}| > \eta\Big) < \eta,
\end{equation}
whenever $0 < \delta \leq \delta_0$.
\end{proposition}

A similar result was proved in Section 2 of~\cite{HairerKoralovPajorGyulai2014}.
However, in order to make this paper self-contained, we sketch the main steps involved in the proof and explain the necessary modifications in Appendix~\ref{sxnCLTFirstHitShort}.
Although we will not use it explicitly, we remark that the $p_i$'s above are proportional to the $q_i$.
This follows from the proof of Proposition~\ref{ppnCLTFirstHitShort}, and Corollary 2.4 in \cite{freidlin2000sheu}.

Let now $\CX_\cG$ denote the space of $\cG$-valued excursions. In other words, elements $h \in \CX_\cG$ 
are continuous functions $h \in \CC(\R_+, \cG)$ with the property that, if $h(t) = O$ for some $t\ge 0$, then $h(s) = O$ for
all $s \ge t$. Furthermore, we impose that $T(h) = \inf\{t \ge 0\,:\, h(t) = O\}$ is finite for every $h \in \CX_\cG$.
We turn $\CX_\cG$ into a metric space by setting
\begin{equation}\label{eq:exc_metric}
d(h, \bar h) = |T(h) - T({\bar h})| + \sup_{t \ge 0} d_\cG(h(t), \bar h(t))\;,
\end{equation}
with $d_\cG$ as in \eqref{e:defdG}.

We also write $\CX^\infty = (\R^2 \times \CX_\cG)^\N$, endowed with the topology of pointwise convergence, and we define a ``projection'' 
$\CP_\delta \colon \CC(\R_+; \Gb) \to \CX^\infty$
as follows.
Given an element $\omega \in \CC(\R_+; \Gb)$, we write $\omega = (V,G)$ where $V$ and $G$ are continuous $\R^2$-valued and $\cG$-valued functions respectively.
We first define the ``stopping times'' $\mu_n^\delta(\omega)$ and $\kappa_n^\delta(\omega)$ as in \eqref{eqnMuDef}, with $\Gammaeps(Z)$ replaced by $G$.
We then write $J_n(\omega) \in \CX_\cG$ for the $n^\text{th}$ downcrossing of the
process $G$. In other words, suppressing the argument $\omega$ for conciseness, we have
\begin{equ}
J_n(t) = \big(G((t+\mu_n^\delta) \wedge \kappa_n^\delta)\big)\;,
\end{equ}
so that in particular $|J_n(\omega)(0)| = \delta$ for $n > 0$ and $T(J_n(\omega)) = \kappa_n^\delta - \mu_n^\delta$.
We also define $U_n(\omega) \in \R^2$ by
$U_n(\omega) = V(\mu_{n+1}^\delta) - V(\kappa_{n}^\delta)$.
With these notations at hand, we set
\begin{equ}
\CP_\delta(\omega) = \bigl(U_n(\omega), J_n(\omega) \bigr)_{n \ge 0}\;.
\end{equ}
We then have the following lemma.

\begin{lemma} \label{cor:indlemma}
Let $\hat\nu_{\e}$ be a family of probability measures on $\fd$ such that the push forward measures~$\Gammaeps^* \hat\nu_{\e}$ converge weakly, as $\epsilon\to0$, to a probability measure $\hat\nu$ on $\cG$.
Then, the law of $\CP_\delta(\hge(Z))$
converges weakly under $\Pe^{\hat\nu_{\e}}$, as $\eps \downarrow 0$, to 
the law of $\CP_\delta(\Xi)$ under $\Pb^{\hat\nu}$.
\end{lemma}

\begin{proof}
We first note that, under $\Pb^{\hat\nu}$, $\CP_\delta(\Xi)$ is a random vector $(U_n, J_n)_{n \ge 0}$ with independent components.
The distribution of $U_n$ is as in Proposition~\ref{ppnCLTFirstHitShort}, i.e., it is equal to the distribution of $\sqrt{\delta\xi} \CN(0,Q)$.
This follows from the fact that the distribution of the local time accumulated up to $\mu_1^{\delta}$ under $\Pb^{O}$ is the same as that of $\delta\xi$. Indeed, exponentiality follows from the fact that $L_t$ can only grow when $Y_t=0$, while the expectation is given by applying \eqref{eq:Ito_in_Freidlin-Sheu} to the function $f(y)=d_{\mathcal{G}}(y,O)$, plugging in $t=\mu_1^{\delta}$, and taking expectations (see e.g. Exercise 4.12 Chapter VI in \cite{revuz1999continuous}).
The distribution of $J_0$ is the distribution of $J_0(Y)$ under 
$\Pb^\mu$, while the distribution for each of the $J_n$ for $n>0$ is equal to the 
distribution of $J_0(Y)$ under $\Pb^\zeta$, where $\zeta$ is as in 
Proposition~\ref{ppnCLTFirstHitShort}. The fact that these are independent follows from the 
strong Markov property, combined with the fact that the location at which the $Y$-component 
of the process first hits $\cV^\delta$ is independent of the local time accumulated until then.

We then see that, by Theorem~\ref{thmAvgPrinShortTime}, the law of $J_0(\hge(Z))$ 
under $\Pe^{\hat\nu_{\e}}$ 
does indeed converge weakly as $\eps \to 0$ to the law of $J_0(Y)$ under $\Pb^{\hat\mu}$. 
This is because, although the map 
$Y \mapsto J_0(Y)$ is not continuous, its points of discontinuity, which consist precisely 
of those paths which either never hit $O$
or such that their first hit of $O$ is not transverse, are of measure $0$ under $\Pb^{\hat\nu}$.

The convergence of the other components of the random vector $(U_n, J_n)_{n \ge 0}$
to their respective limits follows in the same way
from Proposition~\ref{ppnCLTFirstHitShort}, combined with Theorem~\ref{thmAvgPrinShortTime}. 
The independence of the components of the limiting vector immediately follows from the 
strong Markov property of the process $Z$, the fact that the convergences in  
Proposition~\ref{ppnCLTFirstHitShort} and Theorem~\ref{thmAvgPrinShortTime} are 
uniform with respect to the initial condition, and the fact that $\zeta$
is independent of the other limiting random variables in Proposition~\ref{ppnCLTFirstHitShort}. 
\end{proof}

\begin{proof}[Proof of Theorem~\ref{thmMainR2}]
  We first note that as a consequence of the periodicity of the problem, we can
(and will henceforth) restrict ourselves to the case when the probability measure
$\nu_\eps$ is concentrated on $\fd$, so that the limiting measure $\hat \nu$ is of the form
$\hat \nu = \delta_0 \otimes \nu$ for some probability measure $\nu$ on $\cG$.
We thus only need to prove that, under the conditions of the theorem, $\hge(Z)$ converges in law to $\Xi$ with initial measure $\delta_0\otimes \nu$.
We begin by defining a ``concatenation''  map $\Rde \colon \CX^\infty \to \CC(\R_+; \Gb )$ as follows.
Given $\delta > 0$, $U,V \in \R^2$ and $G = (i,y) \in \cG$, we define the interpolation
$L_\delta(U,V,G) \colon [0,\delta^2] \to \Gb$ by
\begin{equ}
L_\delta(U,V,G)(t) \defeq \paren[\big]{U + \delta^{-2} t(V-U), (i,  \delta^{-2} ty)}\;,
\end{equ}
so that $L_\delta(U,V,G)(0) = (U,O)$ and $L_\delta(U,V,G)(\delta^2) = (V,G)$.
Given $\X = (\X_n)_{n \ge 0}$ with $\X_n = (U_n, J_n) \in \R^2 \times \CX_\cG$, we define recursively two sequences of ``excursion times'' 
$E_n, E_n' \in \R_+$ and locations $W_n \in \R^2$ by
\begin{equation*}
    E_n(\X) \defeq n\delta^2 + \sum_{j=0}^{n-1} T(J_j)\;,
    \qquad E_n'(\X) = E_n(\X) + T(J_n)\;,
    \qquad W_n = \sum_{j=0}^{n-1} U_j\;.
\end{equation*}
with the natural conventions that $E_0 = 0$ and $W_0 = 0$. With these notations at hand, we then set
\begin{equ}
\Rde(\X)(t) = 
\left\{\begin{array}{cl}
	(W_n, J_n(t-E_n(\X))) & \text{for $t \in [E_n(\X), E_n'(\X)]$,} \\
	L_\delta(W_n,W_{n+1},J_{n+1}(0))(t-E_n'(\X)) & \text{for $t \in [E_n'(\X), E_n'(\X)+\delta^2]$.}
\end{array}\right.
\end{equ}
This definition is unambiguous
(and the function $\Rde(\X)$ is continuous) since, at $t = E_n'(\X)$, both 
expressions equal $(W_n,O)$, while at $t = E_{n+1}(\X) = E_n'(\X) + \delta^2$ both expressions
equal $(W_{n+1}, J_{n+1}(0))$.

It is straightforward to see that $\Rde$ is a right inverse for $\CP_\delta$, i.e.\
$\CP_\delta \Rde = \id$.
On the other hand, clearly $\Rde \CP_\delta \neq \id$, however, we will construct a set of trajectories $\omega$ on which $(\Rde \CP_\delta)(\omega)$ is close to $\omega$.
For this, we need a bit of additional notation.
Given a trajectory $\omega \in \CC(\R_+; \Gb)$ and times $\mu_n^\delta(\omega)$
and $\kappa_n^\delta(\omega)$ as above,
we define the corresponding downcrossing and upcrossing durations
by
\[
T_{n,\delta}^d=\kappa_n^\delta-\mu_n^\delta,\qquad T_{n,\delta}^u=\mu_n^{\delta}-\kappa_{n-1}^{\delta},\qquad n \ge 0\;.
\]
We also define the number of down / upcrossings up to time $t$ by
\begin{equ}
D_t^{\delta}=\inf\{n\geq 0: \kappa_n^{\delta} \geq t\}\;,\qquad
N_t^{\delta}=\sup\{n\geq 0: \mu_n^{\delta} \leq t\}\;,
\end{equ}
see Figure~\ref{fig:times}, as well as 
the quantities
\begin{equ}
\edelta(t) = t + \sum_{n=1}^{\bar N_t^\delta} T_{n,\delta}^u\;,\quad
\edeltab(t) = t + \delta^2 \bar N_t^\delta\;,
\qquad   \bar N_t^\delta = \sup \Bigl\{k\ge 0\,:\, \mu_k^\delta \le t + 
\sum_{n=1}^{k} T_{n,\delta}^u\Bigr\}\;.
\end{equ}

\begin{figure}
    \centering
  \begin{minipage}{.8\linewidth}\centering
\begin{tikzpicture}[scale=1.5]
\draw[->] (0,0) -- (7,0);
\node at (7.2,0) {$t$};

\def\xxx{{0.5,1.2,2,2.8,3.4,5.5,6.2}}
\def\double{{0,0,1,1,2,2,3,3,4,4}}

\foreach \xm in {0,1,2} {
\pgfmathtruncatemacro{\x}{\xm + 1}
\pgfmathsetmacro{\taux}{\xxx[2*\xm+1]}
\pgfmathsetmacro{\thetax}{\xxx[2*\xm+2]}
	\fill[darkred!10] (\taux,0) rectangle (\thetax,1.5);
	\draw[line width=3pt,draw=darkred] (\taux,0) -- (\thetax,0);
	\draw[dashed] (\taux,0) -- (\taux,1.5);
	\draw[dashed] (\thetax,0) -- (\thetax,1.5);
	\node at (\taux,-0.3) {$\kappa_\xm^\delta$};
	\node at (\thetax,-0.3) {$\mu_\x^\delta$};
}

\foreach \xm in {0,...,5} {
\pgfmathsetmacro{\xx}{0.5*(\xxx[\xm] + \xxx[\xm+1])}
\pgfmathsetmacro{\NN}{\double[\xm]}
\pgfmathsetmacro{\DD}{\double[\xm+1]}
	\node at (\xx,0.3) {$\NN$};
	\node at (\xx,0.9) {$\DD$};
}

\draw (0.5,0) -- (0.5,1.5);

\draw[dotted] (0,0.6) -- (7,0.6);
\draw[dotted] (0,1.2) -- (7,1.2);
\node at (0,0.3) {$N_t^\delta$};
\node at (0,0.9) {$D_t^\delta$};
\end{tikzpicture}
\caption{Values of $D_t^\delta$ and $N_t^\delta$ in relation to the stopping times $\kappa_i$ and $\mu_i$.
The highlighted regions are when the upcrossings from $O$ to $\cV^\delta$ occur. The range of $\edelta$ falls into the non-highlighted regions, while the range of $\edeltab$ falls into the
region obtained by shrinking / expanding the highlighted regions in such a way that each
of them has length $\delta^2$.}
\label{fig:times}
  \end{minipage}
\end{figure}
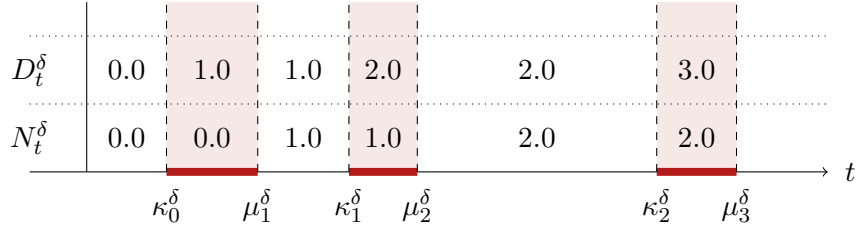

These quantities can be interpreted as follows: we stop a `special' clock every time the process hits the vertex $O$, 
and re-start it once the process reaches the level set $\partial \cV^\delta$.
Then $\edelta(t)$ is the real time that has elapsed when the special clock reaches time $t$, $\bar N_t^\delta$ is the number of upcrossings completed before this happens, and $\edeltab(t)$ is the analogous quantity to $\edelta(t)$ when, for every upcrossing, we count $\delta^2$ in real time.

Given $\eta, r, T, \delta > 0$, we then define a set $\CF(\eta,r,T,\delta)$ of trajectories
$\omega$ such that the following properties hold:
\begin{enumerate}
\item For all $s,t \le T+r$ with $|t-s| \le r$, one has $d_\Gb(\omega(t),\omega(s)) \le \eta$.
\item One has the bounds $\edelta(T)-T \le r$ 
and $\delta^2 (2+\bar N_T^\delta(\omega)) \le \eta$.
\item Writing $\omega = (V,G)$ as above, 
for every $n \in \{0,\ldots,\bar N_T^\delta(\omega)\}$, one has the bounds
\begin{equation*}
\sup_{t \in [\mu_n^\delta, \kappa_n^\delta]} |V(t) - V(\mu_n^\delta)| \le \delta^2,
\qquad\text{and}\qquad
|V(0)| \le \delta^2 \;.
\end{equation*}
\end{enumerate}

The following lemmas are the crucial ingredients for our proof of the theorem.

\begin{lemma}\label{lmaGoodSet}
Provided that $r>\delta^2$, every
$\omega \in \CF(\eta,r,T,\delta)$ satisfies the bound
\begin{equ}[e:approxDelta]
\sup_{t \in [0,T]} d_\Gb\bigl(\omega(t), \bigl(\Rde\CP_\delta\omega\bigr)(t)\bigr) \le 4 \eta\;.
\end{equ}
\end{lemma}

\begin{lemma}\label{lmaGoodBound}
For every $\eta > 0$, every $T>0$, and every sequence $\nu_\eps$ of probability measures on
$\R^2$ such that $\Gamma_{\e}^{*}\nu_{\e}$ is tight, 
there exist $\delta_0 > 0$, $r \geq \delta_0^2$ such that for $\delta\in(0,\delta_0)$, there is $\eps_0 > 0$ such that
\begin{equ}
\Pe^{\nu_\eps} \bigl(\hge(Z) \in \CF(\eta,r,T,\delta)\bigr) \ge 1-\eta\;,
\end{equ}
for every $\eps \le \eps_0$.
\end{lemma}
\begin{lemma}\label{lmaUpcrossingNegligible}
We have
\begin{equ}
\sup_{y \in \cG} \Eb^y|\edelta(t)- t| =\CO(\delta) \;
\qquad\text{and}\qquad
\sup_{y \in \cG} \Eb^y|\delta(\bar N^\delta_t -D_t^{\delta})| \xrightarrow{\delta \to 0} 0 \;.
\end{equ}
\end{lemma}

\begin{lemma}[\cite{freidlin2000sheu}]\label{lmaDowncrossing}
For every~$t > 0$, we have
$\lim_{\delta \downarrow 0} \sup_{y\in \cG} \Eb^y \abs{ \delta D_t^{\delta} - L_t }  = 0$.
\end{lemma}

Lemma~\ref{lmaDowncrossing} is contained in in the proof of Lemma~2.2 in~\cite{freidlin2000sheu}, and we do not prove it here.
For clarity of presentation the proofs of Lemmas~\ref{lmaGoodSet}--\ref{lmaUpcrossingNegligible} are postponed until the proof of Theorem~\ref{thmMainR2} is complete.

The rest of the proof
is a standard ``triangle'' argument.
Fix $T>0$ and let $f$ be a uniformly continuous bounded functional on $\CC([0,T]; \Gb )$.
Pick any $\eta'>0$ and choose $\eta>0$ small enough such that 
\[
|f(\omega)-f(\omega')|<\eta'\qquad\textrm{whenever}\qquad\sup_{t\in[0,T]}d_{\Gb}(\omega(t),\omega'(t))< 4\eta.
\]

Note that the reconstruction map $\mathcal{R}_{\delta}: \CX^{\infty}\to \CC(\R_+,\Gb)$ is continuous with the choice of the metric \eqref{eq:exc_metric}. Since the restriction operator $\Pi_{[0,T]}: \CC(\R_+,\Gb)\to \CC([0,T],\Gb)$ is also continuous, Lemma \ref{cor:indlemma} implies that
\begin{equation}\label{eq:tri1}
\Ee^{\nu_\eps}f(\Pi_{[0,T]}\mathcal{R}_{\delta}\mathcal{P}_{\delta}(\hge(Z)))\;\xrightarrow{\epsilon \to 0}\; \Eb^{\hat\nu}f\left(\Pi_{[0,T]}\mathcal{R}_{\delta}\mathcal{P}_{\delta}(\Xi)\right)\;.
\end{equation}
On the other hand, Lemma \ref{lmaGoodSet} and Lemma \ref{lmaGoodBound} imply that we can find a $\delta>0$ such that for any small enough $\e>0$, we have
\begin{equation}\label{eq:tri2}
|\Ee^{\nu_\eps}f(\Pi_{[0,T]}\hge(Z))-\Ee^{\nu_\eps}f(\Pi_{[0,T]}\mathcal{R}_{\delta}\mathcal{P}_{\delta}(\hge(Z))|\leq 2\eta\norm{f}_{\infty}+\eta'.
\end{equation}

For the limiting process $\Xi$, standard results on the Brownian modulus of continuity, and Lemmas \ref{lmaUpcrossingNegligible}--\ref{lmaDowncrossing} imply that, by possibly making $\delta$ smaller, we have
\begin{equ}
\Pb^{\hat\nu} \bigl(\Xi\in \CF(\eta,r,T,\delta)\bigr) \ge 1-\eta\;.
\end{equ}
Together with Lemma \ref{lmaGoodSet}, this implies
\begin{equation}\label{eq:tri3}
|\Eb^{\hat\nu}f(\Pi_{[0,T]}\Xi)-\Eb^{\hat\nu}f(\Pi_{[0,T]}\mathcal{R}_{\delta}\mathcal{P}_{\delta}(\Xi)|\leq 2\eta\norm{f}_{\infty}+\eta'.
\end{equation}
Combining \eqref{eq:tri1}, \eqref{eq:tri2}, \eqref{eq:tri3} and noting that $\eta$ and $\eta'$ can be made arbitrarily small gives the convergence of 
$\Pi_{[0,T]}\hge(Z)$ in law to $\Pi_{[0,T]}\Xi$. Since $T>0$ was also arbitrary, this finishes the proof.
\end{proof}


It remains to prove Lemmas~\ref{lmaGoodSet}--\ref{lmaUpcrossingNegligible}.

\begin{proof}[Proof of Lemma~\ref{lmaGoodSet}]
For fixed small $\lambda>0$ and any $T>0$,  consider $\omega = (V,G) \in \CF(\eta,r,T,\delta)$. We
want to show that if $\delta$ and $\eta$ are sufficiently small, then,
writing $\omega^\delta = (V^\delta,G^\delta) := \Rde\CP_\delta\omega$, one
has $d_\Gb(\omega^\delta(t),\omega(t)) \le \lambda$ for $t \le T$.

For this, we first build the time change
\begin{equ}
f_\delta(t) = \edelta \bigl(\inf \{s\,:\, \edeltab(s) \ge t\}\bigr)\;.
\end{equ}
(On the range of $\edeltab$, this equals $\edelta(\edeltab^{-1}(t))$.)
The second property of $\CF$ then guarantees that  
\begin{equ}[e:propf]
|f_\delta(t) - t| \le r\;,\qquad \forall t \le T\;.
\end{equ}
It also follows from the constructions of $\Rde$ and $\CP_\delta$ that, for all $t$
in the range of $\edeltab$, one has
\begin{equ}[e:firstbound]
G^\delta(t) = G(f_\delta(t))\;,\qquad
V^\delta(t) = \sum_{n=1}^{\bar N_t^\delta} (V(\mu_{n}^\delta) - V(\kappa_{n-1}^\delta))\;. 
\end{equ}
We can rewrite the second identity as
\begin{equ}
V^\delta(t) = V(f_\delta(t)) - V(0) - \big(V(f_\delta(t)) - V(\mu_{\bar N_t^\delta}^\delta)\big) - \sum_{n=1}^{\bar N_t^\delta} \bigl(V(\kappa_{n-1}^\delta) - V(\mu_{n-1}^\delta)\bigr)\;.
\end{equ}
Since we have $f_\delta(t) \in [\mu_n^\delta, \kappa_n^\delta]$
for $n = \bar N_t^\delta$ by definition, we can combine this with the third
property of $\CF$, thus yielding the bound $|V^\delta(t) - V(f_\delta(t))|
\le (2+ \bar N_t^\delta) \delta^2 \le \eta$.
Together with the first equality in \eqref{e:firstbound} and the first property of $\CF$, 
this finally yields
\begin{equ}
d_{\Gb}(\omega(t), \omega^\delta(t)) \leq  d_{\Gb}(\omega(t), \omega( f_\delta(t)))  +|V^\delta(t) - V(f_\delta(t))| \le 2 \eta,
\end{equ}
for all times $t \le T$ belonging to the range of $\edeltab$.
It remains to consider times outside the range of $\edeltab$, which correspond
to the upcrossings. Write $t_0 < t$ for the start of the upcrossing, so that $|t-t_0| \le \delta^2 < r$
by definition. Then, one has
\begin{equ}
d_{\Gb}(\omega(t), \omega^\delta(t))
\le 
d_{\Gb}(\omega(t), \omega(t_0))
+ d_{\Gb}(\omega(t_0), \omega^\delta(t_0))
+ d_{\Gb}(\omega^\delta(t), \omega^\delta(t_0))
\le 4 \eta\;.
\end{equ}
This is because the first and last terms are bounded by $\eta$ as a consequence of the 
first and second properties of $\CF$, while the second term is bounded by $2 \eta$ from before. 
\end{proof}

Since Lemma~\ref{lmaUpcrossingNegligible} is used in the proof of Lemma~\ref{lmaGoodBound}, we prove it first.
\begin{proof}[Proof of Lemma~\ref{lmaUpcrossingNegligible}]
By the strong Markov property, the $T_{i,\delta}^u$ are independent and identically 
distributed under $\Pb^y$, while the $T_{i,\delta}^d$ are identically distributed, but not independent of the upcrossing durations in between.
However, when conditioned on the corresponding downcrossing taking place on edge $j$, 
they have the same distribution as the hitting time of the point $\delta/\sqrt{a_j}$ by a standard 
Brownian motion starting at the origin.

We note that for any $\lambda>0$ and $K \in \N$, Chebyshev's inequality implies
\[
\Pb^{y}\big(\bar D^{\delta}_{t}>K\big)=\Pb^y \Big(\sum_{i=0}^KT_{i,\delta}^d<t \Big)\leq e^{\lambda t}\Eb^ye^{-\lambda \sum_{i=0}^KT_{i,\delta}^d}=e^{\lambda t}\prod_{i=0}^K\Eb^ye^{-\lambda T_{i,\delta}^d},
\]
where we defined
$\bar D^{\delta}_{t}= D^{\delta}_{e_{\delta}(t)}$.
Let $j$ be the index of the slowest edge, that is $a_j=\min_{i=1,\dots,n} a_i$. Then, by the strong Markov property,
\[
\Eb^ye^{-\lambda T_{i,\delta}^d}\leq\Eb^{(j,\delta)} e^{-\lambda \tau_{0}^\delta}=e^{-\delta b_\lambda}\;,
\qquad b_\lambda = \sqrt{\frac{2\lambda}{a_j}}\;.
\]
Inserting this into the above yields 
\begin{equ}[eq:mean_D]
\Pb^{y}\big(\bar D^{\delta}_{t}>K\big)\leq \exp\big({\lambda t-b_\lambda K\delta}\big)\;.
\end{equ}
Next, writing
$e_{\delta}(t)-t=\sum_{i=1}^{\infty}\one_{i \le \bar D_{t}^{\delta}} T_{i,\delta}^u$,
we obtain from the Cauchy-Schwartz inequality,
\begin{equation}\label{eq:afterWald}
\Eb^y|e_{\delta}(t)-t|=\Eb^y(e_{\delta}(t)-t)=
\sum_{i = 1}^\infty \sqrt{\P(i\le \bar D^\delta_t)\Eb^y (T_{i,\delta}^u)^2}
\le 
C e^{\lambda t/2} \delta^2 \sum_{i = 1}^\infty e^{-\delta b_\lambda i/2}
\le 
C(t) \delta \;,
\end{equation}
where we used \eqref{eq:mean_D}, combined with the fact that $\Eb^y(T_{i,\delta}^u)^2=\CO(\delta^4)$ by the Brownian scaling.

To prove the second claim,  pick an $\eta>0$. By the monotonicity of $D_t^{\delta}$ in $t$ and since $| \bar N^{\delta}_{t}- \bar D^{\delta}_{t}| \leq 1$, 
\begin{equation}\label{eq:decompNminD}
\Eb^y|\delta(\bar N_t^{\delta}-D_t^{\delta})|\leq \delta+ \delta\Eb^y (D^{\delta}_{t+\eta}-D^{\delta}_t)
+\delta\Eb^y\big( \bar D^{\delta}_{t}\one_{\{e_{\delta}(t)-t>\eta\}}\big).
\end{equation}

The expectation 
$\delta \Eb^y (D^{\delta}_{t+\eta}-D^{\delta}_t)$ above can be estimated by comparing it to the local time.
Using the Markov property at time $t$ and Lemma~\ref{lmaDowncrossing}, we have
\[
\delta \bar{\E}^y (D^{\delta}_{t+\eta}-D^{\delta}_t)\leq\bar{\E}^y\bar{\E}^{Y(t)}L_{\eta}+o(1).
\]
The right hand side of the above can be made arbitrarily small by choosing $\eta$ small enough.

On the other hand, the Cauchy-Schwartz inequality implies that the 
last term on the right hand side of \eqref{eq:decompNminD} can be estimated from above by
\[
\delta  \sqrt{ \bar{\E}^y (\bar D^{\delta}_{t})^2\,\Pb^y(e_{\delta}(t)-t>\eta)}=o(1).
\]
Indeed, the probability converges to zero by \eqref{eq:afterWald} and Chebyshev's inequality, while the 
remaining factor can be bounded using \eqref{eq:mean_D}, thus concluding the proof.
\end{proof}

Finally, we turn to Lemma~\ref{lmaGoodBound}.
The proof relies on tightness (stated as Lemma~\ref{lmaTightness}, below) and the fact that upcrossing durations are negligible compared to the downcrossings durations (Lemma~\ref{lmaUpcrossingNegligible}).

\begin{lemma}\label{lmaTightness}
The law of $\hge(Z)$ under $\P^{\nu_\eps}$ is tight in $\CC(\R_+, \Gb)$. In particular, for every $T,\eta>0$, there is an $r>0$ and $\e_0>0$ such that for $\e\in (0,\e_0]$, we have
\[
\P^{\nu_\eps} \paren[\Big]{ \sup_{|t-s|\leq r\atop s,t\in[0,T]}d_{\Gb}(\hge(Z(t)),\hge(Z(s)))\geq \eta} < \eta
\]
\end{lemma}
\begin{proof}
The tightness of the $\mathcal{G}$-component follows from Theorem~\ref{thmAvgPrinShortTime}, we only have to prove the tightness of the $\R^2$ component.
Using the strong Markov property and the fact that the displacement is bounded by~$\cO(\epsilon)$ as long as the process remains inside a cell, tightness of the~$\R^2$ component reduces to showing the following:
for every~$\eta > 0$ and $r \in (0, 1)$ sufficiently small, there exists $\epsilon_0 = \epsilon_0(\eta, r) > 0$ such that
\begin{equation} \label{teq1}
\Pe^x \Big( \sup_{0 \leq t \leq r}  | \eps^{\frac{1-\alpha}{4}}Z_t| > \eta \Big) \leq r \eta\;,
\end{equation}
for every $\eps \leq \eps_0$ and $x \in \CL \cap \fd$ (see also Theorem 18.17 in \cite{koralovsinai2007}).
We prove this below.

Let $R^\delta_1$, $R^\delta_2$, etc.\ be independent, distributed as $\sqrt{\delta \xi} \CN(0, Q)$.
Doob's maximal inequality then shows that there exists a constant $C$ such that
\[
\P\Big(l^{-1/2}\max_{1\leq m\leq l}|R_1^\delta +\cdots+R_m^\delta|> K\Big)
  \le \frac{C \delta^5}{K^{10}} \E \abs{R_1^\delta}^{10} \;,
\]
for all $K > 0$.
Choosing $K = \frac{1}{4} \eta  \sqrt{\delta/k}$, we see that for a given $\eta>0$, there exist $k_0 \in (0,1)$ and $\delta_1>0$ such that
\begin{equation} \label{tt01}
\P \Big(\max_{1 \leq m \leq k/\delta} |R^\delta_1+\cdots+R^\delta_m| > {\eta\over 4}\Big) \leq {k^4 \eta\over 4}\;,
\end{equation}
whenever $k\in (0,k_0)$ and $\delta \in (0,\delta_1)$. From \eqref{tt01} and Lemma~\ref{cor:indlemma}, it follows that there is $\eps_1(k,\delta) > 0$ such that
\begin{equation}\label{eq:tt02}
\Pe^x \Big(\max_{1 \leq m \leq k/\delta} \eps^{(1-\alpha)/4}|\DeltaZ_1+\cdots+\DeltaZ_m| > {\eta \over 3}\Big) 
\leq {k^4 \eta\over 3}\;,
\end{equation}
provided that $\eps \leq \eps_1(k,\delta)$.
Note that this estimate and those below are uniform in $x \in \CL \cap \fd$.
Combining  \eqref{eq:tt02} and \eqref{extraeq}, it now follows that there is $\eps_2(k,\delta) > 0$ such that
\begin{equation} \label{fr1}
\Pe^x \Big(\sup_{0 \leq t \leq \kappa^{\delta ,\eps}_{[k/\delta]}} \eps^{(1-\alpha)/4}|Z_t| > \eta/2\Big) \leq 
{k^4  \eta \over 2}\;.
\end{equation}
provided that $\eps \leq \eps_2(k,\delta)$.

By Lemma~\ref{lmaDowncrossing}, for a given $\eta > 0$,  we can find $r > 0$ and $\delta_2 = \delta_2(r) > 0$ such that, for any
$\delta \leq \delta_2$, and $\ell = [{r^{1/4}/ \delta}]$ we have
\begin{equation} \label{tt02}
\sup_{y \in \cG} \Pb^y(D_{r}^{\delta} \geq \ell) <\sup_{y\in \cG}\Pb^y(L_r\geq r^{1/4})+{\eta r\over 4}\leq r^2\Eb^O\Big({L_r\over r^{1/2}}\Big)^8+{\eta r\over 4} \leq {\eta r\over 3}.
\end{equation}
Here the second inequality follows from the Chebyshev inequality and the strong Markov property, while the last
inequality follows from the fact that the distribution of 
$L_r/r^{1/2}$ under $\Pb^O$ does not depend on $r > 0$ and has Gaussian tails. As a consequence of Theorem~\ref{thmAvgPrinShortTime}, there is $\eps_3(r, \delta)$ such that
if $\eps \leq \eps_3(r, \delta)$, and $x \in \CL$ we have
\begin{equation} \label{fr2}
\Pe^x\big( \kappa^{\delta,\eps}_{\ell} < r \big) \leq  \Pb^O\big(D_r^{\delta}\geq \ell\big)+{\eta r\over 6},
\end{equation}
and hence
\[
\Pe^x \Big( \sup_{0 \leq t \leq r}   \eps^{(1-\alpha)/4} | Z_t | > \eta \Big)\leq\Pe^x\big(\kappa_{\ell}^{\delta,\eps}<r\big)+\Pe^x\Big(\sup_{0\leq t\leq\kappa_{\ell}^{\delta,\eps}}\eps^{(1-\alpha)/4}|Z_t|>\eta\Big).
\]

Applying \eqref{fr1},~\eqref{tt02}, and~\eqref{fr2} with
\begin{equation*}
k = r^{1/4},
\qquad
\delta < \min(\delta_1, \delta_2)
\qquad\text{and}\qquad
\eps<\min(\eps_1(k, \delta),\eps_2(k,\delta), \eps_3(r,\delta)),
\end{equation*}
we obtain \eqref{teq1} as required.
\end{proof}

Finally we prove Lemma~\ref{lmaGoodBound}.

\begin{proof}[Proof of Lemma~\ref{lmaGoodBound}]
Fix $\eta > 0$ and $T > 0$. As a consequence of Lemma~\ref{lmaTightness},
we can find $\eps_0$ and $r>0$ such that the first property is satisfied with probability
at least $1-\eta$, uniformly over $\eps < \eps_0$.
By Lemma~\ref{lmaUpcrossingNegligible}, we then choose $\delta$ with $\delta^2 < r$ sufficiently small
so that the second estimate holds. The third bound immediately follows from the definitions as soon as $\eps^{1-\alpha \over 4} \le \delta^2$,
thus concluding the proof.
\end{proof}

\section{The averaging principle on the short time scales}\label{sxnShortTimeAveraging}

We now turn to the proof of Theorem~\ref{thmAvgPrinShortTime}.
For notational simplicity, we view~$Z$ itself as a process on the torus~$\torus$ and define set $Y^\eps_t=\Gammaeps(Z_t)$.
Let $\Psi \subset \mathcal{C}_0(\cG)$ be the dense subset consisting of all compactly supported functions that are continuously differentiable on each edge. 

The proof of Theorem \ref{thmAvgPrinShortTime} relies on the following two lemmas (compare with the result of Freidlin and Wentzell~\cite[Ch.\ 8, Lemma~3.1]{FreidlinWentzell12}).

\begin{lemma}\label{lei1}
  Let $\mathcal{A}$ be the operator on the domain $D(\mathcal{A})$ introduced in Section \ref{sxnMainResults} and  $\mathcal{D} \subset D( \mathcal{A})$ be the subset consisting of all the functions $f$ for which $ \mathcal{A}f \in \Psi$.
  For each $f\in \mathcal{D}$, $T>0$,  we have
  \begin{equation}\label{eq:marting_problw}
    \sup_{x\in \torus}\left|\Ee^x\left[f(Y^{\e}_T)-f(Y^{\e}_0)-\int_0^T\mathcal{A}f(Y^{\e}_t) \, dt\right]\right|\xrightarrow{\epsilon \to 0} 0.
  \end{equation}
\end{lemma}

\begin{lemma}\label{tins}
  For each compact set $K\subset \cG$, the laws of the processes $\{Y^{\e}\}$ under the measures $\Pe^x$ for ${\e\in(0,1]}$ and $x \in \Gammaeps^{-1}(K) \subset \torus$ are tight.
\end{lemma}

Momentarily postponing the proof of Lemmas~\ref{lei1} and~\ref{tins}, we use them to prove Theorem~\ref{thmAvgPrinShortTime}.
\begin{proof}[Proof of Theorem~\ref{thmAvgPrinShortTime}]
  Lemma~\ref{lei1}, the Markov property, and the time-independence of $\mathcal{A}$ ensure that any subsequential limit of $Y^\e$ solves the martingale problem with the operator $\mathcal{A}'=\mathcal{A}|_{\mathcal{D}}$.
  By the Ito formula, $Y$ is a solution of the martingale problem for $\mathcal{A}'$ and any initial measure $\mu$ on $\cG$.
  Since $\mathcal{D}$ is a core for $\mathcal{A}$ (as one can easily verify using \cite[Ch 1, Proposition 3.1]{EthierKurtz86}), Theorem~4.1 in~\cite[Ch 4]{EthierKurtz86} implies that $Y$ is actually the unique solution to the martingale problem for $\mathcal{A}$ with any initial measure $\mu$ on $\cG$. Moreover, it also follows that the subsequential limits are solutions of the martingale problem for $\mathcal{A}$ as well.
  Therefore any subsequential limit of $Y^{\e}$ must equal $Y$ and Lemma~\ref{tins} and hence Prokhorov's theorem imply the convergence of $Y^\e$ itself.
\end{proof}

The rest of the section is devoted to the proof of Lemma~\ref{lei1}, while the proofs of several auxiliary lemmas and of Lemma~\ref{tins} are relegated to the next section. 
Some elements in our proof are similar to those used in \cite{DolgopyatKoralov08,DolgopyatKoralov12}, 
where a different extension of the original averaging principle (Chapter~8 of \cite{FreidlinWentzell12}) was addressed.

For $h \geq 0$, we will write $\CL(h) =\{x\in \torus : |H(x)|=h\}$, so that $\CL_\torus = \CL(0)$.
Take a function $\beta = \beta(\eps) \in (\alpha/2, 1/2)$ such that
\begin{equation}\label{eq:choosinglevels}
\beta(\eps) - \frac{\alpha}{2} \rightarrow 0~~~\text{and}~~ \eps^{\beta(\eps) - \frac{\alpha}{2}} \rightarrow 0~~~\text{as}~~\e \downarrow 0.
\end{equation}
 Denote
$\Lbar = \CL( \eps^{\beta})$ and
let $\sigma$ be the first time when
the process $Z_t$ reaches~$\CL_\torus$ (this coincides with $\beta^{\eps}_0$ introduced earlier) and $\tau$ be the first time when it
reaches $\Lbar$. We inductively define the following
two sequences of stopping times.  Let $\sigma_1 = \sigma$. For $n \geq
1$ let $\tau_n$ be the first time following $\sigma_n$ when the
process reaches $\Lbar$. For $n \geq 2$ let $\sigma_n$ be the first
time following $\tau_{n-1}$ when the process reaches
$\CL_\torus$.

We can consider the following discrete time Markov chains $\xi^1_n
= Z_{\sigma_n}$ and $\xi^2_n = Z_{\tau_n}$
with state spaces $\CL_\torus$ and $\Lbar$,
respectively. Let $P_1(x, dy)$ and $P_2(x, dy)$ be transition
operators for the Markov chains $\xi^1_n$ and $\xi^2_n$,
respectively. It was then shown in \cite[Lem~2.3]{DolgopyatKoralov08} that they 
are uniformly exponentially mixing in the
following sense.
\begin{lemma} \label{mixing}
There exist constants $0 < c < 1$, $\eps_0 > 0$, $n_0 > 0$,
and probability measures $\nu$ and $\mu$ (which depend on
$\eps$) on $\CL_\torus$ and $\Lbar$, respectively,
such that for $\eps < \eps_0$ and $n \geq n_0$ we
have
\begin{equation} \label{mixingeq}
\sup_{x \in \CL_\torus} \|P_1^n(x, \cdot) - \nu\|_\TV
\leq c^n,~~~~\sup_{x \in \Lbar} \|P_2^n(x, \cdot) -
\mu\|_\TV \leq c^n,
\end{equation}
where $\norm{\cdot}_\TV$ denotes the total variation norm of a measure.
\end{lemma}

We will need to control the number of excursions between $c$ and $\CL_\torus$ before time $T$.
This is our next lemma.
\begin{lemma} \label{le1}
There is a constant $r > 0$ such that for all sufficiently small
$\eps$ we have
\[
\sup_{x \in \Lbar} \Ee^x \exp({-\sigma}) \leq 1 - r
\eps^{\beta - \frac{\alpha}{2}}.
 \]
\end{lemma}
The proof of this lemma is given in Section~\ref{sxnShortTimeAveraging2}.
Using the Markov property of the process and Lemma \ref{le1}, we
get the estimate
\begin{equation} \label{est12}
\sup_{x \in \torus} \Ee^x \exp({-\sigma_n}) \leq \sup_{x \in
\Lbar} \Ee^x \exp({-\sigma_{n-1}}) \leq (\sup_{x \in
\Lbar} \Ee^x \exp({-\sigma}))^{n-1} \leq ( 1 - r
\eps^{\beta - \frac{\alpha}{2}})^{n-1}.
\end{equation}
The next lemma, also proved in Section~\ref{sxnShortTimeAveraging2}, allows us to
estimate expressions of the type \eqref{eq:marting_problw} over the random intervals
$[0,\tau]$ and $[0,\sigma]$.
\begin{lemma} \label{inte}
For each $f \in \mathcal{D}$,  we have the following
asymptotic estimates
\begin{equs}[2]\label{est1a}
\sup_{x\in \torus}\left|\Ee^x\left[f(Y^{\e}_\sigma)-f(Y^{\e}_0)-\int_0^\sigma\mathcal{A}f(Y^{\e}_t) \, dt\right]\right|&\to 0~~&\text{as}~~\eps &\rightarrow 0,\\
\label{est1b}
\sup_{x\in \torus}\left|\Ee^x\left[f(Y^{\e}_\tau)-f(Y^{\e}_0)-\int_0^\tau\mathcal{A}f(Y^{\e}_t) \, dt\right]\right|&\to 0~~~~&\text{as}~~\eps &\rightarrow 0,\\
\label{est2a}
\sup_{x \in \Lbar} \left|\Ee^x\left[f(Y^{\e}_\sigma)-f(Y^{\e}_0)-\int_0^\sigma\mathcal{A}f(Y^{\e}_t) \, dt\right]\right| &=
o(\eps^{\beta - \frac{\alpha}{2}})~~~~&\text{as}~~\eps &\rightarrow 0,\\
\label{est3a}
\Ee^\nu\left[f(Y^{\e}_\tau)-f(Y^{\e}_0)-\int_0^\tau \mathcal{A}f(Y^{\e}_t) \, dt\right] &=
o(\eps^{\beta - \frac{\alpha}{2}})~~~~&\text{as}~~\eps &\rightarrow 0.
\end{equs}
Here~$\nu$ is the invariant measure on~$\CL_\torus$ given by Lemma~\ref{mixing}
\end{lemma}
We prove Lemma \ref{lei1} by splitting the time interval $[0,T]$ into subsequent upcrossing and downcrossing periods. The first downcrossing from the general starting point is special and the contribution to \eqref{eq:marting_problw} is estimated using \eqref{est1a}. The estimate \eqref{est12} gives us sufficient control on the growth rate of the number of upcrossing-downcrossings, so that by the stronger estimates \eqref{est2a}, \eqref{est3a}, we can show that the contribution from these time intervals to \eqref{eq:marting_problw} is also negligible. In order to be able to use \eqref{est3a}, we will use Lemma \ref{mixing} to argue that after many such crossings, the section of the process on the separatrix can be approximated by its stationary counterpart. Finally, estimate \eqref{est1b} is used to show that the error thus introduced is negligible for small $\e$.
\begin{proof}[Proof of Lemma \ref{lei1}]
Let $f \in \mathcal{D}$, $T > 0$,
and $\eta > 0$ be fixed. We would like to show that the absolute
value of the left hand side of \eqref{eq:marting_problw} is less than $\eta$
for all sufficiently small positive $\eps$.

First, we replace the time interval $[0,T]$ by a  larger one, $[0,
\widetilde{\sigma}]$, where $\widetilde{\sigma}$ is the first of the
stopping times $\sigma_n$ that is greater than or equal to $T$,
that is
\[
 \widetilde{\sigma} = \sigma_{N+1}\;,\qquad  N = \max\{n: \sigma_n < T\}\;.
 \]
Using the Markov property of the process, the difference can be
rewritten as
\begin{equs}
\Big|\Ee^x&\Big[f(Y^{\e}_{\widetilde{\sigma}})-f(Y^{\e}_0)-\int_0^{\widetilde{\sigma}}\mathcal{A}f(Y^{\e}_t)dt\Big] -
\Ee^x\Big[f(Y^{\e}_T)-f(Y^{\e}_0)-\int_0^T\mathcal{A}f(Y^{\e}_t)dt\Big]\Big| \\
&= \Big|\Ee^x \Ee^{Z_{T}}
\Big[f(Y^{\e}_\sigma)-f(Y^{\e}_0)-\int_0^\sigma \mathcal{A}f(Y^{\e}_t)dt\Big]\Big|.
\end{equs}
Using~\eqref{est1a} we can ensure that the right hand side of the above is smaller than $\frac{\eta}{5}$ for all sufficiently small $\eps$.
Therefore, it remains to show that
\[
\left|\Ee^x\left[f(Y^{\e}_{\widetilde{\sigma}})-f(Y^{\e}_0)-\int_0^{\widetilde{\sigma}}\mathcal{A}f(Y^{\e}_t)dt\right]\right| < \frac{4\eta}{5}
\]
for all sufficiently small $\eps$.
 Using the stopping times $\tau_n$ and $\sigma_n$, we can rewrite
the expectation in the left hand side of this inequality as
\begin{equs}
\Ee^x&\left[f(Y^{\e}_{\widetilde{\sigma}})-f(Y^{\e}_0)-\int_0^{\widetilde{\sigma}}\mathcal{A}f(Y^{\e}_t)dt\right] =
\Ee^x\left[f(Y^{\e}_{{\sigma}})-f(Y^{\e}_0)-\int_0^{{\sigma}}\mathcal{A}f(Y^{\e}_t)dt\right] \\
&\qquad+
 \Ee^x\left( \sum_{n = 1}^{N} \Ee^{Z_{\sigma_n}} \left[f(Y^{\e}_{{\tau}})-f(Y^{\e}_0)-\int_0^{{\tau}}\mathcal{A}f(Y^{\e}_t)dt\right]  \right) \label{e:term1}\\
&\qquad+
\Ee^x\left( \sum_{n = 1}^{N}  \Ee^{Z_{\tau_n}} \left[f(Y^{\e}_{{\sigma}})-f(Y^{\e}_0)-\int_0^{{\sigma}}\mathcal{A}f(Y^{\e}_t)dt\right]  \right)\;,\label{e:term2}
\end{equs}
provided that the sums in the right hand side converge absolutely
(which follows from the arguments below). Due to \eqref{est1a},
the absolute value of the first term on the right hand side of
this equality can be made smaller than $\frac{\eta}{5}$ for all
sufficiently small $\eps$. Therefore, it remains to
estimate the two sums.

Let us start with the second sum~\eqref{e:term2}. Note that
\[
\Pe^x (\sigma_{n} < T ) = \Pe^x
\bigl(e^{-\sigma_{n}} > e^{-T} \bigr) \leq e^T \bigl( 1 - r
\eps^{\beta - \frac{\alpha}{2}}\bigr)^{n-1},
\]
where the last inequality follows from~\eqref{est12} and Chebyshev's inequality.
Taking the sum in $n$, we obtain
\[
\Ee^x N \leq \sum_{n =1}^\infty e^T \left( 1 - r
\eps^{\beta - \frac{\alpha}{2}}\right)^{n-1} \leq \frac{K}{\eps^{\beta - \frac{\alpha}{2}}},
\]
where the constant $K$ depends on $T$ and $r$. By Lemma
\ref{inte}, we can find $\eps_0 > 0$ such that for all
$\eps \in (0,  \eps_0)$ we have
\[
\sup_{x \in \Lbar} \left|\Ee^{x} \left[f(Y^{\e}_{{\sigma}})-f(Y^{\e}_0)-\int_0^{{\sigma}}\mathcal{A}f(Y^{\e}_t)dt\right] \right| \leq
\frac{\eta \eps^{\beta - \frac{\alpha}{2}}}{5 K}.
\]
Multiplying these two bounds it follows that, for $\eps < \eps_0$, the term \eqref{e:term2}
is bounded by $\eta / 5$.

Next, to estimate the term~\eqref{e:term1}, we first note that~\eqref{est3a} and the above argument shows
\begin{equation}\label{eqnTerm1tmp1}
\left| \Ee^x\sum_{n=1}^N \Ee^{\nu} \left[f(Y^{\e}_{{\tau}})-f(Y^{\e}_0)-\int_0^{{\tau}}\mathcal{A}f(Y^{\e}_t)dt\right]  \right| \leq \frac{\eta}{5}\;.
\end{equation}
The left hand side of this inequality, however, is not quite the term~\eqref{e:term1}, since the inner expectation is with respect to the invariant measure~$\nu$ rather than individual points.
(This limitation is due to~\eqref{est3a}.)
Thus, in view of~\eqref{eqnTerm1tmp1}, to estimate~\eqref{e:term1} we only need to bound
\[
  \abs*{ \Ee^x \sum_{n=1}^N (F(Z_{\sigma_n}) - \overline{F} ) },
\]
where
\[
F(x) \defeq \Ee^{x} \left[f(Y^{\e}_{{\tau}})-f(Y^{\e}_0)-\int_0^{{\tau}}\mathcal{A}f(Y^{\e}_t) \, dt\right]
\qquad\text{and}\qquad
\overline{F} \defeq \int_{\CL_\torus} F\, d \nu\;.
\]
Observe
\begin{equs}
\Big| \Ee^x\sum_{n=1}^N \big(F(Z_{\sigma_n}) - \overline{F} \big) \Big| &\leq
\Big|\sum_{n=1}^\infty \Ee^x\left( F(Z_{\sigma_n}) - \overline{F}  \right)\Big| +\Big| \Ee^x\sum_{n>N} \big(F(Z_{\sigma_n}) - \overline{F} \big) \Big|\\
&=
\Big|\sum_{n=1}^\infty \Ee^x\left( F(Z_{\sigma_n}) - \overline{F}  \right)\Big| +
\Big| \Ee^x\Ee^{Z({\tau_N})}\sum_{n=1}^\infty \big(F(Z_{\sigma_n}) - \overline{F} \big) \Big| \\
&\leq
2 \sup_{x \in \CL_\torus} | F(x) |
 \sum_{n=1}^\infty \sup_{x \in \CL_\torus}
\|P_1^{n-1}(x, \cdot) - \nu\|_\TV,
\end{equs}
which is smaller than $\frac{\eta}{5}$ for all sufficiently small
$\eps$ due to \eqref{mixingeq} and \eqref{est1b}.
Consequently the term~\eqref{e:term1} is bounded by $2\eta / 5$ when $\epsilon$ is sufficiently small.

Combining the above estimates, we see that the absolute value of the left
hand side of \eqref{eq:marting_problw} is less than $\eta$ for all sufficiently
small positive $\eps$. This completes the proof of Lemma~\ref{lei1}.
\end{proof}


\section{Proofs of lemmas used in Section~\texorpdfstring{\ref{sxnShortTimeAveraging}}{6}}\label{sxnShortTimeAveraging2}

In this section we prove Lemmas~\ref{tins}, \ref{le1}, and \ref{inte}.
We start with estimates on the transition times and transition probabilities between different level sets of $H$.  Recall that $\CL(h) =\{x\in \torus : |H(x)|=h\}$
and define $\CL_i(h) = \CL(h) \cap \overline{U}_i$.
For $h \geq 0$, let
\begin{equation*}
  \bar\tau_h
  = \bar\tau(h)
  \defeq \inf \set{t \geq 0 : Z_t \in \CL(h)}\;,
\end{equation*}
so in particular, $\bar\tau_0 = \sigma$ and $\bar\tau\paren{\e^\beta} = \tau$.
For $0 \leq h_1 \leq h_2$, let $U(h_1, h_2) = \{ x \in \torus: h_1 \leq |H(x)| \leq h_2 \}$ and $U_i(h_1, h_2) = U(h_1, h_2) \cap \overline{U}_i$.

In what follows, we take a more detailed look at the behaviour of $Z_t$ near the separatrix. Let $z_t = z_t^{\e}(x)$ be the deterministic process
\[
dz_t = \alpha {\eps^{\alpha -1} \abs{\log \e} } \, v(z_t) \,dt,
\quad z_0 = x.
\]
This is the same as the process $Z_t$ under $\Pe^x$, but 
with the stochastic term removed.
Let $\Teps = \Teps(x)$ be the time it takes for the process $z_t^\eps$,
starting at $x$, to make one rotation along the level set, i.e.,\ $\Teps(x) =\inf\{t > 0: z_t = x \}$.

\begin{lemma} \label{lstNew}
Suppose that $\lambda_1(\e), \lambda_2(\e)$ are such that $\e^{\lambda_1(\e)}, \e^{\lambda_2(\e)} = \CO(\e^{\alpha/2})$
and $\lambda_1(\e) \leq \lambda_2(\e) < 1/2-c $ for some $c > 0$.
\begin{enumerate}[\hspace{2ex}(a)]
\item\label{lstNewA} There are positive constants $c_1$ and $c_2$ such that $c_1
\eps^{1-\alpha} \leq \Teps(x) \leq c_2 \eps^{1-\alpha}$ for
all sufficiently small $\eps$ and all $x \in
U(\eps^{\lambda_2}, \e^{\lambda_1})$. Moreover, there are constants $\overline{c}_i>0$ such that if $\alpha' > 0$ and $\lambda'_1(\e),\lambda'_2(\e)\to\alpha'/2$ as $\e\downarrow 0$, then $\eps^{-(1-\alpha)}\Teps(x)\to \overline{c}_i\alpha'/\alpha$ as $\eps\downarrow 0$ uniformly in $x\in U_i(\eps^{\lambda'_2}, \e^{\lambda'_1})$.

\item\label{lstNewB}
  For each $\delta > 0$, $R > 0$, and all sufficiently small
$\eps$ we have
\begin{equation} \label{dspn}
\Pe^x\Big(\sup_{t \leq T(x)}|H(Z_t) - H(x)| >
\eps^{\frac{1}{2}-\delta} \Big) < \eps^R~~~{\it
for}~~{\it all}~~x \in U(\e^{\lambda_2(\e)}, \e^{\lambda_1(\e)}).
\end{equation}

\item\label{lstNewC}
  For each $\delta>0$, $R > 0$, and all sufficiently small
$\eps$ we have
\[
\Pe^x\Big(\sup_{t \leq T(x)}|Z_t - z^{\e}_t(x)| >
\eps^{\frac{1}{2}-\lambda_2(\e)-\delta} \Big) < \eps^R~~~{\it
for}~~{\it all}~~x \in  U(\e^{\lambda_2(\e)}, \e^{\lambda_1(\e)}).
\]
\end{enumerate}
\end{lemma}

Statement~(\ref{lstNewA}) of the above lemma follows from a direct computation for the deterministic process, and the fact that
\[
\overline{c}_i = \lim_{\e\downarrow 0}\frac{T_i(\e^{1/2})}{\abs{\log\epsilon}},
\]
where the limit is the same arising in~\eqref{eq:shorttimecoeff}.
Statement (\ref{lstNewB}) of the above lemma is basically a large deviation estimate on probability of the stochastic process to 
move transversal to the level-sets of $H$. Statement (\ref{lstNewC}) is  a large deviation estimate plus the fact that purely deterministic flow may separate points by an amount
$\e^{-\lambda_2  -\delta}$,
if both points are outside the boundary layer  $|H(x)| \leq \e^{\lambda_2}$.
Heuristic explanation for this deterministic separation is as follows.  The rotation time $T$ is of the same order as the time needed to pass the neighbourhood of a saddle and
the largest separation also occurs near saddles. 
Let us analyse 
the linearised system $\dot{x}=x$, $\dot{y}=-y$.  For the linearised system the particle trajectories 
and separation between particles behave as the $x$-component given by $x=x_0 e^t$. The time to pass the neighbourhood of a saddle 
if found from $x_0 e^t=\CO(1)$. Thus the rotation time $T = - C \log x_0$, and separation after one rotation is  $\CO(e^{T})=c/x_0 \leq \CO(\e^{-\lambda_2})$. In order to claim 
the separation in the nonlinear system occurs at the same rate, one needs to use the normal forms argument (see e.g.~\cite{BakhtinAlmada2011}), because eigenvalues of 
the linearised system are resonant. Therefore we only obtain an estimate $\e^{-\lambda_2  -\delta}$ with the linear approximation argument.
The rigorous proof of Lemma~\ref{lstNew} is identical to that of Lemma 3.3  in~\cite{DolgopyatKoralov08}, and we do not repeat it here.

\begin{lemma} \label{trt} Suppose that $\lambda_1(\e)\leq \lambda_2(\e)$ are such that $\e^{\lambda_1(\e)}, \e^{\lambda_2(\e)} = \CO(\e^{\alpha/2})$, $\lambda_1(\e), \lambda_2(\e) < 1/2 - \kappa$ for some $\kappa > 0$, and $(\lambda_2(\e)-\lambda_1(\e))\abs{\log\epsilon}\to \infty$ as $\e\downarrow 0$. Then

\begin{enumerate}[\hspace{2ex}(a)]
\item\label{trtA} When $\lambda_1(\e)-\lambda_2(\e)\to 0$ as $\e \downarrow 0$,  we have the following upper bound on the expected exit time from a channel.
\[
\sup_{x \in U(0, \e^{\lambda_2(\e)} )} \Ee^{x} (\bar\tau\paren{\e^{\lambda_1(\e)}} \wedge \bar\tau_0) = \CO(\e^{\lambda_2(\e) + \lambda_1(\e) - \alpha})~~{\it as}~\e \downarrow 0.
\]
\item\label{trtB}
  If $\lambda_2(\e) \leq 2\lambda_1(\e) -c $ for some $c > 0$,  the asymptotic behaviour of the exit probabilities is given by
\[
\Pe^{x}(\bar\tau\paren{\e^{\lambda_1(\e)}} < \bar\tau_0) \sim \e^{\lambda_2(\e) - \lambda_1(\e)}~~{\it as}~\e \downarrow 0~~{\it uniformly}~{\it in}~~x \in \CL(\e^{\lambda_2(\e)}).
\]
\item\label{trtC} The asymptotic expected exit time from a two sided channel satisfies
\[
\sup_{x \in U(0, \e^{\lambda_1(\e)} )} \Ee^{x} \bar\tau\paren{\e^{\lambda_1(\e)}} = \CO(\e^{ 2 \lambda_1(\e)  - \alpha})~~{\it as}~\e \downarrow 0.
\]
\item\label{trtD} There is a constant $c>0$ such that
\[
\sup_{x\in\CL(\e^{\lambda_2(\e)})}\Ee^{x} (\bar\tau\paren{\e^{\lambda_1(\e)}} \wedge \bar\tau_{0})>c\e^{\lambda_2(\e)+\lambda_1(\e)-\alpha}.
\]
\end{enumerate}
%
%
\end{lemma}
\begin{proof}
Without loss of generality, we may assume that the initial point $x$ belongs to  $U_i$, where $U_i$ has the property that $H(x) \geq 0$ for all $x \in U_i$. Recall from Lemma 4.2 (after applying an appropriate time change) in \cite{Koralov04} that
\begin{equation} \label{jjii}
\sup_{x\in U_i(0,h)}\Ee^{x}\bar{\tau}(h)\leq c\e^{-\alpha}h^2\;,
\end{equation}
for some $c>0$, which implies the third statement.  Also, by  \cite[Lem.~4.3]{Koralov04},
\begin{equation}\label{eq:lem43Kor}
\Pe^{x}\big(\bar{\tau}(h)<\bar{\tau}_0\big)=\frac{H(x)}{h}+\CO(h\abs{\log\epsilon}),
\end{equation}
which implies the second statement as $\lambda_2(\eps)<2\lambda_1(\eps) -c $. (Lemma~4.3 in \cite{Koralov04}
can be improved to the extent where the assumption $\lambda_2(\eps)<2\lambda_1(\eps) -c $ is not necessary, but we don't need
it here.)

It remains to prove the first and fourth statements. From now on, we write $\lambda_1$ instead
of $\lambda_1(\eps)$, and similarly for $\lambda_2$.
We also introduce $\lambda_3 = {1\over 2}-\kappa$ for a small number $\kappa > 0$
and set $u^{\eps}(x)=\Ee^x(\bbtau)$. Note that $u^{\e}$ satisfies the boundary value problem
\begin{equs}[2]
L^{\e}u^{\e}(x)&=-1,&\qquad x&\in U_i(\e^{\lambda_3},\e^{\lambda_1}),\\
u^{\e}(x)&=0,&\qquad x&\in \partial  U_i(\e^{\lambda_3},\e^{\lambda_1}),
\end{equs}
where
\[
L^{\e}=\frac{\alpha\e^{\alpha}\abs{\log\epsilon}}{2}\Delta+\frac{\alpha\abs{\log\epsilon}}{\e^{1-\alpha}}\nabla^{\perp}H\cdot \nabla
\]
is the generator of the process $Z_t$ under $\Pe^x$. Let also $\tilde{u}^{\e}(x)=\e^{-\alpha}(H(x)-\e^{\lambda_3})(\e^{\lambda_1}-H(x))$. It is not hard to see that
\begin{equs}[2]
L^{\e}\tilde{u}^{\e}(x)&=-\alpha\abs{\log\epsilon}|\nabla H(x)|^2+h^{\e}(x),&\qquad x&\in U_i(\e^{\lambda_3},\e^{\lambda_1}),\\
\tilde{u}^{\e}(x)&=0,&\qquad x&\in \partial  U_i(\e^{\lambda_3},\e^{\lambda_1}),
\end{equs}
where
\[
h^{\e}(x)=\frac{\alpha\abs{\log\epsilon}}{2}\Delta H(x)(\e^{\lambda_1}+\e^{\lambda_3}-2H(x))=\CO(\e^{\lambda_1}\abs{\log\epsilon})\;,
\]
uniformly in $x\in U_i(0,\e^{\lambda_1})$.

Let $c>0$ be a constant to be specified later and note that
\begin{equs}[2]
L^{\e}(u^{\e}-c\tilde{u}^{\e})(x)&=c \alpha \abs{\log\epsilon}|\nabla H(x)|^2-1-ch^{\e}(x),&\qquad x&\in U_i(\e^{\lambda_3},\e^{\lambda_1}),\\
u^{\e}(x)-c\tilde{u}^{\e}(x)&=0,&\qquad x&\in \partial  U_i(\e^{\lambda_3},\e^{\lambda_1}).
\end{equs}

Writing $\bbtau = \bar{\tau}({\e^{\lambda_1}})\wedge\bar{\tau}\paren{\e^{\lambda_3}}$
for the first exit time from the region $U_i(\e^{\lambda_3},\e^{\lambda_1})$,
it follows from the Feynman-Kac formula that 
\[
u^{\e}(x)-c\tilde{u}^{\e}(x)=\Ee^x\int_0^{\bbtau}\left(1-\alpha c\abs{\log\epsilon}|\nabla H(Z_t)|^2+ch^{\e}(Z_t)\right)dt.
\]
We will show that for all sufficiently small $\kappa>0$, by choosing $c>0$ large (small), the right hand side can be made positive (negative) for all $x\in U_i(2\e^{\lambda_3},\e^{\lambda_2})$ for small enough $\e$, which then implies that there is a constant $c>0$ such that
\begin{equation}\label{eq:boundss}
\frac{1}{c}\e^{-\alpha}(H(x)-\e^{\lambda_3})(\e^{\lambda_1}-H(x)) \leq \Ee^x (\bbtau) \leq c \e^{-\alpha}(H(x)-\e^{\lambda_1})(\e^{\lambda_1}-H(x)),
\end{equation}
for all $x\in U_i(2\e^{\lambda_3},\e^{\lambda_2})$.

To do this, it clearly suffices to show that there are constants $0<A<B$ such that, for small enough $\e$, the quantity
\[
I(x,\bbtau,\e) := \Ee^x\int_0^{\bbtau}\abs{\log\epsilon}|\nabla H(Z_t)|^2\,dt\;,
\]
satisfies
\begin{equ}[e:wantedI]
\frac{I(x,\bbtau,\e)}{\Ee^x\big(\bbtau\big)}\in [A,B]\;,\qquad \forall x\in U_i(2\e^{\lambda_3},\e^{\lambda_2})\;.
\end{equ}

Let us rewrite $I(x,\bbtau,\e)$ by breaking the domain of integration into intervals corresponding to individual rotations of the unperturbed process $z_t^{\e}$. We inductively define the stopping times
\[
\hat T_0^\eps=0\;,\qquad \hat T_{n+1}^\eps = \hat T_n^\eps + \Teps\big(Z({\hat T_n^\eps})\big)\;,\qquad
T_n^\eps = \hat T_n^\eps \wedge \bar{\tau}_0\wedge\bar{\tau}({2\e^{\alpha/2}})\;,
\]
and note that by part (\ref{lstNewB}) of Lemma \ref{lstNew}, we have that for every $R$ and  small enough $\e$,
\begin{equation}\label{eq:T1notT}
\Pe^{x}(T_1^\eps\neq \Teps(x))\leq\Pe^x\Big(\sup_{0\leq t\leq \Teps(x)}|H(Z_t)-H(x)|>\e^{\lambda_3}\Big)\leq\e^{R}
\end{equation}
for $x\in U_i(\e^{\lambda_3},\e^{\lambda_1})$.
Setting
\[
\bTeps = \min\{T_n^\eps : T_n^\eps\geq \bbtau\}\;,
\]
we replace the exit time in $I(x,\bbtau,\e)$ by $\bTeps$ in the upper limit of the integration. 
We will show later that the error introduced in this way is of order $\CO(\e^{1-\alpha})$. 
We have the identity
\begin{equs}[eq:rot1]
I(x,\bTeps,\e) &=\sum_{n=0}^{\infty}\Ee^x\left[\one_{\{T_n^\eps<\bTeps\}}\Ee^{Z(T_n^\eps)}\int_0^{T_1^\eps}\abs{\log\epsilon}|\nabla H(Z_t)|^2dt\right] \\
&= \sum_{n=0}^{\infty}\Ee^x \Bigl(\one_{\{T_n^\eps<\bTeps\}}I(Z(T_n^\eps),T_1^\eps,\e)\Bigr)\;.
\end{equs}
By \eqref{eq:T1notT} and part (\ref{lstNewC}) of Lemma~\ref{lstNew}, it is not hard to see that for any $y\in U_i(\e^{\lambda_3},\e^{\lambda_1})$,
\begin{equation}\label{eq:rot2}
I(y,T_1^\eps,\e) = I(y,T^\eps(y),\e) +o(\e^{1-\alpha}).
\end{equation}
 Since $y \in U_i$, one can write this as
\begin{equation}\label{eq:rot3}
I(y,T^\eps(y),\e) = \frac{\e^{1-\alpha}}{\alpha}\oint_{\CL_i(H(x))}|\nabla H| \, dl=\frac{\e^{1-\alpha}}{\alpha}\oint_{\partial U_i}|\nabla H| \, dl+o(\e^{1-\alpha}).
\end{equation}
Putting together \eqref{eq:rot1}, \eqref{eq:rot2}, \eqref{eq:rot3}, we conclude that there exists a constant $c$ such that
\begin{equ}
I(x,\bTeps,\e)
= \Ee^x\Bigl(\sum_{n=0}^{\infty}\one_{\{T_n^\eps<\bTeps\}}\left(c \eps^{1-\alpha} +o(\eps^{1-\alpha})\right) \Bigr)\;.
\end{equ}
By part (\ref{lstNewA}) of Lemma~\ref{lstNew}, there exist constants $0 < a < b$ such that, for $\eps$ small enough 
and on the event 
$\{T_n^\eps<\bTeps\}$, one has $\eps^{\alpha-1}(T_{n+1}^\eps-T_{n}^\eps) \in (a,b)$ almost
surely. Therefore,  there exists a closed interval $J \subset \R_+ \setminus \{0\}$, a sequence of random variables 
$c_n(\eps) \in J$ and an element $c(\eps) \in J$ such that 
\begin{equ}
I(x,\bTeps,\e)
= \Ee^x\Bigl(\sum_{n=0}^{\infty}\one_{\{T_n^\eps<\bTeps\}}\bigl(c_n(\eps) + o(\eps^{1-\alpha})\bigr)[T_{n+1}^\eps-T_{n}^\eps]\Bigr) =
\bigl(c(\e) + o(1)\bigr)\Ee^x\bTeps\;.
\end{equ}

In order to obtain \eqref{e:wantedI}, we would like to replace 
$\bTeps$ by $\bbtau$. First note that
\begin{equ}[e:boundDiffExit]
\Ee^x\big(\bTeps-\bbtau\big)=\CO(\e^{1-\alpha}).
\end{equ}
It is clear from the analysis in \eqref{eq:rot1}, \eqref{eq:rot2}, and \eqref{eq:rot3} and part (\ref{lstNewB}) of Lemma \ref{lstNew} that this is indeed a small error term compared to each of the stopping times under the expectation. Indeed, by the conditions on $\lambda_1,\lambda_2$, it follows that $\e^{\lambda_2-\lambda_1}\to 0$ as $\e\downarrow 0$. Pick a $\delta>0$ such that $\e^{\lambda_2}+\e^{1/2-\delta}<\e^{\lambda_1}$ for sufficiently small $\e$. Also choose $0<\bar\kappa<\kappa\wedge\delta$ and note that in order to get out of $U_i(\e^{\lambda_3},\e^{\lambda_1})$ with starting point in $U_i(2\e^{\lambda_3},\e^{\lambda_2})$ 
within $\e^{-\bar \kappa}$ rotations, $H(Z_t)$ needs to change by more than 
$\e^{1/2+\bar \kappa -(\kappa\wedge\delta)}$ during at least one of these rotations and thus, writing $\Rot$ for the number of rotations before
the exit time, one obtains from part (\ref{lstNewB}) of Lemma~\ref{lstNew}
\[
\Pe^x(\Rot \leq \e^{-\bar \kappa})<\e^{-\bar \kappa}\e^{R}
\]
for every $R>0$. This implies the primitive a priori lower bound on the expected exit time
\[
\Ee^x(\bbtau)>c\e^{1-\alpha}\Ee^x(\Rot)>c\e^{1-\alpha}\e^{-\bar \kappa}(1-\e^{R-\bar \kappa}),
\]
which, when combining it with \eqref{e:boundDiffExit}, shows that
\[
\Ee^x\big(\bTeps-\bbtau\big)=
o(1)\Ee^x\big(\bbtau\big).
\]
Similarly, replacing $\bTeps$ by $\bbtau$  in the integral produces an error term of order $\CO(\e^{1-\alpha}\abs{\log\epsilon})$, which is much smaller than $\Ee^x(\bbtau)$ by exactly the same argument.

This way we have proved
\[
I(x,\e)= c_2(\e)\Ee^x(\bbtau),
\]
where $c_2(\e)\in [A_2,B_2]$ for some positive constants $A_2,B_2$ and all  $x\in U_i(2\e^{\lambda_3},\e^{\lambda_2})$, and therefore by \eqref{eq:boundss},
\begin{equation}\label{eq:boundssss}
 \frac{1}{c}H(x)\e^{\lambda_1-\alpha}\leq\Ee^x (\bbtau) \leq c H(x)\e^{\lambda_1-\alpha}
\end{equation}
for all $x\in U_i(2\e^{\lambda_3},\e^{\lambda_2})$. This immediately implies the fourth statement of the lemma. Indeed, if $H(x)=\e^{\lambda_2}$, then
\[
\Ee^x \big(\bar\tau({\e^{\lambda_1}}) \wedge \bar\tau_0\big)>\Ee^x (\bbtau)>c^{-1}\e^{\lambda_1+\lambda_2-\alpha}.
\]

Similarly, for $x\in U_i(2\e^{\lambda_3},\e^{\lambda_2})$ we have the upper bound
\[
\Ee^x (\bbtau)  = \CO(\e^{\lambda_1+\lambda_2-\alpha}).
\]
To prove the first statement of the lemma, note that
\[
\Ee^x \big(\bar\tau({\e^{\lambda_1}}) \wedge \bar\tau_0\big)\leq \Ee^x (\bbtau)+\sup_{x\in\CL_i(\e^{\lambda_3})}\Ee^x\big(\bar\tau({\e^{\lambda_1}}) \wedge \bar\tau_0\big).
\]
The first term is of the right order by \eqref{eq:boundssss}. To treat the second one, let $\kappa_1>\kappa$. By \eqref{jjii},
\begin{equation}\label{eq:smallstep}
\sup_{x\in\CL_i(\e^{\lambda_3})}\Ee^x\big(\bar\tau({\e^{\lambda_1}}) \wedge \bar\tau_0\big)\leq c \e^{1-2\kappa_1-\alpha}+\sup_{x\in\CL_i(\e^{\lambda_3})}\Pe^x(\bar\tau({\e^{1/2-\kappa_1}}) < \bar\tau_{0})\sup_{x\in\CL_i(\e^{1/2-\kappa_1})}\Ee^x\big(\bar\tau({\e^{\lambda_1}}) \wedge \bar\tau_{0}\big).
\end{equation}
By \eqref{eq:lem43Kor}, the probability appearing on the right hand side is less than $\eps^{\kappa_1-\kappa}+c\e^{1/2-\kappa_1}\abs{\log\epsilon}$. By this and the third statement of the lemma, we have
\[
\sup_{x\in\CL_i(\e^{1/2-\kappa})}\Ee^x(\bar\tau({\e^{\lambda_1}}) \wedge \bar\tau_{0})<c(\e^{1-2\kappa_1-\alpha}+\e^{2\lambda_1-\alpha+\kappa_1-\kappa}+\e^{\frac{1}{2}-\kappa_1+2\lambda_1-\alpha}\abs{\log\epsilon}).
\]
It is easy to check by choosing $\kappa$ and $\kappa_1$ small enough and recalling that in this case $\lambda_2(\e)-\lambda_1(\e)$ goes to zero as $\e\downarrow 0$, all three of these terms are $\CO(\e^{\lambda_1+\lambda_2-\alpha})$, which finishes the proof for $x\in U_i(2\e^{\lambda_3},\e^{\lambda_2(\e)})$.

Finally if the starting point is somewhere in $U_i(0,2\e^{\lambda_3})$, then we first wait until the process gets out of this set. By \eqref{jjii}, the expected value of how long this takes is $\CO(\e^{2\lambda_3-\alpha})$ which does not change the conclusion.
%
\end{proof}

\begin{remark}
It is possible to prove the first statement of the previous lemma without the assumption that $\lambda_1(\e)-\lambda_2(\e)\to 0$ as $\e\downarrow 0$ by iterating \eqref{eq:smallstep}.
We don't, however, need this here.
\end{remark}

Before we proceed,
we introduce a little more notation. For $x \in \R^2 \setminus \CL$, 
write $U_{i(x)}$ for the $U_i$ containing $\pi(x)$ and define
$\Hep \colon \R^2 \to \R$ by
\begin{equ}
\Hep(y) = 
\left\{\begin{array}{cl}
	\eps^{-{\alpha\over 2}}|H(y)| & \text{if $\pi(y) \in U_{i(x)}$,} \\
	-\eps^{-{\alpha\over 2}}|H(y)| & \text{otherwise.}
\end{array}\right.
\end{equ}
The reason for introducing $\Hep$ is that it provides us with a ``signed'' version of the 
distance $d_\cG$ on $\cG$, i.e.,\ $|\Hep(y) - \Hep(x)| = d_\cG(\Gammaeps(x),\Gammaeps(y))$.

\begin{lemma} \label{nn3new} Suppose that $\beta_1 = \beta_1(\eps)$ and $\beta_2 = \beta_2(\eps)$ are such that $0 < \beta_2 < \beta_1 < 1/2$ and $\beta_1(\e), \beta_2(\e) \rightarrow \alpha/2$ as $\e \downarrow 0$. For each $f \in \mathcal{D}$,
\begin{equation}\label{ABC}
\sup_{x\in U_i(\eps^{\beta_1}, \e^{\beta_2})}\left|\Ee^x\left[f(Y^{\e}_{T(x)})-f(\Gammaeps(x))-\int_0^{T(x)}\mathcal{A}f(Y^{\e}_t)dt\right]\right| = o(\eps^{1-\alpha})~~~{\it as}~~\e \downarrow 0.
\end{equation}
Moreover,
\begin{equation}\label{ABC1}
\sup_{x\in U_i(\eps^{\beta_1}, \e^{\beta_2})} \left| \Ee^x(\Hep(Z_{ T(x)}) - \Hep(x)) \right| = o(\eps^{1-\alpha})~~~{\it as}~~\e \downarrow 0.
\end{equation}
\begin{equation}\label{ABC2}
\sup_{x\in U_i(\eps^{\beta_1}, \e^{\beta_2})} \left| \Ee^x[(\Hep(Z_{ T(x)}) - \Hep(x))^2 -  T(x) a_i] \right| = o(\eps^{1-{\alpha\over 2}})~~~{\it as}~~\e \downarrow 0.
\end{equation}

\end{lemma}
\begin{proof} Let $ T'(x)  = \min(T(x), \sigma)$. By parts (\ref{lstNewA}) and (\ref{lstNewB}) of Lemma~\ref{lstNew}, for every $R$,
\begin{equs}
\Pe^x (T(x) \neq T'(x)) &= o(\eps^R) = o(\eps^{1-\alpha})\;,\\
\Ee^x[T(x)-T'(x)]&=\Ee^x\one_{\{\sigma<T(x)\}}[T(x)-\sigma]\leq T(x)\Pe^x(\sigma<T(x))=\CO(\eps^{1-\alpha})o(\eps^R)=o(\eps^{1-\alpha})\;,
\end{equs}
uniformly in  $x\in U_i(\eps^{\beta_1}, \e^{\beta_2})$.
Due to this and the boundedness of $f$ and $\mathcal{A}f$, it is sufficient to prove \eqref{ABC} with $T(x)$ replaced by $T'(x)$.
By It\^o's formula,
\begin{equs}
\Ee^x &\left[f(Y^{\e}_{T'(x)})-f(\Gammaeps(x))-\int_0^{T'(x)}\mathcal{A}f(Y^{\e}_t)dt\right]\\
&=\frac{\alpha |\log \e|}{2} \Ee^x  \int_0^{T'(x)} \left(|\nabla H(Z_t)|^2 f''(Y^{\e}_t)
 + \e^{\alpha/2} \Delta H(Z_t) f' (Y^{\e}_t)\right) dt - \Ee^x \int_0^{T'(x)}\mathcal{A}f(Y^{\e}_t) dt.
\end{equs}
The contribution from the second term in the first integral on the right hand side can be ignored due to Part (\ref{lstNewA})  of Lemma~\ref{lstNew} and
the presence of the factor $\e^{\alpha/2}$. To deal with the first term, we observe that
\begin{equs}
\frac{\alpha |\log \e|}{2} \Ee^x  &\int_0^{T'(x)} (|\nabla H(Z_t)|^2 f''(Y^{\e}_t) -  |\nabla H(z^{\e}_t(x))|^2 f''(\Gammaeps(x)) ) \,dt \\
&=\CO(|\log \e| \e^{\frac{1}{2} - \beta_1 - \delta} T(x))   = o(\eps^{1-\alpha})\;,
\end{equs}
as $\eps \downarrow 0$, provided that $\delta$ is sufficiently small.  This is where we use that $H$ is $C^2$.
We also used the fact that $f'''$ is bounded (since $ \mathcal{A} f \in \Psi$) and parts (\ref{lstNewB}) and (\ref{lstNewC}) of Lemma~\ref{lstNew} for the first equality. We also used   part (\ref{lstNewA}) of Lemma~\ref{lstNew} for the second one. Thus the first term can
be replaced by
\begin{equs}
\frac{\alpha |\log \e|}{2} &f''(\Gammaeps(x)) \Ee^x  \int_0^{T'(x)} |\nabla H(z^{\e}_t(x))|^2  dt = \frac{\alpha |\log \e|}{2} f''(\Gammaeps(x)) \int_0^{T(x)} |\nabla H(z^{\e}_t(x))|^2  dt  + o(\eps^{1-\alpha})\\
&=\frac{\e^{1-\alpha} }{2}f''(\Gammaeps(x)) \oint_{\CL(|H(x)|) \cap U_i} |\nabla H| \, dl + o(\eps^{1-\alpha}) = \frac{\e^{1-\alpha} }{2} f''(\Gammaeps(x)) \oint_{\partial U_i} |\nabla H| \, dl + o(\eps^{1-\alpha}),
\end{equs}
as $\eps \downarrow 0$. The last term is treated similarly:
\[
\Ee^x \int_0^{T'(x)}\mathcal{A}f(Y^{\e}_t) dt = \frac{a_i}{2} T(x)  f''(\Gammaeps(x))  + o(\eps^{1-\alpha}) = \frac{\e^{1-\alpha}}{2}  f''(\Gammaeps(x)) \oint_{\partial U_i} |\nabla H| \, dl + o(\eps^{1-\alpha}),
\]
as $\eps \downarrow 0$,
where the last equality is due to the definition of $a_i$ (which is well defined by part (\ref{lstNewA}) of Lemma~\ref{lstNew}). Note that in the definition of $a_i$ in \eqref{eq:shorttimecoeff}, $T_i(x)$ is the period of the unperturbed $x^{x}_t$ while here $T(x)$ is the period of the process $z_t^{\eps}$. Collecting all the terms, we obtain
\[
\Ee^x \left[f(Y^{\e}_{T'(x)})-f(\Gammaeps(x))-\int_0^{T'(x)}\mathcal{A}f(Y^{\e}_t) \, dt\right]=o(\eps^{1-\alpha}),~\text{as}~ \eps \downarrow 0,
\]
as required for the proof of \eqref{ABC}.

Formulas \eqref{ABC1} and \eqref{ABC2} are proved similarly. For \eqref{ABC1}, we apply the same arguments to a smooth bounded function $f$ such that $f(h) = h - \Gammaeps(x)$ in a neighbourhood of $\Gammaeps(x)$. 
(Here, we identify the half line in $\cG$ containing $\Gammaeps(x)$ with $\R^+$ with the origin at $O$.)
For \eqref{ABC2}, we apply the above arguments to a smooth bounded function $f$ such that $f(h) = (h - \Gammaeps(x))^2$ in a neighbourhood of $\Gammaeps(x)$.
\end{proof}

The following version of Donsker's theorem will be useful. Suppose that $\{\mathcal{F}_n\}_{n \ge 0}$ is an increasing
sequence of $\sigma$-algebras, that $\tilde u_n$ and $\tilde \xi_n$ are two families of real-valued random variables 
measurable with respect to $\mathcal{F}_n$ and $\mathcal{F}_{n+1}$ respectively, and let
$\gamma > 0$ be a (small) parameter. Assume that 
the $\tilde u_n$ are positive and that there exists a deterministic $\bar u \in (0, \gamma)$ and such that
$|\tilde u_n - \bar u| \le \gamma \bar u$ almost surely for all $n$.
The quantities $\tilde u_n$ play the role of time steps for a random walk with spatial steps $\tilde \xi_n$. 
Define the partial sums $S_n = \tilde u_0 + \cdots + \tilde u_{n-1}$.
 Given an interval $[a,b]$ with $0 \in (a,b)$, define $\widetilde{R}_0 = 0$ and $\widetilde{R}_{n+1} = \widetilde{R}_{n} + \tilde \xi_{n}\one_{\tilde R_n \in (a,b)}$, $n \geq 0$.
 Assume that
\begin{enumerate}[\hspace{2ex}(a)]
\item\label{condA} For each $n$, $|\E (\tilde \xi_n|\mathcal{F}_{n})| \leq \gamma \bar u$ almost surely.
\item\label{condB} For each $n$, 
$|\E(\tilde \xi_n^2 |\mathcal{F}_{n})- \bar u|  \leq \gamma \bar u$ almost surely.
\item\label{condC} For each $\lambda > 0$ and all $n$, $\P( |\tilde \xi_n| > \lambda |\mathcal{F}_{n})  \leq \gamma \bar u$, almost surely.
\end{enumerate}

Let $n(t) = \max(n: S_n \leq t)$ and
define the continuous time process $R_t$ by
\[
R_t = \widetilde{R}_{n(t)} + \frac{t - S_{n(t)}}{S_{n(t)+1} - S_{n(t)}} \tilde \xi_{n(t)},~~t \geq 0\;.
\]
We now claim the above process is close to stopped Brownian motion.
\begin{lemma} \label{Donsker}
Let $\tilde W^{[a,b]}$ be a standard Brownian motion, stopped when it leaves $[a,b]$.
For every $\delta > 0$ and every continuous $F \colon C(\R^+, \R) \to \R$
there exists $\gamma_0$ such that, for every $\gamma \le \gamma_0$, one has
\begin{equ}
 \bigl|\E F(R) - \E F(\tilde W^{[a,b]})\bigr| \le \delta\;.
\end{equ}

Moreover, for arbitrary positive $t_0$, $C$, $\eta$, and $r$, there are  $\delta \in (0,1)$ and $\gamma_0 >0$ 
such that, for every $\gamma \le \gamma_0$ and every $a,b$ with $-C \le a < 0< b \le C$, one has
\begin{equation} \label{ttns}
\sup_{t \in [0, t_0]} \P\Big( \sup_{s \in [t, t+\delta]} |R_{s} - R_{t}| > r\Big) \leq \delta \eta\;.
\end{equation}
\end{lemma}
\begin{proof} 
A standard proof of Donsker's theorem (see \cite[Chap.~7.4]{EthierKurtz86} for the first statement and  \cite[Chap.~8]{billingsley1999convergence} for the second one) can be easily adapted to our situation.  
\end{proof}

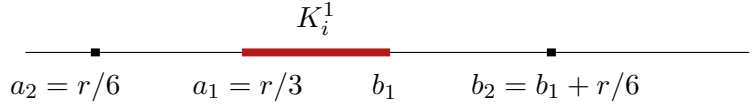
\begin{figure}
    \centering
\begin{tikzpicture}[scale=1.5]

\node at (1,-0.3) {$a_2=r/6$};
\node at (2.6,-0.3) {$a_1=r/3$};
\node at (3.8,-0.3) {$b_1$};
\node at (5.3,-0.3) {$b_2=b_1+r/6$};

\draw[-](0.65,0) -- (7,0);
\draw[line width=3pt,line cap=rect,draw=black] (1.26,0) -- (1.26,0);
\draw[line width=3pt,line cap=rect,draw=black] (5.26,0) -- (5.26,0);
\draw[line width=3pt,draw=darkred] (2.55,0) -- (3.85,0);

\node at (3.2,0.3) {$K_i^1$};

\end{tikzpicture}
\caption{Intervals $K_i^1=[a_1, b_1]$ and $K_i^2=[a_2,b_2]$.}
\label{fig:K}
\end{figure}

\begin{proof}[Proof of Lemma~\ref{tins}]
Recall that we identify one of the edges $I_i$ of the graph $\cG$ with the
semi-axis $\R_+$. Define $K_i^1 = [a_1,b_1] \subset I_i$, where $a_1 = r/3$, $b_1 > a_1$. Let $K_i^2 = [a_2, b_2]$, where $a_2 = r/6$, $b_2 = b_1 + r/6$.
%
For $x \in (\Gammaeps)^{-1}(K_i^2)$, define inductively
\[
y = \Gammaeps(x),\qquad u_0^\e = \Teps(x),\qquad u_n^\e = \Teps(Z_{S_n^\e}),\qquad \xi_n^{\e} = \Hep(Z({S_{n+1}^{\e}})) - \Hep(Z({S_{n}^{\e}})).
\]
Then Lemma~\ref{nn3new} (relations \eqref{ABC1} and \eqref{ABC2}), parts  (\ref{lstNewA}) and (\ref{lstNewB}) of Lemma~\ref{lstNew}, and the Markov property of the process imply that conditions (\ref{condA})--(\ref{condC}) preceding Lemma~\ref{Donsker} are met for
\begin{equ}
\widetilde{R}^{y,\eps}_n = a_i^{-1/2}\bigl(\Hep(Z_{S_n^\e}) - \Hep(x)\bigr)\;,\qquad 
\tilde \xi_n = a_i^{-1/2} \xi_n^{\e}\;,\qquad \tilde u_n = u_n^\e\;,
\end{equ}
with some constant $\gamma = \gamma^\eps$ converging to $0$ as $\eps \to 0$.
Therefore, by Lemma~\ref{Donsker} applied to the segment $K_i^2$ (suitably centred and rescaled) 
and part (\ref{lstNewB}) of Lemma~\ref{lstNew}, for arbitrary positive  $\eta$ and $r$, there are $\delta > 0$ 
and $\eps_0 > 0$ such that
\begin{equ}
\sup_{x \in \Gammaeps^{-1}(K_i^1) } \Pe^x \Big( \sup_{s \in [0,\delta]} d_\cG\bigl(Y^{\e}_s, \Gammaeps(x)\bigr) > \frac{r}{3}\Big) =
\sup_{x \in \Gammaeps^{-1}(K_i^1) } \Pe^x \Big( \sup_{s \in [0,\delta]} |\Hep(Z_s) - \Hep(x)| > \frac{r}{3}\Big) \leq {\delta \eta}
\end{equ}
for all $\eps \leq \eps_0$, where $d_\cG$ stands for distance on $\cG$. (There is no need to introduce the stopping time associated with exiting the segment $K_i^2$ since the starting point belongs to a smaller segment $K_i^1$.) Let ${K_{r/3}} \subset \cG$ be the set of points whose distance from $O$ does not exceed $r/3$. Using the strong Markov property for the process $Z_t$, i.e., stopping it when $\e^{-\alpha/2}|H(Z_t)| = r/3$, we obtain
\[
\sup_{x \in \Gammaeps^{-1}(K_{r/3}) } \Pe^x \Big( \sup_{s \in [0,\delta]} d_\cG\bigl(Y^{\e}_s, \Gammaeps(x)\bigr) > r\Big) \leq
\sup_{x \in \Gammaeps^{-1}(\bigcup_{i=1}^n K_i^1) } \Pe^x \Big( \sup_{s \in [0,\delta]}d_\cG\bigl(Y^{\e}_s, \Gammaeps(x)\bigr) > \frac{r}{3}\Big) \leq {\delta \eta}.
\]
For any compact set $K\subset \cG$, we can ensure $K \subseteq \bigcup_{i=1}^n K_i^1 \cup K_{r/3}$
by choosing $b_1$ sufficiently large.
Then
\[
\sup_{x \in \Gammaeps^{-1}(K) } \Pe^x \Big( \sup_{s \in [0,\delta]} d_\cG\bigl(Y^{\e}_s, \Gammaeps(x)\bigr) > r\Big) \leq {\delta \eta},
\]
which implies the statement by the Markov property.
\end{proof}

\begin{proof}[Proof of Lemma \ref{le1}] For any $\beta <\alpha/2$
conditioning and the strong Markov property imply
\begin{equation}\label{eq:thisisthelastnumber}
\Ee^x \bigl(e^{-\sigma}\bigr)\leq 1-\Pe^x \paren[\big]{ \bar{\tau}\paren{\eps^{\alpha/2}}<\bar{\tau}_0 }+
\Pe^x \paren[\big]{ \bar{\tau}\paren{\eps^{\alpha/2}}<\bar{\tau}_0}
  \sup_{y\in\CL(\e^{\alpha/2})} \Ee^y \bigl(e^{-\sigma}\bigr),
\end{equation}
for any $x\in \Lbar =\cL(\eps^\beta)$.
It follows from Lemma~\ref{Donsker} and Lemma \ref{lstNew}(\ref{lstNewB}) that there is a constant $r>0$ independent of $\e$ such that
\[
\sup_{y\in\CL(\e^{\alpha/2})}\Ee^y \bigl(e^{-\sigma}\bigr)
  \le \sup_{y\in\CL(\e^{\alpha/2})}\Ee^y \exp\brak[\big]{-\paren[\big]{\bar{\tau}\paren{\eps^{\alpha/2}/2}\wedge \bar{\tau}\paren{2\eps^{\alpha/2}}}}\leq1-2r\;,
\]
for all sufficiently small $\eps$. 
Using this in \eqref{eq:thisisthelastnumber}, we get
\[
\Ee^x \bigl(e^{-\sigma}\bigr)\leq 1- 2r \Pe^x \paren[\big]{
  \bar{\tau}\paren{\eps^{\alpha/2}} < \bar{\tau}_0
}
  \leq 1 - r\eps^{\beta - \frac{\alpha}{2}},
\]
for any         $x\in \Lbar$. The second inequality above follows from the second statement of Lemma~\ref{trt}. 
\end{proof}

\begin{proof}[Proof of Lemma~\ref{inte}]
First, let us show that \eqref{est1a} holds. As before, we define inductively
\[
u_0^\e = \Teps(x),\qquad u_n^\e = \Teps(Z_{S_n^\e}),\qquad S_n^\e = u_0^\e+\cdots+u_{n-1}^\e.
\]
Recall that $\bar\tau_h = \bar \tau(h)$ is the first time when $Z_t$ reaches ${\CL(h)}$.
By the third statement of Lemma~\ref{trt}, there is a constant $c$ such that
\begin{equation} \label{yu1}
\sup_{x \in {U(0,\e^{\alpha/2}r)}} \Ee^x \bar\tau\paren{r \e^{\alpha/2}} \leq c r^2,
\end{equation}
 while by \eqref{eq:lem43Kor} we have
\begin{equs} \label{yu2}
\sup_{x \in {U(0,\e^{\alpha/2}r)}} \Pe^x \paren[\Big]{ \bar\tau\paren{2r\e^{\alpha/2}} < \bar\tau\paren{\e^{\beta}} }
  &\leq \sup_{x \in \CL(r\eps^{\alpha/2})} \Pe^x\paren[\Big]{\bar\tau\paren{2r\e^{\alpha/2}} < \bar\tau_0}\\
  &\leq \frac{1}{2}+\CO\paren[\big]{\eps^{\alpha/2}\abs{\log\epsilon}}
  \leq \frac{2}{3}
\end{equs}
for all sufficiently small $\eps$. We claim that for each $r > 0$,
\begin{equation}\label{tuyee}
\sup_{x\in U(\e^\beta, r \e^{\alpha/2})}\left|\Ee^x \left[f\left(Y^{\e}_{\bar\tau\paren{\e^\beta} \wedge \bar\tau\paren{r\e^{\alpha/2}}}\right)
    -f(\Gammaeps(x))
    -\int_0^{\bar\tau\paren{\e^\beta} \wedge \bar\tau\paren{r\e^{\alpha/2}}}\mathcal{A}f(Y^{\e}_t) \, dt \right]\right|
  \xrightarrow{\epsilon\to0} 0.
\end{equation}
Indeed, let $\widetilde{n} = \min\{n: S_n^\e \geq \bar\tau\paren{\e^\beta} \wedge \bar\tau\paren{r\e^{\alpha/2}} \}$. Then, due to \eqref{dspn}, it is sufficient to show that
\begin{equation}\label{tuyee2}
\sup_{x\in U(\e^\beta, r \e^{\alpha/2})}
  \left|
    \Ee^x \brak[\Big]{
      f(Y^\epsilon_{S_{\widetilde{n}}^\e})-f(\Gammaeps(x))-\int_0^{S_{\widetilde{n}}^\e}\mathcal{A}f(Y^{\e}_t)dt}\right|\xrightarrow{\epsilon \to 0} 0.
\end{equation}
By part (\ref{lstNewA}) of Lemma~\ref{lstNew} and \eqref{yu1}, there is $\overline{c} > 0$ such that
\begin{align*}
  \Ee^x \widetilde{n}
    &=\sum_{n=0}^{\infty}\Pe^x \paren[\Big]{ S_n^{\eps} < \bar\tau\paren{\e^\beta} \wedge \bar\tau\paren{r\e^{\alpha/2}} }
  \leq \sum_{n=0}^{\infty}\Pe^x \paren[\Big]{ n\leq\frac{\bar\tau\paren{\e^\beta} \wedge \bar\tau\paren{r\e^{\alpha/2}} }{\overline{c}\eps^{1-\alpha}} }
  \\
 &\leq\frac{\Ee^x{\bar\tau}\paren{\e^\beta} \wedge \bar\tau\paren{r\e^{\alpha/2}} }{\overline{c}\eps^{1-\alpha}}\leq\frac{\Ee^x \bar\tau\paren{r\e^{\alpha/2}} }{\overline{c}\eps^{1-\alpha}}\leq c r^2 \e^{\alpha -1}
\end{align*}
for some constant $c$ and all sufficiently small~$\e$. Now the validity of \eqref{tuyee2} follows from \eqref{ABC}. By Lemma~\ref{trt} and \eqref{eq:choosinglevels},
\begin{equation*}
\sup_{x\in U(0,\e^\beta)} \Ee^x \paren[\big]{\bar\tau\paren{0} \wedge \bar\tau\paren{r\e^{\alpha/2}}}
  = \CO(\e^{\beta - \frac{\alpha}{2}}),
\end{equation*}
and
\begin{equation}\label{iji}
\sup_{x\in U(0,\e^\beta)} \Pe^x ({\bar\tau\paren{0} >  \bar\tau\paren{r\e^{\alpha/2}}})  \leq \sup_{x\in \CL(\e^\beta)} \Pe^x ({\bar\tau\paren{0} >  \bar\tau\paren{r\e^{\alpha/2}}}) = \CO(\e^{\beta - \frac{\alpha}{2}}),
\end{equation}
as~$\epsilon \to 0$.
Therefore, since $f'$ and $ \mathcal{A}f$ are bounded,
\begin{equation}\label{tuyee3}
\sup_{x\in U(0,\e^\beta)}
  \abs[\Bigg]{\Ee^x \brak[\Big]{
      f\paren[\Big]{ Y^{\e}_{\bar\tau\paren{0} \wedge \bar\tau\paren{r\e^{\alpha/2}}}}
      -f(\Gammaeps(x))
      -\int_0^{\bar\tau\paren{0} \wedge \bar\tau\paren{r\e^{\alpha/2}}}
	\mathcal{A}f(Y^{\e}_t) \, dt
      }
    }
  = \CO(\e^{\beta - \frac{\alpha}{2}}),
\end{equation}
as $\epsilon \to 0$.

Now take $r$ sufficiently large so that $f(\Gammaeps(x)) = 0$ whenever $x \notin U(0, r \e^{\alpha/2})$, which is possible since $f$ has compact support. By the strong Markov property,
\begin{equs}
I_0 &\defeq
  \sup_{x\in U(\e^\beta, r \e^{\alpha/2})}
    \abs[\Bigg]{
      \Ee^x \brak[\Big]{
	f(Y^{\e}_{\sigma})-f(\Gammaeps(x))
	  -\int_0^{\sigma}\mathcal{A}f(Y^{\e}_t) \, dt
      }
    }
    \\
  &\leq \sup_{x\in U(\e^\beta, r \e^{\alpha/2})}
  \abs[\Bigg]{
    \Ee^x \brak[\Big]{
      f\paren[\Big]{ Y^{\e}_{\bar\tau\paren{\e^\beta} \wedge \bar\tau\paren{2r\e^{\alpha/2}}}}
      -f(\Gammaeps(x))
      -\int_0^{\bar\tau\paren{\e^\beta} \wedge \bar\tau\paren{2r\e^{\alpha/2}}}\mathcal{A}f(Y^{\e}_t)
      \, dt
    }
  }\\
  &\quad + \sup_{x \in {U(\e^\beta,r \e^{\alpha/2})}} \Pe^x \paren[\Big]{ \bar\tau\paren{2r\e^{\alpha/2}} < \bar\tau\paren{\e^{\beta}} }
  \sup_{x\in \CL(2 r \e^{\alpha/2})}
    \abs[\Bigg]{
      \Ee^x \brak[\Big]{
	f\paren[\big]{ Y^{\e}_{\sigma} } -f(\Gammaeps(x))
	-\int_0^{\sigma}\mathcal{A}f(Y^{\e}_t) \, dt
      }
    }
  \\
  &\quad+
  \sup_{x\in \CL(\e^\beta)}
    \abs[\Bigg]{ \Ee^x \brak[\Big]{
	f(Y^{\e}_{\sigma})-f(\Gammaeps(x))
	-\int_0^{\sigma}\mathcal{A}f(Y^{\e}_t) \, dt
      }
    }
  \defeq I_1+ I_2 + I_3.
\end{equs}
By \eqref{yu2}, $I_2 \leq \frac{2}{3} I_0$ for all sufficiently small $\e$. By \eqref{iji} and \eqref{tuyee3}, $I_3 \leq \frac{1}{6} I_0 + \CO(\e^{\beta - \frac{\alpha}{2}})$ as $\e \downarrow 0$. Therefore, $I_0 \leq 6 I_1 + \CO(\e^{\beta - \frac{\alpha}{2}})$, which implies that $I_0 \rightarrow 0$ due to \eqref{tuyee} used with $2r$ instead of $r$. Finally, \eqref{est1a} follows by combining this with \eqref{tuyee3} and using the strong Markov property.

The proof of \eqref{est1b} is nearly identical, and we, therefore, omit its proof. 

Estimate~\eqref{tuyee3} is  not quite sufficient to obtain~\eqref{est2a}. Instead, we introduce $\beta'(\e)$ to be chosen later, such that $1/2>\beta'(\e) >\beta(\e)$, and express the supremum in question as
\begin{equs}
J&=\sup_{x\in \CL({\eps^{\beta}})}\Ee^{x}
  \left[f\left(Y^\epsilon_{\bar{\tau}({r\eps^{\alpha/2}})\wedge\bar{\tau}({\eps^{\beta'}})}\right)-f(\Gammaeps(x))-\int_0^{\bar{\tau}({r\eps^{\alpha/2}})\wedge\bar{\tau}({\eps^{\beta'}})}\mathcal{A}f(Y^\epsilon_t) \, dt\right]\\
&\quad+\sup_{x\in \CL({\eps^{\beta}})}
  \Ee^{x}\left(\one_{\{\bar{\tau}({r\eps^{\alpha/2}})<\bar{\tau}({\eps^{\beta'}})\}}
\Ee^{Z({\bar{\tau}({r\eps^{\alpha/2}})})}
\left[f\left(Y^\epsilon_{\sigma}\right)-f(Y^\epsilon_0)-\int_0^{\sigma}\mathcal{A}f(Y^\epsilon_t)dt\right]\right)\\
&\quad +\sup_{x\in \CL({\eps^{\beta}})}\Ee^{x}\left(\one_{\{\bar{\tau}({r\eps^{\alpha/2}})>
\bar{\tau}({\eps^{\beta'}})\}}\Ee^{Z({\bar{\tau}({\eps^{\beta'}})})}
\left[f\left(Y^\epsilon_{\sigma}\right)-f(Y^\epsilon_0)-\int_0^{\sigma}\mathcal{A}f(Y^\epsilon_t) \, dt
  \right]\right)\\
&=J_1+J_2+J_3.
\end{equs}

Note that \eqref{iji} implies that
\[
\sup_{x\in \CL({\eps^{\beta}})}\Pe^x (\bar{\tau}\paren{r\eps^{\alpha/2}}<\bar{\tau}\paren{\eps^{\beta'}})\leq\sup_{x\in\Lbar}\Pe^x (\bar{\tau}\paren{r\eps^{\alpha/2}}<\bar{\tau}\paren{0})=\CO(\eps^{\beta-\alpha/2}).
\]
Together with \eqref{est1a}, this implies $J_2=o(\eps^{\beta-\alpha/2})$. On the other hand, we have as before that
\begin{equation*} 
 \sup_{x\in \CL(\e^{\beta'})} \Pe^x ({\bar\tau\paren{0} >  \bar\tau\paren{\beta}}) = \CO(\e^{\beta'- \beta}),
\end{equation*}
as~$\epsilon \to 0$.
Similarly to \eqref{tuyee3},
\begin{equation*}
\sup_{x\in \CL(\e^{\beta'})}\left|\Ee^x \left[f\left(Y^{\e}_{\bar\tau\paren{0} \wedge \bar\tau\paren{\beta}}\right)-f(\Gammaeps(x))-\int_0^{\bar\tau\paren{0} \wedge \bar\tau\paren{\beta}}\mathcal{A}f(Y^{\e}_t)dt\right]\right|= \CO(\e^{\beta' - \frac{\alpha}{2}}),
\end{equation*}
as~$\epsilon \to 0$, and therefore by
\[
|J_3|\leq \sup_{x\in\CL(\eps^{\beta'})}\left|\Ee^x \left[f\left(Y^{\eps}_{\sigma}\right)-f(Y^{\eps}_0)-\int_0^{\sigma}\mathcal{A}f(Y^{\eps}_t)dt\right]\right|\leq \frac{|J|}{2}+\CO(\eps^{\beta'-\frac{\alpha}{2}}),
\]
where we want that $\eps^{\beta'-\beta} \rightarrow 0$, which can be achieved by choosing $\beta'$ appropriately.

Consequently, $|J|\leq 2J_1+o(\eps^{\beta-\alpha/2})$ and it remains to show that $|J_1|=o(\eps^{\beta-\alpha/2})$. This can be proved by first showing that
\begin{equation*}
J_1=o\paren[\Bigg]{ \sup_{x\in\CL(\e^{\beta})}\Ee^x \bar{\tau}\paren{r\eps^{\alpha/2}} \wedge\bar{\tau}\paren{\e^{\beta'}}}
\end{equation*}
(by breaking up the time interval into individual rotations and using~\eqref{ABC} and then applying Lemma~\ref{trt}(\ref{trtA})).
Since this is only a slight modification of the machinery we used above, we omit the details.

Finally, the left hand side of \eqref{est3a} can be written using a Taylor-expansion as
\begin{equation}\label{eq:exitbetaglue}
  \eps^{\beta-\alpha/2}\sum_{i=1}^n D_i f(0)
    \Pe^{\nu}(Z_{\tau}\in U_i )+o(\eps^{\beta-\alpha/2})+\norm{\mathcal{A}f}_{\infty}\Ee^x \tau.
\end{equation}
Recall, that by (78) in \cite{DolgopyatKoralov08}, there exists a constant $c > 0$ such that
\[
  \Pe^\nu(Z_{\tau}\in U_i)
    =\mu(\Lbar\cap U_i)
    =c (1+o(1)) \int_{\partial U_i} \abs{\nabla H} \, dl
    \xrightarrow{\epsilon \to 0} c q_i.
\]
Using this, the assumption $f \in \mathcal D$ and the third statement in Lemma~\ref{trt}, it follows that \eqref{eq:exitbetaglue} is $o(\eps^{\beta-\alpha/2})$, which proves \eqref{est3a}.
\end{proof}



\appendix
\section{Proof of Proposition~\ref{ppnCLTFirstHitShort}}\label{sxnCLTFirstHitShort}
In this appendix we sketch the proof of our limit theorem regarding the displacement of $Z_t$ during an upcrossing. The verification of this result is based on an abstract lemma that was proved in \cite{HairerKoralovPajorGyulai2014}.

Before we state the lemma we introduce its setup. Let $\CM$ be a metric space that can be written as a disjoint union
\[
\CM = \CX \sqcup C_1 \sqcup \ldots \sqcup C_n\;,
\]
where the sets $C_i$ are closed. Assume also that $\CX$ is a  $\sigma$-locally compact separable subspace, i.e., locally compact that is the union of countably many compact subspaces. Let $p_{\eps}(x,dy)$, $0\leq\eps\leq\eps^0$, be a family of transition probabilities on $\CM$
and let $g \in \CC_b(\CM,\R^2)$. Later, $p_{\eps}(x,dy)$ will come up as transition probabilities of a certain discrete time process associated to $Z_t$. We assume that the following properties hold:

\begin{enumerate}
\item $p_0(x, \CX) = 1$ for all $x \in \CM$ and $p_\eps(x, \CX) = 1$ for all $x \in \CM \setminus \CX$.
\item $p_0(x,dy)$ is weakly Feller, that is the map $x \mapsto \int_\CM f(y)p_0(x,dy)$ belongs
to $\CC_b(\CM)$ if $f\in \CC_b(\CM)$.
\item (Small escape probability from $\CX$.) There exist bounded continuous functions $h_1,\ldots,h_n: \CX \rightarrow \R_+$ such that
\[
\eps^{-\frac{1-\alpha}{2}} p_\eps(x,C_i) \rightarrow h_i(x)\;,
\]
uniformly over $x \in K$ for $K \subseteq \CX$ compact,
while $\sup_{x \in \CX} |\eps^{-\frac{1-\alpha}{2}} p_\eps(x,C_i)| \leq  c$ for some positive constant $ c $. We also have
\[
J(x) \defeq h_1(x) + \cdots + h_n(x) > 0~~\text{for}~ x \in \CX.
\]
\item $p_{\eps}(x,dy)$ converges weakly to $p_0(x,dy)$ as $\eps \rightarrow 0$, uniformly in $x\in K$ if $K\subseteq \CX$ is compact.
\item (Doeblin condition) The transition functions satisfy a strong Doeblin condition uniformly in~$\eps$. Namely, there exist a probability measure $\eta$ on $\CX$, a constant $a > 0$, and an integer $m > 0$ such that
\[
p_{\eps}^m(x,A) \geq a~\eta(A)~~~\text{for}~~x \in \CM,~A \in \mathcal{B}(\CX),~\eps \in [0, \eps_0].
\]
It then follows that for every $\eps$, there is a unique invariant measure $\lambda^{\eps}(dy)$ on $\CM$ for $p_{\eps}(x,dy)$, and the associated Markov
chain is uniformly exponentially mixing, i.e., there are $\Lambda>0,c>0$, such that
\[
|p_{\eps}^k(x,A)-\lambda^{\eps}(A)|\leq c~e^{-\Lambda k}~~~\text{for all}~ x \in \CM,~A \in \mathcal{B}(\CM),~\eps \in [0, \eps_0].
\]
\item The function $g$ is such that $\int_\CM g\, d\lambda^{\eps}=0$ for each $\eps \in [0,\eps_0]$.
\end{enumerate}

\begin{lemma} \label{abstractlemma}
(Lemma 2.4 in~ \cite{HairerKoralovPajorGyulai2014})
Suppose that Properties 1--6 above are satisfied and let $R^{x,\eps}_k$ be the Markov chain on $\CM$ starting at $x$, with transition function $p_\eps$. Let $\tau = \tau(x,\eps)$ be the first time when the chain reaches the set $C = C_1 \sqcup \ldots \sqcup C_n$. Let $e(R^{x,\eps}_k) = i$ if $R^{x,\eps}_k \in C_i$.
Then, as $\eps \to 0$,
\[
\Big( \eps^{\frac{1-\alpha}{4}} (g(R^{x,\eps}_1) + \cdots +g(R^{x,\eps}_\tau)), e(R^{x,\eps}_\tau) \Big) \rightarrow (F_1, F_2)
\]
in distribution, uniformly in $x \in \CX$, where $F_1$ takes values in $ \mathbb{R}^2$, $F_2$ takes values in $\{1,\ldots,n\}$, and $F_1$ and $F_2$ are independent. The random variable $F_1$ is distributed as $(\xi/\int_\CX J d \lambda^0)^{1\over 2} N(0, \bar{Q})$, where $\xi$ is exponential with parameter one independent of $N(0,\bar{Q})$ and $\bar{Q}$ is the matrix such that
\[
( g(R^{x, 0}_1) + \cdots + g (R^{x, 0}_k)) /\sqrt{k} \rightarrow N(0,\bar{Q})
\]
in distribution as $k \rightarrow \infty$. The random variable $F_2$ satisfies 
\begin{equ}\label{oo}
p_i= \P(F_2 = i) = \frac{\int_\CX h_i\,d\lambda^0}{\int_\CX J \,d \lambda^0},~i=1,\ldots,n.
\end{equ}
\end{lemma}

Let us now show that Lemma~\ref{abstractlemma} is applicable to a certain Markov chain associated with $Z_t$. 
Our objective is to define this Markov chain and verify that it satisfies Properties (1)-(6). 
We start by explaining 
what it means for a process $Z_t$ to pass a saddle point.
Consider ``the projection on the separatrix'' mapping $\rho \colon \cV^{\delta,\e} \to \CL$
with
\begin{equ}
  \cV^{\delta,\e}=\{x \in \R^2 :|H(x)|\leq \delta\e^{\alpha/2}\}\;,
\end{equ}
and given by 
\begin{equ}
\rho(x) = \CL \cap \overline{\{\Phi^x_t:\, t \in \R\}}\;,\qquad \dot \Phi^x_t = \nabla 
H(\Phi^x_t)\;,\quad\Phi^x_0 = x.
\end{equ}
For  sufficiently  small $\delta$ and $\e$ this $\rho$ is uniquely defined because the closure of the 
orbit of the gradient flow $\Phi^x_t$ does indeed intersect  the separatrix $\CL$ at exactly one point.
We will say that $Z$ passes a saddle point $A_i$ if  its trajectory intersects the curves
\[
B(A_i)=\{x\in \cV^{\delta,\e},~\pi(\rho(x))= A_i\}\;,
\]
where $\pi:\mathbb{R}^2\to \torus$ is the quotient map from the plane to the torus.
Set
$\gamma_0^\eps = \beta_{0}^\eps =0$,  and then recursively
\begin{equation}\label{eqnAlphaBetaDef}
\gamma_n^\eps=\inf\Big\{t\geq\beta_{n-1}^\eps\,:\,Z_t \in \Big( \bigcup_{k\neq i}B(A_k) \Big) \bigcup \partial \cV^{\delta,\e}\Big\}\;,\quad
\beta_n^\eps=\inf\{t\geq\gamma_n^\eps: Z_t \in\CL\}\;,
\end{equation}
provided that $\pi(Z(\beta_{n-1}^\eps))\in \{A_i\} \cup \bigcup_{j} E(A_i\to A_j)$. 
Here, $E(A_i \to A_j)$ denotes the 
heteroclinic connection emanating from the saddle $A_i$ and ending at $A_j$, 
or the empty set if no such connection exists. It means the stopping time $\gamma_n^\eps$ clocks 
the first time after $\beta_{n-1}^\eps$ that the process either 
hits $\partial \cV^{\delta,\e}$, or goes past a saddle point different from the one behind $Z(\beta_{n-1}^\eps)$. Recall that we assumed that there are no 
homoclinic orbits, and therefore the definition of $\gamma_n^{\e}$ makes sense. 

Define an auxiliary metric space $\bar{\CM} = \CL \sqcup \partial \cV^{\delta,\e}$. Let us define a family of transition
functions $\bar{p}_\eps (x, dy)$ on $\bar{\CM}$. For $x \in \CL$, we
define $\bar{p}_\eps (x, dy)$ as the
distribution of $Z_{\bar\tau}$ (under~$\Pe^x$) with the random transition time $\bar\tau = \mu_1^{\eps} \wedge \beta_1^{\eps}$, where~$\beta_1^\epsilon$ 
and~$\mu_1^\epsilon$ are 
defined in~\eqref{eqnAlphaBetaDef} and~\eqref{eqnMuDef}, respectively.
In other words, it is the measure
induced by the process stopped when it either reaches the boundary of $\cV^{\delta,\e}$ or reaches the separatrix after passing by a saddle point.
For $x \in \partial \cV^{\delta,\e}$, we define $\bar{p}_\eps (x, dy)$ as the
distribution of $Z_{\bar\tau}$ (under~$\Pe^x$) with $\bar\tau = \beta_1^{\eps}$, i.e., the measure
induced by the process stopped when it first reaches the separatrix.
We write $\bar{R}^{x,\eps}_k$ for the corresponding
Markov chain starting at $x \in \bar{\CM}$.

Note that $\bar{\CM}$ depends on $\e$ since it contains $\partial \cV^{\delta,\e}$, and we would like to get rid of this dependence in order to use the abstract setup of Lemma~\ref{abstractlemma}. The projection on the separatrix  mapping $\rho$ (when lifted to the torus $\torus$) defines a natural
homeomorphism between  $\pi(\partial \cV^{\delta,\e}) \cap U_i$ and a circle $S_i \subset \pi(\CL) $ with circumference 
$\int_{\partial U_i} |\nabla H| \, dl$. We will assume that the circles $S_i$ and $S_j$ are disjoint for $i \neq j$, and denote this homeomorphism by $\rho^\e$.

While we introduced $\bar{\CM}$ as a subset of $ \mathbb{R}^2$, it is going to be more convenient to keep track of $\pi(\bar{R}^{x,\eps}_k)$ and the latest displacement separately. 
Let $N$ be a bounded measurable set in $ \mathbb{R}^2$ such that $ \mathbb{R}^2 = \bigsqcup_{z \in {\mathbb{Z}}^2} (N +z)$. For $x \in  \mathbb{R}^2$, let $[x] \in \mathbb{Z}^2$ 
be such that $x \in N + [x]$. We can choose $N$ in such a way that $[x]$ is constant on each connected component of $\pi^{-1}(U_i)$, i.e., $N$ consists of a finite number of 
cells and parts of their boundaries.
It is then natural to define a metric space $\CM$  independent of $\e$ by
\begin{equ}
\CM = (\pi (\CL) \sqcup S_1 \sqcup \dots \sqcup S_n ) \times \mathbb{Z}^2 =: \CM_1 \times \Z^2\;,
\end{equ}
which is indeed of the above type by setting $\CX = \pi (\CL) \times \mathbb{Z}^2$
and $C_i = S_i \times \mathbb{Z}^2$.
Let $\varphi: \bar{\CM} \to \CM$ 
be given by  $\varphi(x) = (\rho^{\e}(\pi(x)), [x])$ for $x \not\in \CL$, and  $\varphi(x) = (\pi(x), [x])$ for $x \in \CL$. We will write 
$\varphi_1: \bar{\CM} \to \CM_1$ and $\varphi_2:  \bar{\CM} \rightarrow   \mathbb{Z}^2$ for 
the first and second components of $\varphi$, respectively.
This allows us to define transition probabilities $p_\eps$ on $\CM$
by setting $p_\eps((x_0,k),\cdot)$ to be the law of $\varphi(Z_{\bar \tau})$ under
$\P_\eps^{\bar x_0}$ with $\bar x_0$ the only element of $N$ with $\pi(\bar x_0) = x_0$.
Similarly to before, we write $R_k^{x,\eps}$  for the 
Markov chain starting at $x \in \CM$ with transition probabilities $p_\eps$.

We finally define the function $g$ appearing in Property (6) by, for $x = (q,\xi) \in \CM$, 
setting $g((q,\xi)) = \xi \in \mathbb{Z}^2$. This continuous 
function measures the displacement during the last step if the chain is 
viewed as a process on $\mathbb{R}^2$, 
where only the integer parts of the initial and end points are counted.

The Markov chain is now defined, and it remains to verify that Properties (1)-(6) are
satisfied. Here we will adopt the approach~\cite{Koralov04}.
 The thrust of~\cite{Koralov04} is the asymptotic analysis of the behaviour of the process $Z_t$  in an $\e$-neighbourhood of the separatrix. A Markov chain on separatrices, similar
 to our $R^{x,\eps}_k$, also arises there.
In~\cite{Koralov04} the analysis, however, was done in a scaling slightly different from ours. More specifically,
in~\cite{Koralov04} the width of the separatrix region is restricted to be of order $\eps^{\alpha_1}$ with some $\alpha_1\in (1/4,1/2)$, while our width is $\delta\e^{\alpha/2}$, 
$\alpha >0$.
It is fairly straightforward to verify that
$\eps^{\alpha_1}$ may be replaced by $\delta\e^{\alpha/2}$, i.e., the imposed restriction $\alpha > 1/2$ is not needed and can be replaced by $\alpha > 0$.
We will, therefore, simply quote here the corresponding statements 
from~\cite{Koralov04} and explain their modifications necessary to imply our Properties (1)-(6).

Properties (1), (2) and (4). The existence of the limit  of the transition functions $p_{\eps}$ in the sense of Property (4) was justified in \cite[Lem.~3.1]{Koralov04}. This limit is denoted by $p_0$. An explicit formula for the density of $p_0$ was also provided in \cite[Eq.~(9)]{Koralov04}, which implies that Property (2) is satisfied. Observe that the probability of $\beta_1^{\eps}$ 
being less than $\mu_1^{\eps}$ tends to one as $\eps\downarrow 0$, uniformly in $x\in\CL$ by  \cite[Eq.~(26)]{Koralov04}. This implies Property (1).

Let us sketch the proof of the Doeblin condition (Property (5)). It suffices to show that there exists $E' \subset \CL_0$,  a connected segment of the same heteroclinic orbit of 
$H$, and there are a constant $c > 0$ and an integer $m > 0$ such that 
  \begin{equ}\label{doeblin}
\bar{p}_{\eps}^m(x, E') \geq c~|E'|~~~\text{for all}~~x \in  \fd \cap \bar{\CM} ,~\eps \in [0, \eps_0],
\end{equ}
where $|E'|$ is the arclength of $E'$. 
Since 
 $\bar{p}_\eps(x,   \fd \cap \CL) = 1$, if $x \in  \fd \cap \partial \cV^{\delta,\e}$, it suffices to show 
 the last estimate immediately holds for all $x \in \fd \cap \CL$.
 For  $x \in \fd \cap \CL$ we can obtain~\eqref{doeblin} once we show that 
 there is a set $J \subset \R^2$, that may depend on $\e$, such that it has the following two properties.
 Firstly, 
 there  is $m > 0$ such that
\[
\P^x_\eps \left( Z_t \in J~~\text{for some}~\gamma^{\eps}_{m} < t < \beta^{\eps}_{m}\right) > c > 0,  \forall x \in  \fd \cap \CL\;. 
\]
Secondly, 
\[
\P^x_\eps \left( Z_t \in E' ~~\text{for some}~t>0 \right) > c > 0,  \forall x \in  J. 
\]

We construct $J$ as follows. Suppose $a_2$ and $a_3$ are the endpoints of $E'$, and $a_1$ lies on the same heteroclinic orbit of $H$ so that
the points are ordered in the direction of the flow $v$,  as depicted on Figure~\ref{fig:J}.  
 Let $J$ be a piece of the curve in $\cV^{\delta,\e}$, 
 that is mapped by $\rho$ to $a_1$:
$J = \{ x \in \cV^{\delta,\e} \bigcap \pi^{-1}(U_k)  : \sqrt{\e} \leq |H(x)| \leq 2 \sqrt{\e}, \rho(x)=a_1 \}$ for some $k$.
  
 Roughly speaking, the first property  means that the process has a positive chance of going to a particular curve at a distance $\sqrt{\eps}$ from the separatrix, 
 transversal to the flow lines, prior to passing by $m$ saddle points. This is not surprising since the ratio of the parallel advection to the diffusion is of order $1/\eps$. 
 The proof follows along the same lines as the proof of  \cite[Lem.~3.1]{Koralov04}. 
 Similarly, the second property is true, because if the process $Z$ starts on $J$, then it will take $\CO(\e)$ time for the flow $v$ to carry $Z$ past the segment $E'$, but this is sufficient
 for diffusion move the distance $\CO(\sqrt{\e})$, and to reach the separatrix. The latter argument is the same as the derivation of (63) in \cite{DolgopyatKoralov08}. Thus
  the Doeblin condition for $R^{x, \eps}_k$ holds.

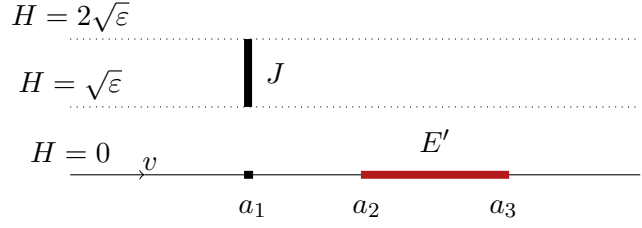
\begin{figure}
    \centering
\begin{tikzpicture}[scale=1.5]

\node at (1.6,-0.3) {$a_1$};
\node at (2.6,-0.3) {$a_2$};
\node at (3.8,-0.3) {$a_3$};

\draw[-](0.65,0) -- (5,0);
\draw[line width=3pt,line cap=rect,draw=black] (1.56,0) -- (1.56,0);
\draw[line width=3pt,draw=darkred] (2.55,0) -- (3.85,0);
\draw[line width=3pt,draw=black] (1.56,0.6) -- (1.56,1.2);

\node at (1.8,0.9) {$J$};
\node at (3.2,0.3) {$E'$};

\draw[->] (0,0) -- (0.65,0);
\node at (0.7,0.1) {$v$};
\draw[dotted] (0,0.6) -- (5,0.6);
\draw[dotted] (0,1.2) -- (5,1.2);
\node at (0,0.2) {$H= 0$};
\node at (0,0.8) {$H= \sqrt{\e}$};
\node at (0,1.4) {$H=2 \sqrt{\e}$};
\end{tikzpicture}
\caption{Construction of $J$.}
\label{fig:J}
\end{figure}

With our definition of $g$,
\[
\int_\CM g(x)\, d\lambda^{\eps}(x) \left(\int_\CM \E\bar\tau\, d\lambda^\eps(x) \right)^{-1} = \lim_{t \rightarrow \infty} (\E Z_t /t),
\]
where $ \bar\tau= \mu_1^{\eps} \wedge \beta_1^{\eps}$ is the random transition time for our Markov chain, and the right hand side is the effective drift for the original process starting from an arbitrary point $x$. Note that for some $C(\e)>0$,
\begin{equation*}
  \lim_{t \rightarrow \infty} \frac{\E Z_t}{t}
    = C(\e)\int_{\torus} v(x) \, dx
    = 0,
\end{equation*}
which implies Property~(6).

Small escape probability from $\CX$ (Property~(3)) follows from \cite[Lem.~4.1 \& 4.3]{Koralov04}. Indeed, the first lemma describes the
asymptotics of the distribution of $H(Z_{\gamma_{1}^{\eps}})$, while the second
describes the probability of the process starting at $x$ to exit the boundary layer before reaching the separatrix,
assuming that $H(x)$ is fixed. The two lemmas, combined with the Markov property of the process, imply
Property~(3).  Moreover, our functions $h_i(x) = h_i^\delta(x)$ depend on $\delta$ and can be identified as
\begin{equation}\label{hi}
h_i^\delta(x) =  \lim_{\eps \rightarrow 0} \eps^{-\frac{1-\alpha}{2}} \P\Big(\hbox{ the process starting at}~Z_{\gamma_1^{\eps}}~\text{reaches}~\partial \cV^{\delta,\e} 
\cap \pi^{-1}\left(U_i\right)~\hbox{ before reaching}~\cL\Big).
\end{equation}

\begin{proof}[Proof of Proposition~\ref{ppnCLTFirstHitShort}]
 From \cite[Lemma 4.1 and Lemma 4.3]{Koralov04}, it follows that $h_i^\delta(x)$ in~\eqref{hi} satisfy 
\[
\int_X h^\delta_i(x) \,d \lambda^0(x) = \delta^{-1}\bar{p}_i,~~i =1,\ldots,n,
\]
for some $\bar{p}_i > 0$. Now Lemma~\ref{abstractlemma} implies that covariance matrix $Q$ and probabilities $p_i$ in~\eqref{oo}  must satisfy
\[
Q = \bar{Q} /(\bar{p}_1 + \cdots + \bar{p}_n)\;,\qquad
p_i = \bar{p}_i/(\bar{p}_1 + \cdots + \bar{p}_n).
\]
The non-degeneracy of $Q$ was demonstrated in \cite{HairerKoralovPajorGyulai2014} in the case when $\alpha = 0$. The proof does not require modifications in the current case.

It is clear that Lemma \ref{abstractlemma} with $g$ and $C_i$ introduced above implies the first statement of Proposition~\ref{ppnCLTFirstHitShort}. The second statement of Proposition~\ref{ppnCLTFirstHitShort} requires a slight strengthening of the abstract Lemma \ref{abstractlemma} without major modifications in the proof (see \cite{HairerKoralovPajorGyulai2014}).
\end{proof}

\section{A formal asymptotic expansion}\label{sxnAsymptoticExpansion}
  We devote this appendix to a heuristic derivation of the system~\eqref{eqnThetaSys} through a formal asymptotic expansion.
  Let~$\delta$ be a small parameter such that
  \begin{equation*}
    \delta \xrightarrow{\epsilon \to 0} 0,
    \qquad\text{and}\qquad
    \frac{\epsilon}{\delta} \xrightarrow{\epsilon \to 0} 0.
  \end{equation*}
  The typical situation arising throughout this paper is for~$\delta = \epsilon^{1 - \alpha}$.
  This specific choice, however, is unnecessary for the formal asymptotics presented here. 

  Define~$\varphi = \varphi^{\epsilon,\delta}$ by
  \begin{equation*}
    \varphi(z, t) = \tilde \theta\paren[\Big]{ z, \frac{t \log\paren{\delta/\epsilon}}{\delta}}
  \end{equation*}
  and observe
  \begin{equation*}
    \frac{1}{\log\paren{\delta/\epsilon}} \partial_t \varphi
      = \frac{1}{\delta} v \cdot \grad \varphi + \frac{\epsilon}{2\delta} \lap \varphi.
  \end{equation*}
  For notational convenience, we suppress the $\epsilon$ and $\delta$-dependence of $\varphi$ and other quantities below.
  We perform an asymptotic expansion for~$\varphi$ by writing
  \begin{equation*}
    \varphi(z, t) = \varphi_0( x, z, t ) + \gamma \varphi_1( x, z, t ) + \CO(\gamma^2).
  \end{equation*}
  where $z$ is the ``fast variable'', $\gamma = \delta^{1/4}$ and $x = \gamma z$.
  Using
  \begin{equation*}
    \grad = \gamma \grad_x + \grad_z
    \qquad\text{and}\qquad
    \lap = \gamma^2 \lap_x + 2\gamma \grad_x \cdot \grad_z + \lap_z,
  \end{equation*}
  we compute
  \begin{multline}\label{eqnVarphi}
    \frac{1}{\log\paren{\delta/\epsilon}} \partial_t \varphi
      = \frac{\gamma}{\delta} v(z) \cdot \grad_x \varphi_0
	+ \frac{1}{\delta} v(z) \cdot \grad_z \varphi_0
	+ \frac{\epsilon}{2\delta} \lap \varphi_0
      \\
	+ \gamma\paren[\Big]{
	  \frac{\gamma}{\delta} v(z) \cdot \grad_x \varphi_1
	  + \frac{1}{\delta} v(z) \cdot \grad_z \varphi_1
	  + \frac{\gamma^2 \epsilon}{2\delta} \lap_x \varphi_1
	  + \frac{\gamma \epsilon}{\delta} \grad_x \cdot \grad_z \varphi_1
	  + \frac{\epsilon}{2\delta} \lap_z \varphi_1
	}
	+ \CO(\gamma^2).
  \end{multline}

  Since there is only one term of order $\CO(1/\delta)$, it must vanish.
  This gives
  \begin{equation}\label{eqnVarphi0}
    v(z) \cdot \grad_z \varphi_0 = 0,
  \end{equation}
  and hence, the dependence of~$\varphi_0$ on $z$ is only through~$\Gammaeps(z)$, the projection of~$z$ onto~$\cG$.%
  \footnote{%
    Here, analogous to~\eqref{eqnGammaepsDef}, $\Gammaeps(z) = (i, (\delta/\epsilon)^{1/2} \abs{H(z)}) \in \cG$ when $z \in U_i$.
  }
  In view of our rescaling, we expect
  \begin{equation*}
    \theta(x, y, t) = \varphi_0(x, z, t) + o(1),
    \qquad
    \text{where } y \in \cG \text{ and } y = \Gammaeps(z).
  \end{equation*}
  Consequently, in order to justify~\eqref{eqnThetaSys}, we will formally obtain equations for the function~$\theta$ defined by
  \begin{equation*}
    \theta(x, y, t) \defeq \lim_{\epsilon \to 0} \varphi_0(x, z, t),
    \qquad
    \text{where } y \in \cG \text{ and } y = \Gammaeps(z).
  \end{equation*}

  We begin by balancing the $\CO(\gamma/\delta)$ terms by choosing $\varphi_1$ to be the solution of
  \begin{equation}\label{eqnVarphi1}
    \frac{\epsilon}{2} \lap_z \varphi_1 +  v(z) \cdot \grad_z \varphi_1
      = - v(z) \cdot \grad_x \varphi_0,
  \end{equation}
  with periodic boundary conditions in~$z$.
  In order to do this we would need to verify the compatibility condition
  \begin{equation*}
    \int_{U} v(z) \cdot \grad_x \varphi_0 \, dz = 0.
  \end{equation*}
  Indeed, let $U$ be any connected component of the support of $\varphi_0$.
  Then, for fixed~$x$ and~$t$, equation~\eqref{eqnVarphi0} implies that for all $z \in U$, we have $\grad_x \varphi_0 (x, z, t) = F\circ H(z)$ for some function~$F$.
  Consequently,
  \begin{equation}\label{eqnVint}
    \int_{U} v(z) \cdot \grad_x \varphi_0 \, dz
      = \int_{U} v(z) \cdot F(H(z)) \, dz
      = \int_U \curlz (F \circ H) \, dz
      = \int_{\partial U} F \circ H \cdot dl
      = 0,
  \end{equation}
  since $F(H)$ is constant on $\partial U$.
  This ensures that the compatibility condition for~\eqref{eqnVarphi1} is satisfied.

  In order to express~$\varphi_1$ more conveniently, define the corrector $\chi = (\chi_1, \chi_2)$ to be the solution of the normalised cell problem
  \begin{equation}\label{eqnCorrector2}
    \frac{\epsilon}{2} \lap_z \chi_i +  v(z) \cdot \grad_z \chi_i
      = - v_i(z) \paren[\Big]{
	  \frac{\partial_{x_i} \varphi_0}{\av{\partial_{x_i} \varphi_0}}
	},
  \end{equation}
  with periodic in~$z$ boundary conditions.
  Here $\av{f}$ denotes the average of~$f$ with respect to the fast variable~$z$.
  We remark that the ``standard corrector'', denoted by $\bar \chi = (\bar \chi_1, \bar \chi_2)$, is usually chosen to be the solution of the cell problem
  \begin{equation}\label{eqnChiBar}
    \frac{\epsilon}{2} \lap_z \bar\chi_i +  v(z) \cdot \grad_z \bar\chi_i
      = - v_i(z),
  \end{equation}
  Our corrector~$\chi$ has an extra term depending on the fast variable, however this dependence is only through the projection onto the Reeb graph.

  With this notation, observe
  \begin{equation*}
    \varphi_1 = \chi \cdot \grad_x \av{\varphi_0}.
  \end{equation*}
  Balancing terms in~$\varphi$ yields
  \begin{align}
    \nonumber
    \frac{1}{\log\paren{\delta/\epsilon}} \partial_t \varphi_0
      &= \frac{\epsilon}{2\delta} \lap \varphi_0
	+ \frac{\gamma^2}{\delta} v(z) \cdot \grad_x \varphi_1
	+ o(1)
      \\
      \label{eqnVarphi0evol}
      &= \frac{\epsilon}{2\delta} \lap \varphi_0
	+ \frac{\gamma^2}{\delta} v(z) \cdot \grad_x \paren[\Big]{ \chi \cdot \grad_x \av{\varphi_0}}
	+ o(1).
  \end{align}
  We will deduce~\eqref{eqnThetaY} and~\eqref{eqnThetaX} from this by multiplying by an appropriate test function and integrating. 

  To obtain~\eqref{eqnThetaY}, let $\phi = \phi(z)$ be a test function that is compactly supported in $\R^2 - \mathcal L$, and only depends on $z$ through the projection $\Gammaeps(z)$ (i.e., $v(z) \cdot \grad_z \phi = 0$).
  Multiplying~\eqref{eqnVarphi0evol} by~$\phi$ and integrating gives
  \begin{equation}\label{eqnVChi2}
    \int_{z \in \torus} \paren[\Big]{
      \frac{\partial_t \varphi_0}{\log\paren{\delta/\epsilon}} - \frac{\epsilon}{2\delta} \lap \varphi_0 } \phi \, dz
      = \frac{\gamma^2}{\delta} \int_{z \in \torus}
	  v(z) \cdot \grad_x \paren[\Big]{ \chi \cdot \grad_x \av{\varphi_0}}
	  \, \phi(x, z) \, dz.
  \end{equation}
  Note that away from the separatrix the corrector $\chi$ has oscillations of order $\epsilon$ on connected components of level sets of~$H$ (see for instance~\cite{Childress79,NovikovPapanicolaouEtAl05}).
  Hence
  \begin{equation*}
    v(z) \cdot \grad_z \chi = \CO(\epsilon),
    \qquad\text{and}\qquad
    v(z) \cdot \grad_z \phi = 0.
  \end{equation*}
  Consequently (following~\eqref{eqnVint}) the integral on the right of~\eqref{eqnVChi2} vanishes as $\epsilon \to 0$.
  Restricting our attention to $U_i$, writing~\eqref{eqnVChi2} in terms of $y$ and using the co-area formula gives
  \begin{align*}
    0 &= \int_{U_i} \paren[\Big]{
      \frac{\partial_t \varphi_0}{\log\paren{\delta/\epsilon}} - \frac{\epsilon}{2\delta} \lap \varphi_0 } \phi \, dz + o(1)
      = \int_{U_i} \paren[\Big]{
	\frac{\partial_t \varphi_0}{\log\paren{\delta/\epsilon}} - \frac{\epsilon}{2\delta} \lap_z \varphi_0 } \phi \, dz
	+ o(1)
      \\
      &= \int_{y_i = 0}^{\paren{\frac{\delta}{\epsilon}}^{1/2}\max\limits_{U_i} \abs{H}}
	\int_{U_i \cap \set{\abs{H} = \paren{\frac{\epsilon}{\delta}}^{1/2}y_i}}
	  \begin{multlined}[t]
	    \Bigl(
	      \frac{\partial_t \varphi_0}{\log\paren{\delta/\epsilon}
		\abs{\grad H}}
	      - \frac{\abs{\grad H}}{2} \partial_{y_i}^2 \varphi_0
	      \\
	      \qquad- \paren[\Big]{\frac{\epsilon}{\delta}}^{1/2}
		\frac{\lap H}{2\abs{\grad H}}
		\partial_{y_i} \varphi_0
	    \Bigr)
	    \phi \, dl \, d y_i
	    + o(1)
	  \end{multlined}
      \\
      &\xrightarrow{\epsilon \to 0}
	\int_{y \in I_i} \paren[\Big]{
	  \frac{q_i}{a_i} \partial_t \theta
	  - \frac{q_i}{2} D_i^2 \theta
	} \phi(y) \, dy.
  \end{align*}
  This is exactly the weak form of~\eqref{eqnThetaY}.

  For the gluing condition~\eqref{eqnThetaX} define the boundary layer $\Vep$ to be a small neighbourhood of~$\mathcal L$ where the effects of the diffusion and convection balance.
  Explicitly, set
  \begin{equation*}
    \Vep \defeq \set[\Big]{ z \in \torus : \abs{H(z)} <
	\sqrt{N \epsilon},
      },
  \end{equation*}
  where $N = N(\epsilon) \to \infty$ arbitrarily slowly as $\epsilon \to 0$.
  Note that the volume of the boundary layer satisfies
  $\abs{\Vep} = \CO( (N \epsilon)^{1/2} \abs{\log \epsilon} )$ asymptotically as $\epsilon \to 0$.
  Multiplying~\eqref{eqnVarphi0evol} by $(\delta/N \epsilon)^{1/2}$ and integrating over~$\Vep$ gives
  \begin{multline*}
    o(1) = \paren[\Big]{\frac{\delta}{N \epsilon}}^{1/2}
      \int_{\Vep} \paren[\Big]{
      \frac{\partial_t \varphi_0}{\log\paren{\delta/\epsilon}}
	- \frac{\epsilon}{2\delta} \lap \varphi_0
	- \frac{\gamma^2}{\delta} v(z) \cdot \grad_x \paren[\Big]{ \chi \cdot \grad_x \av{\varphi_0}}
    }
    \, dz
    \\
      = -\frac{1}{2}\paren[\Big]{\frac{\epsilon}{N \delta}}^{1/2}
	  \int_{\partial \Vep} \grad_z \varphi_0 \cdot \hat n \, dl
	- \frac{\gamma^2}{\sqrt{N \epsilon \delta}} \sum_{i,j = 1}^2
	    \paren[\Bigg]{\int_{\Vep} v_i \chi_j  \, dz}
	    \partial_{x_i} \partial_{x_j} \av{\varphi_0}
	+o(1).
  \end{multline*}
  This forces
  \begin{equation}\label{eqnGluing}
    \paren[\Big]{\frac{\epsilon}{\delta}}^{1/2}
      \int_{\partial \Vep} \grad_z \varphi_0 \cdot \hat n \, dl
    + \frac{\gamma^2}{ \sqrt{\delta}}
      \sum_{i,j = 1}^2
	\paren[\Bigg]{
	  \frac{2}{\sqrt{\epsilon}} \int_{\Vep} v_i \chi_j  \, dz
	}
	\partial_{x_i} \partial_{x_j} \av{\varphi_0}
    = o(1).
  \end{equation}

  Clearly
  \begin{equation*}
    \paren[\Big]{\frac{\epsilon}{\delta}}^{1/2}
      \int_{\partial \Vep} \grad_z \varphi_0 \cdot \hat n \, dl
    = \paren[\Big]{\frac{\epsilon}{\delta}}^{1/2}
      \sum_{i = 1}^M \int_{\partial \Vep \cap U_i}
	\grad_z \varphi_0 \cdot \hat n \, dl
    \xrightarrow{\epsilon \to 0}
      \sum_{i = 1}^M q_i D_i \theta.
  \end{equation*}
  For the second term in~\eqref{eqnGluing} define the $2 \times 2$ matrix $Q$ by
  \begin{equation}\label{eqnQijDef}
    Q_{i,j}
      = \paren[\Big]{ \sum_{k=1}^M q_k }^{-1}
	\lim_{\epsilon \to 0} \frac{1}{\sqrt{\epsilon}}
	  \int_{\Vep} \paren[\Big]{ v_i \chi_j + v_j \chi_i } \, dz.
  \end{equation}
  If we replace $\chi$ with $\bar \chi$, then the limit on the right is well studied.
  Indeed, using~\eqref{eqnChiBar} we see
  \begin{equation*}
    \frac{1}{\sqrt{\epsilon}}
      \int_{\Vep} \paren[\Big]{ v_i \bar \chi_j + v_j \bar \chi_i } \, dz
    = \sqrt{\epsilon} \int_{\Vep} \grad \bar \chi_i \cdot \grad \bar \chi_j \, dz
      - \frac{\sqrt{\epsilon}}{2} \int_{\partial \Vep}
	  \paren[\Big]{
	    \bar \chi_i \grad \bar \chi_j + \bar \chi_j \grad \bar \chi_i
	  } \cdot \hat n \, dl.
  \end{equation*}
  It is well known  that $\grad \bar \chi$ is $\CO(\epsilon^{-1/2})$ in an $\sqrt{\epsilon}$ neighbourhood of the boundary, and $\CO(1)$ elsewhere.
  Consequently, we expect the right hand side of the above to converge as $\epsilon \to 0$ (see for instance~\cite{FannjiangPapanicolaou94}).

  In our situation, the extra term $\partial_{x_i} \varphi_0 / \av{\partial_{x_i}\varphi_0}$ only depends on the fast variable $z$ through its projection $\Gammaeps(z)$.
  Hence the asymptotic behaviour of $\chi$ is similar to that of $\bar \chi$, and consequently the limit in~\eqref{eqnQijDef} should exists.
  Thus,  equation~\eqref{eqnGluing} yields
  \begin{equation*}
    \sum_{k = 1}^M q_k D^y_k \theta
      + \paren[\Big]{ \sum_{k = 1}^M q_k } [Q : \grad_x^2] \varphi_0
    = 0,
  \end{equation*}
  since $\gamma^2 = \delta$.
  This is exactly~\eqref{eqnThetaX}, as desired.

  \begin{remark}
    Choosing $\delta = 2\epsilon$, the above method also provides an effective equation for $\tilde \theta$ on time scales of order $1/\epsilon$, as considered in~\cite{HairerKoralovPajorGyulai2014}.
    In this regime the analogue of equation~\eqref{eqnThetaY} is now on a finite graph (the unscaled Reeb graph of~$H$), and the generator $\mathcal A$ has non-constant (singular) coefficients.
  \end{remark}

\endappendix
\bibliographystyle{Martin}
\bibliography{citations}

\end{document}